\RequirePackage{fix-cm}
\documentclass[smallextended]{svjour3}       
\smartqed  
\usepackage{csquotes}
\usepackage{babel}
\usepackage{amsmath,amssymb}
\usepackage[utf8]{inputenc}
\usepackage{lscape}
 \def\bibhang{24pt}%
\usepackage{algorithm}
\usepackage{algorithmic}
\usepackage{comment}
\usepackage{booktabs}
\usepackage{multirow}
\usepackage{subfigure}
\usepackage[colorlinks=true,breaklinks=true,bookmarks=true,urlcolor=blue,
     citecolor=blue,linkcolor=blue,bookmarksopen=false,draft=false]{hyperref}
\usepackage{graphicx}
\def\EMAIL#1{\href{mailto:#1}{#1}}

\DeclareMathAlphabet{\mathpzc}{OT1}{pzc}{m}{it}

\numberwithin{equation}{section}
\numberwithin{proposition}{section}
\numberwithin{theorem}{section}
\numberwithin{remark}{section}
\numberwithin{definition}{section} 
\numberwithin{example}{section}
\usepackage{hyperref}
\usepackage{geometry}
\begin{document}

\title{Relaxation methods for pessimistic bilevel optimization}

\titlerunning{Relaxation methods for pessimistic bilevel optimization}        
\author{\small Imane Benchouk \and Lateef Jolaoso \and Khadra Nachi \and Alain Zemkoho}


\institute{Imane Benchouk \at
              Laboratory of Mathematical Analysis and Applications, University Oran 1, Algeria,             \EMAIL{benchouk.imane@univ-oran1.dz}           
           \and
          Lateef Jolaoso \at
           School of Mathematical Sciences,	University of Southampton, United Kingdom, 
           \EMAIL{l.o.jolaoso@soton.ac.uk}
           \and 
          Khadra Nachi \at
              Laboratory of Mathematical Analysis and Applications, University Oran 1, Algeria, 
              \EMAIL{nachi.khadra@univ-oran1.dz}\\
        (Part of the work was completed while this author was visiting the School of Mathematical Sciences at Southampton)
              \and 
         Alain Zemkoho \at
           School of Mathematical Sciences,	University of Southampton, United Kingdom,
           \EMAIL{a.b.zemkoho@soton.ac.uk}.
}

\date{Received: date / Accepted: date}

\maketitle

\begin{abstract}
We consider a smooth pessimistic bilevel optimization problem, where the lower-level problem is convex and satisfies the Slater constraint qualification. These assumptions ensure that the Karush-Kuhn-Tucker (KKT) reformulation of our problem is well-defined. We then introduce and study the (i) Scholtes, (ii) Lin and Fukushima, (iii) Kadrani, Dussault and Benchakroun, (iv) Steffensen and Ulbrich, and (v) Kanzow and Schwartz relaxation methods for the KKT reformulation of our pessimistic bilevel program. These relaxations have been extensively studied and compared for mathematical programs with complementatrity constraints (MPCCs). To the best of our knowledge, such a study has not been conducted for the pessimistic bilevel optimization problem, which is completely different from an MPCC, as the complemetatrity conditions are part of the objective function, and not in the feasible set of the problem. After introducing these relaxations, we provide convergence results for global and local optimal solutions, as well as suitable versions of the C- and M-stationarity points of our pessimistic bilevel optimization problem. Numerical results are also provided to illustrate the practical implementation of these relaxation algorithms, as well as some preliminary comparisons. 
${}$\\[2ex]
\\[2ex]
\end{abstract}

\section{Introduction}\label{Introduction}
Our focus in this paper is on the pessimistic bilevel optimization problem defined by 
\begin{equation}\tag{P$_p$}\label{PBP}
    \underset{x\in X}\min~\varphi_{p}(x):=\underset{y\in S(x)}\max~F(x,y),
\end{equation}
where $X:=\left\{x\in \mathbb{R}^n\left|\; G(x)\leq 0\right.\right\}$ (resp. $F\, : \mathbb{R}^n\times \mathbb{R}^m \rightarrow \mathbb{R}$) corresponds to the upper-level feasible set (resp. objective function). The upper-level constraint function $G$ is defined from $\mathbb{R}^n$ to $\mathbb{R}^p$. As for  $S: \mathbb{R}^n \rightrightarrows \mathbb{R}^m$, it corresponds to the lower-level optimal solution set-valued mapping; i.e., for all $x\in X$, 
 \begin{equation}\label{S(x)}
    S(x):=\underset{y\in Y(x)}{\arg\min}~f(x,y). 
 \end{equation}
 Note that in \eqref{S(x)}, the function $f\, : \mathbb{R}^n\times \mathbb{R}^m \rightarrow \mathbb{R}$ represents the lower-level objective function, while the set-valued mapping $Y: \mathbb{R}^n \rightrightarrows \mathbb{R}^m$ describes the lower-level feasible set that is represented by 
 \begin{equation}\label{Y(x)}
    Y(x):=\left\{y\in \mathbb{R}^m|\; g(x,y)\leq 0\right\}
 \end{equation}
with $g\, : \mathbb{R}^n\times \mathbb{R}^m \rightarrow \mathbb{R}^q$. Throughout the paper, we assume that these assumptions are valid:
\begin{itemize}
    \item $\forall x\in X:\; S(x) \neq \emptyset$;
    \item $\forall x\in X$:\; the functions $f(x, .)$ and $g_i(x, .)$, for all $i=1, \ldots, q$, are convex;
    \item $\forall x\in X,\;\; \exists y(x)\in\mathbb{R}^{m}: \;g_{i}(x,y(x))<0$ for all $i=1, \ldots, q$;
    \item the functions $F$ and $G$ (resp. $f$ and $g$) are twice (resp. thrice) continuously differentiable.
\end{itemize}
Note that the second assumption means that the lower-level problem is convex in the classical sense, while the third one is the Slater constraint qualification for the lower-level problem. 

Under these assumptions, we can write the Karush-Kuhn-Tucker (KKT) reformulation
\begin{equation}\tag{KKT}\label{MPCC}
    \underset{x\in X}\min~\psi_{p}(x):=\underset{(y, u)\in \mathcal{D}(x)}\max~F(x,y)
\end{equation}
of problem \eqref{PBP}, where the set-valued mapping $\mathcal{D} :\mathbb{R}^{n} \rightrightarrows \mathbb{R}^{m}\times\mathbb{R}^{q}$ is given by
\begin{equation}\label{KKT system}
\mathcal{D}(x) := \left\{\left. (y, u)\in \mathbb{R}^{m+q} \right|\;\; \mathcal{L}(x,y,u)=0, \;\; u\geq 0, \;\, g(x,y)\leq0, \;\, u^\top g(x, y)=0\right\}
\end{equation}
with $\mathcal{L}(x,y,u):=\nabla_{y}L(x,y,u)$, $L:(x,y,u)\rightarrow L(x,y,u):=f(x,y)+u^{T}g(x,y)$ being the Lagrangian function associated to the parametric optimization problem describing the lower-level optimal solution set-valued mapping  \eqref{S(x)}.
Problem \eqref{MPCC}, which will be the main focus of our attention in this paper, is a special class of the minmax optimization problem with coupled inner complementarity or equilibrium constraints, and corresponds to the pessmistic counterpart of the KKT reformulation that is commonly used in optimistic bilevel optimization problem to get a single-level problem; see, e.g., \cite{Au,DempeDutta2012,DempeZemkohoKKTRefNonsmooth,DempeZemkohoKKTRef}.  Hence, in the sequel, we will refer to \eqref{MPCC} as the KKT reformulation of the pessimistic bilevel optimization problem \eqref{PBP} and use the notation
\[
\mathcal{S}_{p}(x):=\left\{  (y,u)\in\mathcal{D}(x)\left\vert \;F(x,y)\geq
\psi_{p}(x)\right.  \right\} \]
to describe the optimal solution set-valued mapping of the parametric optimization problem associated to the optimal value function $\psi_p$ involved in the definition of problem  \eqref{MPCC}.

In the paper \cite{BNZ}, the authors introduced the Scholtes relaxation method for the pessimistic bilevel optimization problem \eqref{PBP}, including the theoretical convergence results as well as numerical implementations that show the potential of the method. Considering the wide range of relaxation techniques in the mathematical program with complementarity constraints (MPCC) literature, the main goal of this new work is to extend the study from \cite{BNZ} to the remaining state-of-the-art MPCC-type relaxation algorithms in the context of problem  \eqref{MPCC}. 

Precisely, based on problem \eqref{MPCC}, we introduce, study, and compare the (i) Scholtes, (ii) Lin and Fukushima, (iii) Kadrani, Dussault and Benchakroun, (iv) Steffensen and Ulbrich, and (v) Kanzow and Schwartz relaxations to our pessimistic bilevel optimization problem \eqref{PBP} in the spirit of the analysis conducted in \cite{HoheiselEtAlComparison2013}. A key difference between our work and what is done in the latter reference is that we are minimizing the optimal value function $\psi_p$, where the complementarity conditions only govern the inner feasible set, and therefore the challenge in our problem is in the objective function, unlike in \cite{HoheiselEtAlComparison2013}, where the main difficulty is in the constraint set. 
Because of this specific nature of problem \eqref{MPCC}, we pay special attention to the computation of its global and local optimal solutions, as well as stationary points, based on the different relaxations. Building on \cite{BNZ} and \cite{HoheiselEtAlComparison2013}, we construct results providing conditions ensuring that sequences of points iteratively computed via the aforementioned relaxations converge to  global optimal, local optimal, or stationary points of problem \eqref{MPCC}. 

To proceed, in the next section, we first provide some basic mathematical tools as well as constraint qualifications and stationarity concepts associated to problem \eqref{MPCC}. Subsequently, in Section \ref{sec:relaxation schemes}, we introduce a general relaxation framework that captures all the schemes considered in this paper and study some of the basic properties of the corresponding set-valued mapping \eqref{KKT system}, as well as their impact on the optimal value function $\psi_p$, which is the objective function of problem \eqref{MPCC}. 
The main theoretical contributions of the paper are given in Section \ref{sec:method and convergence}, where the generic relaxation algorithm for problem \eqref{MPCC} is introduced and conditions ensuring its convergence to global and local optimal solutions, as well as to stationary points, is established for the different schemes (namely, the (i) Scholtes, Lin and Fukushima, (ii) Kadrani, Dussault and Benchakroun, (iv) Steffensen and Ulbrich, and (v) Kanzow and Schwartz relaxations). 

In Section \ref{Numerical experiments}, we present some numerical experiments to illustrate the behavior of each of the relaxations discussed in the previous sections. Considering the challenge in finding an off-the-shelf solution algorithm for minmax optimization with coupled inner constraints that can tractably compute global and local optimal solutions of the relaxation problems, we focus our attention on computing C-stationarity points for probem \eqref{MPCC} via the relaxations introduced in Section \ref{sec:relaxation schemes}. The numerical comparisons focus on suitable test examples that are collected from the bilevel programming literature  \cite{Mitsos2006,WiesemannTsoukalasKleniatiRustem2013}. Our algorithms generate solutions similar to the ones in the literature. Furthermore, we assess the feasibility and C-stationarity, which is primarily satisfied by the solutions obtained via the Scholtes relaxation. This, coupled with other comparison factors such as the number of iterations employed by the algorithm and the execution time, indicates that the Scholtes relaxation method exhibits greater robustness and reliability compared to other methods, thus corroborating the conclusions drawn in the context   {of} MPCC problems  \cite{HoheiselEtAlComparison2013}.


To keep the presentation simple and easily accessible, we try to keep the presentation of Sections 2--5 as much compact as possible and move most of the technical proofs to the appendices.

\section{Preliminaries}\label{sec:preliminaries}
We start this section with some basic  variational analysis properties; see, e.g., \cite{BS2,RTW} for more details on these concepts.
First, given two subsets $A$ and $B$ of $\mathbb{R}^{n}$, their Hausdorff
distance $d_{H}(A,B)$ is defined by
\[
d_{H}(A,B):=\max~\{e(A,B), \;\, e(B,A)\},
\]
where the excess $e(A,B)$ of $A$ over $B$ is obtained from the formula%
\[
e(A,B):=\underset{x\in A}{\sup}~d(x,B)
\]
with $d(x,B):=\underset{y\in B}{\inf}~d(x,y)$ denoting the distance from the point $x$ to the set $B$, while considering the usual conventions
 $e(\emptyset,B)=0$ and $e(A,\emptyset)=+\infty$ if
$A\neq\emptyset$. Furthermore, let
\[
\mathrm{dom}\,\Psi:=\left\{x\in\mathbb{R}^{n}|\;\;\Psi(x)\neq\emptyset \right\} \;\, \mbox{ and }\;\, \mbox{gph}\,\Psi :=\{(x,y)\in\mathbb{R}^{n+m}|\;\;y\in\Psi(x)\}
\]
denote the domain  and graph of the set-valued mapping $\Psi :\mathbb{R}^{n}\rightrightarrows\mathbb{R}^{m}$. 

\begin{definition}
The set-valued mapping $\Psi:\mathbb{R}^{n}\rightrightarrows\mathbb{R}^{m}$
is said to be
\begin{itemize}
  \item[(i)] inner semicompact at $\bar{x}\in\mathrm{dom}\Psi$ if for any
sequence $(x_{k})_{k}\rightarrow\bar{x}$ there exists a sequence $(y_{k})$
with $y_{k}\in\Psi(x_{k})$ such that $(y_{k})$ admits a convergent
subsequence;
\item[(ii)]  inner semicontinuous at $(\bar{x},\bar
{y})\in\mbox{gph }\Psi$ if for any sequence $(x_{k})_{k}\rightarrow\bar{x}$
there exists a sequence $(y_{k})$ such that $y_{k}\in\Psi(x_{k})$ and
$y_{k}\rightarrow\bar{y}$ or equivalently, if $d(\bar{y},\Psi(x))\rightarrow0$
whenever $x\rightarrow\bar{x}$;
\item[(iii)] lower semicontinuous at $\bar{x}
\in\mathrm{dom}\Psi$ (in Hausdorff sense) if, for every $\varepsilon
>0$, there exists a neighborhood $U$ of the point $\bar{x}$ such that for any $x\in U$,
\[
e(\Psi(\bar{x}),\Psi(x))<\varepsilon \;\, \mbox{ or equivalently } \;\, \Psi(\bar{x})\subset\Psi(x)+\varepsilon\boldsymbol{B}_{\mathbb{R}^{m}}
\]
 with $\boldsymbol{B}_{\mathbb{R}^{n}}$  denotes the unit ball in $\mathbb{R}^{n}$ centered at the origin;
\item[(iv)] upper semicontinuous at $\bar{x}\in\mathrm{dom}\Psi$
(in Hausdorff sense) if, for every $\varepsilon>0$, there exists a
neighborhood $U$ of the point $\bar{x}$ such that for any $x\in U$,
\[
e(\Psi(x),\Psi(\bar{x}))<\varepsilon  \;\, \mbox{ or equivalently } \;\,\Psi(x)\subset\Psi(\bar{x})+\varepsilon\boldsymbol{B}_{\mathbb{R}^{m}}.
\]
\end{itemize}
\end{definition}

In the sequel, we will also use the concept of Painlev\'{e}--Kuratowski \emph{outer/upper} and \emph{inner/lower} limit for a set-valued mapping $\Psi :\mathbb{R}^n \rightrightarrows \mathbb{R}^m$ at a point $\bar x$, which are respectively defined by
\[
\begin{array}{rcl}
\underset{x \longrightarrow \bar x}\limsup~\Psi(x) &:=& \left\{y\in \mathbb{R}^m |\;\;\exists x_k \rightarrow \bar x, \;\, \exists y_k \rightarrow y \;\,            \mbox{ with } \;y_k\-\in \Psi(x_k)\; \mbox{ for all }\; k  \right\},\\[2ex]
\underset{x \longrightarrow \bar x}\liminf~\Psi(x)&:=&\left\{y\in \mathbb{R}^m |\;\;\forall x_k   \rightarrow \bar x, \;\, \exists y_k \rightarrow y \;\,   
\mbox{ with } \;y_k\-\in \Psi(x_k)\; \mbox{ for all }\; k \, \mbox{   {sufficiently large}} \right\}.
\end{array}
\]
  {
When we have $\underset{x \longrightarrow \bar x}\limsup~\Psi(x) = \underset{x \longrightarrow \bar x}\liminf~\Psi(x)$, we say that the Painlev\'{e}--Kuratowski limit  exists at $\bar x$ and in this case, we can write $\underset{x \longrightarrow \bar x}\lim~\Psi(x) :=\underset{x \longrightarrow \bar x}\limsup~\Psi(x) = \underset{x \longrightarrow \bar x}\liminf~\Psi(x)$.}


Throughout the paper, stationarity concepts tailored to problem \eqref{PBP} will also play a key role. A point $\bar{x}$ is said to be a C-stationary point for \eqref{PBP} if there exists a vector $\left(\bar{y},\bar{u}, \alpha,\beta,\gamma\right)$ such that 
	\begin{eqnarray}
	(\bar{x}, \bar{y},\bar{u})\in \mbox{gph}\mathcal{S}_{p},\label{St0}\\
	\nabla_{x}F(\bar{x},\bar{y})+\sum\limits_{i=1}^{p}\alpha_{i}\nabla G_{i}
	(\bar{x})+\sum\limits_{l=1}^{m}\beta_{l}\nabla_{x}\mathcal{L}_{l}(\bar{x}
	,\bar{y},\bar{u})+\sum\limits_{i=1}^{q}\gamma_{i}\nabla_{x}g_{i}(\bar{x}
	,\bar{y})=0,\label{St3}\\
	\nabla_{y}F(\bar{x},\bar{y})+\sum\limits_{l=1}^{m}\beta_{l}\nabla
	_{y}\mathcal{L}_{l}(\bar{x},\bar{y},\bar{u})+\sum\limits_{i=1}^{q}\gamma
	_{i}\nabla_{y}g_{i}(\bar{x},\bar{y})=0,\label{St4}\\[1ex]
	\alpha\geq0,\;\; G(\bar{x})\leq 0, \;\; \alpha^{T}G(\bar{x})=0,\label{St5}\\[1.5ex]
	\nabla_{y}g_{\nu}(\bar{x},\bar{y})\beta=0,  \;\; \gamma_{\eta}=0,\label{St6}\\
	\forall i\in\theta:\;\, \gamma_{i}\sum\limits_{l=1}^{m}\beta_{l}\nabla_{y_{l}}g_{i}(\bar{x},\bar{y})\geq0,\label{St8}
	\end{eqnarray}
	where $\gamma_{\eta}:=(\gamma_{i})_{i\in\eta}$,  $\nabla_{y}g_{v}(\bar{x},\bar{y})\beta$ denotes the vector with components $\sum\limits_{l=1}^{m}\beta_{l}\nabla_{y_{l}}g_{i}(\bar{x},\bar{y})$ for $i\in v$,
	and the index sets $\eta$, $\theta$, and $\nu$ are respectively defined as follows:
	\begin{align}
	\eta &  :=\eta(\bar{x},\bar{y},\bar{u}):=\left\{  i\in\left\{
	1,\ldots,q\right\}:\;\; \bar{u}_{i}=0,\text{ }g_{i}(\bar{x},\bar{y})<0\right\},\label{eq}\\
	\theta &  :=\theta(\bar{x},\bar{y},\bar{u}):=\left\{  i\in\left\{
	1,\ldots,q\right\}: \;\;\bar{u}_{i}=0,\text{ }g_{i}(\bar{x},\bar{y})=0\right\},\label{eqq}\\
	\nu &  :=\nu(\bar{x},\bar{y},\bar{u}):=\left\{  i\in\left\{  1,\ldots,q\right\}
	:\;\; \bar{u}_{i}>0,\text{ }g_{i}(\bar{x},\bar{y})=0\right\}.\label{eqqq}
	\end{align}
Similarly, we can define the M- and S-stationarity by respectively replacing condition \eqref{St8} by
\begin{eqnarray}
    \forall i\in\theta:\;\, (\gamma_{i}<0\wedge\sum\limits_{l=1}^{m}\beta
_{l}\nabla_{y_{l}}g_{i}(\bar{x},\bar{y})<0)\vee\gamma_{i}\sum\limits_{l=1}%
^{m}\beta_{l}\nabla_{y_{l}}g_{i}(\bar{x},\bar{y})=0,\label{S7b}\\
\forall i\in\theta:\;\, \text{ \ }\gamma_{i}\leq0\wedge\sum\limits_{l=1}^{m}%
\beta_{l}\nabla_{y_{l}}g_{i}(\bar{x},\bar{y})\leq0. \label{9b}
\end{eqnarray}
These concepts were introduced and studied in \cite{D2}. Further stationarity concepts for pessimistic bilevel optimization and relevant details can be found in the latter reference.

Next, we just recall a result, which provides the framework ensuring that a local optimal solution of problem \eqref{PBP} satisfies the C-stationarity conditions. To proceed, we need to introduce some assumptions. First, the upper-level regularity condition will be said to hold at $\bar{x}$ if
\begin{equation}\label{UMFCQ}
\left.
\begin{array}{r}
\nabla G(\bar{x})^\top \alpha =0\\
\alpha\geq 0, \;\; \alpha^\top G(\bar{x})=0
\end{array}
\right\}
\;\; \Longrightarrow  \;\; \alpha =0.
\end{equation}
Similarly, the lower-level regularity condition will be said to hold at $(\bar{x},\bar{y})$ if we have
\begin{equation}\label{LMFCQ}
\left.
\begin{array}{r}
\nabla_y g(\bar{x},\bar{y})^\top \beta =0\\
\beta\geq 0, \;\; \beta^\top g(\bar{x},\bar{y})=0
\end{array}
\right\}
\;\; \Longrightarrow  \;\; \beta =0.
\end{equation}
Obviously, these regularity   {conditions} correspond to the Mangasarian-Fromowitz constraint qualification for the upper- and lower-level constraints, respectively.
For the remaining qualification conditions, we introduce the C-qualification conditions, which are respectively defined at the point $(\bar{x},\bar{y},\bar
{u})$ by 
\begin{align}
(\beta,\gamma)  &  \in\Lambda^{ec}(\bar{x},\bar{y},\bar{u},0)\implies
\beta=0,\gamma=0,\medskip \label{CQMM-1}\tag{A$_{1}^{c}$}\\
(\beta,\gamma)  &  \in\Lambda_{y}^{ec}(\bar{x},\bar{y},\bar{u},0)\implies
\nabla_{x}\mathcal{L}(\bar{x},\bar{y},\bar{u})^{T}\beta+\nabla_{x}g(\bar
{x},\bar{y})^{T}\gamma=0,\medskip  \label{CQMM-2}\tag{A$_{2}^{c}$}\\
(\beta,\gamma)  &  \in\Lambda_{y}^{ec}(\bar{x},\bar{y},\bar{u},0)\implies
\beta=0,\gamma=0,\text{ }  \label{CQMM-3}\tag{A$_{3}^{c}$}
\end{align}
with the C-multiplier set $\Lambda^{ec}(\bar{x},\bar{y},\bar{u},0)$ resulting from setting $v=0$ in the set-valued mapping
\begin{equation}\label{C-multipliers set}
\Lambda^{ec}(\bar{x},\bar{y},\bar{u},v) := \left\{(\beta,\gamma)\in\mathbb{R}^{m+q}\left|\;\;
\begin{array}{l}
\nabla_{y}g_{v}(\bar{x},\bar{y})\beta=0,\;\; \gamma_{\eta}=0\\[1ex]
\forall i\in\theta: \;\;\gamma_{i}(\nabla_{y}g_{i}(\bar{x},\bar{y}))\beta\geq 0\\[1ex]
v+\nabla_{x,y}\mathcal{L}(\bar{x},\bar{y},\bar{u})^\top
\beta+  {\nabla g(\bar{x},\bar{y})^{T}\gamma}=0
\end{array}
\right.\right\},
\end{equation}
while $\Lambda^{ec}_y(\bar{x},\bar{y},\bar{u},0)$ is obtained similarly by replacing the derivatives w.r.t. $(x, y)$ in the last equation in the set \eqref{C-multipliers set} by the derivatives of the same functions w.r.t. $y$ only. 
%
%
Clearly, we have $\eqref{CQMM-1} \impliedby \eqref{CQMM-3} \implies \eqref{CQMM-2}$. 
Similarly to the $M$- and $S$-stationarity concepts introduced above, we can define $M$- and $S$-qualification conditions.

\begin{theorem}[\cite{D2}]\label{Th21Sect2} Let the point $\bar{x}$ be a local optimal solution to problem 
	\eqref{PBP}, which is assumed to be upper-level regular. Suppose that the set-valued mapping $\mathcal{S}_p$ is inner
	semicontinuous at $(\bar{x},\bar{y},\bar{u})$, where \eqref{CQMM-1} and \eqref{CQMM-2} are also assumed to hold. Then there exists $\left(\alpha,\beta,\gamma\right)$  such that conditions \eqref{St0}--\eqref{St8} are satisfied.
\end{theorem}

\section{Relaxation schemes and general framework}\label{sec:relaxation schemes}
In this section, we consider $\mathcal{R}\in \left\{\mbox{S}, \mbox{LF}, \mbox{KDB}, \mbox{SU}, \mbox{KS}\right\}$, where S, LF, KDB, SU, and KS respectively represents the Scholtes, Lin and Fukushima, Kadrani, Dussault and Benchakroun, Steffensen and Ulbrich, and Kanzow and Schwartz relaxation of the KKT reformulation \eqref{MPCC} of our pessimistic bilevel optimization problem \eqref{PBP}; see \cite{HoheiselEtAlComparison2013} for the definitions of these concepts in the context of the general MPCC literature. The $\mathcal{R}$--relaxation problem associated to \eqref{MPCC} can be written as 
\begin{equation}\tag{KKT$^t_\mathcal{R}$}\label{KKT-RG}
\underset{x\in X}{\min}~\psi_{\mathcal{R}}^{t}(x):=\underset{(y,u)\in\mathcal{D}^{t}_{\mathcal{R}}(x)}{\max}F(x,y),
\end{equation}
where, for all $t>0$,  $\mathcal{D}^{t}_{\mathcal{R}}$ represents the $\mathcal{R}$--relaxation of the set-valued map $\mathcal{D}$ in \eqref{KKT system}; i.e., 
\begin{equation}\label{Dt}
\mathcal{D}^{t}_{\mathcal{R}}(x):=\left\{(y,u)\in\mathbb{R}^{m+q}\left\vert\,
\mathcal{L}(x,y,u)=0, \;\, \phi^t_{i, \mathcal{R}}(x, y, u)\leq 0, \;\; i=1, \ldots, q \right.\right\},
\end{equation}
where for all $t>0$ and $i=1, \ldots, q$, the function $\phi^t_{i, \mathcal{R}}$ is defined by
\begin{equation}
    \phi^t_{i, \mathcal{R}}(x, y, u):=\left\{ 
    \begin{array}{lll}
       \left(\begin{array}{cc}
            g_i(x,y)  \\
            -u_i\\
            -u_i g_i(x,y)-t
       \end{array} \right)  &  \mbox{ if } & \mathcal{R}:=S,\\[1.5ex]
\left(\begin{array}{cc}
           -(u_{i}g_{i}(x,y)+t^{2}) \\[1ex]
-(u_{i}+t)(-g_{i}(x,y)+t)+t^{2}
       \end{array} \right)  &  \mbox{ if } & \mathcal{R}:=LF,\\[3ex]
\left(\begin{array}{cc}
            g_i(x,y) -t  \\
            -u_i - t\\
            -(u_i -t)(g_i(x,y) + t)
       \end{array} \right)  &  \mbox{ if } & \mathcal{R}:=KDB,\\[3ex]
 \left(\begin{array}{cc}
            g_i(x,y)  \\ 
            -u_i\\
            \varphi^t_{i, SU}(x,y,u)
       \end{array} \right)  &  \mbox{ if } & \mathcal{R}:= SU,\\[3ex]
 \left(\begin{array}{cc}
            g_i(x,y)  \\
            -u_i\\
             \varphi^t_{i, KS}(x,y,u)
       \end{array} \right)  &  \mbox{ if } & \mathcal{R}:=KS.
    \end{array}
    \right.
\end{equation}
For $t>0$ and $i=1, \ldots, q$, the function $\varphi^t_{i, SU}$ is given by
\[
\varphi^t_{i, SU}(x,y,u):=\left\{
\begin{tabular}
[c]{lll}
$2u_{i}$ & if & $g_{i}(x,y)+u_{i}\leq-t$,\\
$-2g_{i}(x,y)$ & if & $g_{i}(x,y)+u_{i}\geq t$,\\
$u_{i}-g_{i}(x,y)-t\theta\left(\dfrac{u_{i}+g_{i}(x,y)}{t}\right)$ & if & $\left\vert
u_{i}+g_{i}(x,y)\right\vert <t,$
\end{tabular}
\right.
\]
where $\theta(\cdot)$ is a suitable regularization function (see details in \cite{SU}) 
and $\varphi^t_{i, KS}$ is defined by
\[
\varphi^t_{i, KS}(x,y,u):=\left\{
\begin{tabular}
[c]{ll}
$(u_{i}-t)(-g_{i}(x,y)-t)$ & $\text{if\;\; } u_{i}-g_{i}(x,y)\geq2t$,\\[1ex]
$-\dfrac{1}{2}\left((u_{i}-t)^{2}+(-g_{i}(x,y)-t)^{2}\right)$ & $\text{\ if\;\; }u_{i}
-g_{i}(x,y)<2t.$
\end{tabular}
\right.
\]

Throughout the paper, we assume that for all $t>0$, $x\in \mathbb{R}^n$, and
$\mathcal{R}\in \left\{\mbox{S}, \mbox{LF}, \mbox{KDB}, \mbox{SU}, \mbox{KS}\right\}$, the set  $\mathcal{D}^{t}_{\mathcal{R}}(x)$ is nonempty. Moreover, we use the notation
\[
\mathcal{S}_{\mathcal{R}}^{t}(x):=\left\{  (y,u)\in\mathcal{D}^{t}_{\mathcal{R}}(x)\left\vert
\;F(x,y)\geq\psi_{\mathcal{R}}^{t}(x)\right.  \right\}
\]
to describe the optimal solution set-valued mapping of the parametric optimization problems associated to the optimal value function $\psi^t_{\mathcal{R}}$.
An illustration of the calculation of the function $\psi_{\mathcal{R}}^{t}$ and the set-valued mappings $\mathcal{D}^{t}_{\mathcal{R}}$ and $\mathcal{S}_{\mathcal{R}}^{t}$, for $t>0$, is given in \cite{BNZ} in the case where $\mathcal{R}:=\mbox{S}$.
The following result summarizes some basic properties of the set-valued mapping $\mathcal{D}_{\mathcal{R}}^{t}$ as $t$ and $\mathcal{R}$ vary; see Appendix \ref{A} for the proof. 

\begin{proposition}\label{lem} For any $x\in\mathbb{R}^{n}$, it holds that:
	\begin{enumerate}

\item[$(a)$] $\mathcal{D}(x)$ and  $\mathcal{D}^{t}_{\mathcal{R}}(x)$ are closed for all $t > 0$ and $\mathcal{R}\in \left\{\mbox{S}, \mbox{LF}, \mbox{KDB}, \mbox{SU}, \mbox{KS}\right\}$;

\item[$(b)$] $\mathcal{D}^{t_{1}}_{\mathcal{R}}(x)\subset\mathcal{D}^{t_{2}}_{\mathcal{R}}(x)$ for any
$t_{2}>t_{1}>0$ and $\mathcal{R}\in \left\{\mbox{S},  \mbox{SU}, \mbox{KS}\right\}$;

\item[$(c)$] $\mathcal{D}(x) = \underset{t>0}{\cap}\mathcal{D}^{t}_{\mathcal{R}}(x)$ for all $\mathcal{R}\in \left\{\mbox{S}, \mbox{LF},  \mbox{SU}, \mbox{KS}\right\}$; and for $\mathcal{R}=\mbox{KDB}$, it holds that 
$\underset{t>0}{\cap}\mathcal{D}_{KDB}^{t}(x)\subset\mathcal{D}(x)$, while 
for any sequence $(t_{k})\downarrow0$ and $(y,u)\in\mathcal{D}(x)$ with 
$(u_{i},g_{i}(x,y))\neq(0,0)$ for all $i=1, \ldots, q$, there exists $k_{0}\in\mathbb{N}$ such that $(y,u)\in\underset{k\geq k_{0}}{\cap}\mathcal{D}_{KDB}^{t_{k}}(x)$;

\item[$(d)$] $e\left(\mathcal{D}(x),\mathcal{D}^{t}_{\mathcal{R}}(x)\right)=0$ for all $t>0$ and $\mathcal{R}\in \left\{\mbox{S}, \mbox{LF}, \mbox{SU}, \mbox{KS}\right\}$;

 \item[$(e)$] $\mathcal{D}(x) = \underset{t\downarrow0}{\lim}\mathcal{D}^{t}_{\mathcal{R}}(x)$ in the Painlev\'{e}-Kuratowski sense  for all $\mathcal{R}\in \left\{\mbox{S}, \mbox{LF}, \mbox{SU}, \mbox{KS}\right\}$, while for $\mathcal{R}=\mbox{KDB}$, it holds that    {$\underset{(t,x')\rightarrow (0^+,x)}{\lim\sup}
\mathcal{D}_{KDB}^{t}(x')\subset\mathcal{D}(x)$}.
\end{enumerate}
\end{proposition}
In what follows, we show that properties (b), (c), and (d) are not generally satisfied for the $\mbox{KDB}$-relaxation. 
\begin{example}
Consider the scenario of problem \eqref{PBP} defined  with 
\begin{equation}\label{ex0pb}
F(x,y):=y, \;\; X:=[0, \, 1], \;\; K(x):=[0, \, 1], \;\; \mbox{ and }\;\; f(x,y):= xy.
\end{equation}
Observe that  any $\bar{x}\in\left]  0,1\right]$ is optimal for  \eqref{ex0pb}.
  {The KDB-relaxed problem  can be written  for $x \in ]0,1] $ and $ t>0$ sufficiently small $\left(\text{namely},\; 0 < t < \dfrac{x}{2}\right)$  with }   
\[
\mathcal{D}_{KDB}^{t}(x)=\left\{(y,(u_{1},u_{1}-x))\in\mathbb{R}^{3}\,\left\vert\;
x-t\leq u_{1}\leq x+t\wedge -t\leq y\leq t\right.\right\}.
\]
In the case where $x = 0$, we obtain the following for any $0<t\leq 1/2$: 
\[
\mathcal{D}_{KDB}^{t}(0)=\left\{  (y,(u_{1},u_{1}))\in\mathbb{R}
^{3}\,\ \ \left\vert
\begin{array}
[c]{c}
-t\leq u_{1} < t\wedge t\leq y\leq 1- t\\
\vee\\
u_{1} =t \wedge -t\leq y\leq 1+ t
\end{array}
\right.  \right\}.
\]
Clearly, inclusion $\mathcal{D}_{KDB}^{1/4}(0) \subset \mathcal{D}_{KDB}^{1/2}(0)$ does not hold and similarly, inclusion $\mathcal{D}(0) \subset \mathcal{D}_{KDB}^{1/2}(0)$ does not hold either because  $\mathcal{D}(0)=[0, \, 1]\times \left\{(0,0) \right\}$. And for property (d), we have 
$e\left(\mathcal{D}(0),\mathcal{D}^{1/2}_{KDB}(0)\right)=\frac{1}{2}$. \hfill \qed 
\end{example}
The behavior of  property (b) in the context $\mathcal{R} := \mbox{LF}$  is still unclear, and will be studied in a separate work.

For the next result, we introduce the concept of optimal solution that we use throughout the paper for problems \eqref{MPCC} and \eqref{KKT-RG} for $t>0$ and $\mathcal{R}\in \left\{\mbox{S}, \mbox{LF}, \mbox{KDB}, \mbox{SU}, \mbox{KS}\right\}$. 
A point $\bar{x}\in X$ will be said to be a local optimal solution for \eqref{MPCC}
(resp. \eqref{KKT-RG}) if there exists a neighborhood $U$ of $\bar{x}$ such that 
\begin{equation}\label{OptimalSolDef-New}
\forall x\in X\cap U: \;\psi_{p}(\bar{x})\leq\psi_{p}(x) \left(\mbox{resp.}\;\,\psi_{\mathcal{R}}^{t}(\bar{x})\leq\psi_{\mathcal{R}}^{t}(x)\right)
\end{equation}
for $t>0$ and $\mathcal{R}\in \left\{\mbox{S}, \mbox{LF}, \mbox{KDB}, \mbox{SU}, \mbox{KS}\right\}$. 
Similarly, $\bar{x}\in X$ will be said to be a global optimal solution for \eqref{MPCC}
(resp. \eqref{KKT-RG})  if \eqref{OptimalSolDef-New} holds with $U=\mathbb{R}^n$.
Let us start here by showing that the objective function of problem \eqref{KKT-RG} for $t:=t_k$ can converge to that of problem \eqref{MPCC} as $t_k\downarrow 0$ with $k\rightarrow\infty$.
  {
\begin{theorem} \label{fprop}
Let $\mathcal{R}\in \left\{\mbox{S}, \mbox{LF}, \mbox{SU}, \mbox{KS}\right\}$ then 
\begin{enumerate}
\item[$(a)$]  For all $x\in \mathbb{R}^n$ and $t>0$, $\psi_p(x) \leq \psi^t_{\mathcal{R}}(x)$;
\item[$(b)$] If the function  $t\mapsto\psi_{\mathcal{R}}^{t}(x)$ is  upper semicontinuous at $0^+$ for any $x\in \mathbb{R}^n$ and
	 $(t_{k})\downarrow0$, then for any $x\in\mathbb{R}^{n}$, we have
	$\psi_{\mathcal{R}}^{t_{k}}(x)\rightarrow\psi_{p}(x)$ as $k\rightarrow\infty$.
\end{enumerate}
\end{theorem}
\begin{proof}
Let $\mathcal{R}\in \left\{\mbox{S}, \mbox{LF},  \mbox{SU}, \mbox{KS}\right\}$. From Proposition \ref{lem}(c), $\mathcal{D}(x) \subset\mathcal{D}^{t}_{\mathcal{R}}(x)$ for any $x\in\mathbb{R}^{n}$ and $t>0$; hence, Theorem \ref{fprop}(a) holds. 
Consider now a decreasing sequence $(t_{k})$ to $0$  as well as $x\in\mathbb{R}^{n}$.
Since the sequence $(\psi_\mathcal{R}^{t_{k}}(x))$ is bounded from below by
$\psi_{p}(x)$ and the function $t\mapsto\psi_{\mathcal{R}}^{t}(x)$ is  upper semicontinuous at $0^+$, one has 
\begin{equation*}
\psi_{p}(x)\leq\underset{k\rightarrow\infty}{\lim \inf}\,\psi_\mathcal{R}^{t_{k}}(x)\leq \underset{k\rightarrow\infty}{\lim \sup}\,\psi_\mathcal{R}^{t_{k}}(x)\leq\psi_{p}(x) 
\end{equation*}
so that the assertion in Theorem \ref{fprop}(b) is satisfied. \hfill \qed 
\end{proof}}

The next example shows that assertion (1) in Theorem \ref{fprop} is not necessarily satisfied for $\mathcal{R} =\mbox{KDB}$.
\begin{example}
Consider a scenario of problem \eqref{PBP} defined by
\begin{equation}
F(x,y):=x+y,\text{ \ }X:=[0,1],\text{ \ }Y(x):=\left]  -\infty,1\right]
,\text{ and \ \ }f(x,y):=-xy. \label{e1}
\end{equation}
Then $\psi_{p}(x)=x+1$ given that 
\[
\mathcal{D}(x)=\left\{
\begin{array}{ll}
\left\{ (1,x) \right\} & \text{ \ \ if \ \ }x\neq0,\\
\left]  -\infty,1\right]  \times\left\{  0\right\} & \text{ \ \  if \ \ }x=0.
\end{array}
\right.  
\]
We obtain for the KDB-relaxation that 
$$\mathcal{D}^{t}_{\text{KDB}}(x)=\left\{
\begin{array}
[c]{c}
\left[1-t,1+t\right]\times\left\{   x \right\}  \text{ \ \ \ \ \  if \ \ }x\neq0\\
\left]  -\infty,1-t\right]  \times\left\{  0\right\}  \text{ \ \ \ \ \ \ \   if \ \ }x=0
\end{array}
\right.  
\, \, \, \text{and} \,\ \, \psi_{KDB}
^{t}(x)=\left\{
\begin{array}
[c]{l}
x+1+t\text{ \ \ \ \ if } \;\,x\in\left]  0,1\right] \\
1-t\text{\ \ \ \ \ \ \ \ \ \ \ if }\;\, x=0
\end{array}
\right. $$
for any $t>0$ sufficiently small. So that $ \psi_{\text{KDB}}
^{t}(0)=1-t < \psi_{p}(0)$. \hfill \qed 
\end{example}


\section{Method and convergence analysis}\label{sec:method and convergence}
We start this section by introducing the $\mathcal{R}$-relaxation algorithm in Algorithm~\ref{algorithm1-S} for the KKT reformulation \eqref{MPCC} of the pessimistic bilevel optimization problem \eqref{PBP}. 
Note that a global/local optimal solution of problem \eqref{MPCC}  is equivalent to a global/local optimal solution of problem \eqref{PBP} under mild assumptions \cite{Au}.
\begin{algorithm}[H]
\caption{$\mathcal{R}$--relaxation method for pessimistic bilevel optimization}
\label{algorithm1-S}
\begin{algorithmic}
\STATE \textbf{Step 0}: Choose $\mathcal{R}\in \left\{\mbox{S}, \mbox{LF}, \mbox{KDB}, \mbox{SU}, \mbox{KS}\right\}$, $t^0>0$, and set  $k:=0$.
\STATE \textbf{Step 1}: Solve problem \eqref{KKT-RG} for $t:=t^k$.
\STATE \textbf{Step 2}: Select $0<t^{k+1}<t^k$, set $k:=k+1$, and go to Step 1.
\end{algorithmic}
\end{algorithm}
${}$\\[-7.1ex]

It is important to emphasize that Step 1   {of Algorithm \ref{algorithm1-S}} consists of computing a solution of the corresponding $\mathcal{R}$-relaxation of problem \eqref{MPCC} globally, locally, or just a stationarity point in a sense that will be defined later. Our primary goal in this section is to establish that for a given relaxation, a sequence of   {points} computed from Algorithm~\ref{algorithm1-S} converges to global/local optimal solution, or a stationary point of \eqref{MPCC} as $t^k\downarrow 0$, under suitable assumptions.
First, we give the convergence of Algorithm \ref{algorithm1-S} when problem \eqref{KKT-RG} is solved globally. 
\begin{theorem}\label{sg}
Let $(t_{k})\downarrow0 $
	and $(x^{k})$ be a sequence such that the point $x^{k}$ is a global optimal
	solution of problem \eqref{KKT-RG} for $t:=t_k$. If $x^{k}\rightarrow\bar{x}$ as
	$k\rightarrow\infty$, then the point $\bar{x}$ is a global optimal solution of problem \eqref{MPCC} provided that one of the following assumptions holds:
 \begin{itemize}
     \item[$(a)$] If $\mathcal{R}\in \left\{\mbox{S}, \mbox{LF}, \mbox{SU}, \mbox{KS}\right\}$, the function $x \mapsto \psi_{p}(x)$ is lower semicontinuous at $\bar{x}$ and the function $t\mapsto\psi_{\mathcal{R}}^{t}(x)$ is upper
	semicontinuous at $0^+$ for any $x\in \mathbb{R}^n$.  
 \item[$(b)$] If $\mathcal{R}=KDB$ and for any $x\in \mathbb{R}^n$, the function $(t,u)\mapsto\psi_{\mathcal{R}}^{t}(u)$ is lower semicontinuous at  $(0^+,x) $ and the function $t\mapsto\psi_{\mathcal{R}}^{t}(x)$ is upper semicontinuous at $0^+$. 
 \end{itemize}
\end{theorem}
\begin{proof}
 (a) As $x^{k}\in X$ is a global optimal solution of \eqref{KKT-RG} for $t:=t_k$, from assertion (a) of Theorem \ref{fprop}, 
	\[
	\psi_{p}(x^{k})\leq\psi_{\mathcal{R}}^{t_{k}}(x^{k})\leq\psi_{\mathcal{R}}^{t_{k}}(x)\;\, \mbox{ for all }\;\, x\in X.
	\]
	 Fix $x\in X$. Since $t\mapsto \psi_{\mathcal{R}}^{t}(x)$ is upper semicontinuous at $0^+$ and $\psi_{p}$ is lower semicontinuous at $\bar{x}$, 
	\begin{equation*}
	\psi_{p}(\bar{x})\leq\text{ }\underset{k\rightarrow\infty}{\lim\inf}\text{ }\psi
	_{p}(x^{k})\leq\text{ }\underset{k\rightarrow\infty}{\lim\inf}\text{ }\psi_{\mathcal{R}}^{t_{k}
	}(x)\leq\text{ }\underset{k\rightarrow\infty}{\lim\sup}\text{ }\psi_{\mathcal{R}}^{t_{k}}
	(x)\leq\psi_{p}(x).
	\end{equation*}
(b) Let $x\in X$.	As $x^{k}\in X$ is a
	global optimal solution of problem \eqref{KKT-RG} for $t:=t_k$, we have
	\[
	\psi_{KDB}^{t_{k}}(x^{k})\leq\psi_{KDB}^{t_{k}}(x)\;\, \mbox{ for all }\;\, x\in X.
	\]
	The assumptions of lower and upper semicontinuity lead to 
	\begin{equation*}
	\psi_{p}(\bar{x})\leq \text{ }\underset{k\rightarrow\infty}{\lim\inf}\text{ }\psi_{KDB}^{t_{k}
	}(x^{k})\leq\text{ }\underset{k\rightarrow\infty}{\lim\sup}\text{ }\psi_{KDB}^{t_{k}}
	(x)\leq\psi_{p}(x),
	\end{equation*} 
   {which concludes that $\bar{x}$ is a global optimal solution of problem \eqref{MPCC}}. \hfill \qed
\end{proof}

If in Step 1 of Algorithm \ref{algorithm1-S}, local optimal solutions of problem \eqref{KKT-RG}, for $t>0$, are computed, then we can state the following convergence result. 
\begin{theorem}\label{sg1}
Let $(t_{k})\downarrow0$  and $(x^{k})$ be a sequence of optimal solutions of problem \eqref{KKT-RG}, for $t=t^k$, in $X\cap B(x^{k},r^{k})$, where $(r^{k})$ is a real sequence such that $r^{k}>\bar{r}>0$ for all $k$. The point $\bar{x}$ is a local optimal solution of \eqref{MPCC} provided that $x^{k} \rightarrow\bar{x}$ as $k\rightarrow+\infty$ and one of the following conditions holds:
 \begin{itemize}
     \item[$(a)$] $\mathcal{R}\in \left\{\text{S}, \text{LF}, \text{SU}, \text{KS}\right\}$, the function $x\mapsto\psi_{p}(x)$ is lower semicontinuous at  $\bar{x}$ and the function $t\mapsto\psi_{\mathcal{R}}^{t}(x)$ is upper semicontinuous at $0^+$ for any $x \in \mathbb{R}^n$.  
 \item[$(b)$]  $\mathcal{R}=\text{KDB}$ and  for any $x\in \mathbb{R}^n$, the function  $(t,u)\mapsto\psi_{\mathcal{R}}^{t}(u)$ is lower semicontinuous at  $(0^+,x)$ and the function  $t\mapsto\psi_{\mathcal{R}}^{t}(x)$ is upper semicontinuous at  $0^+$ for any $x\in \mathbb{R}^n$.
 \end{itemize}
\end{theorem}
\begin{proof}
 (a) Let $x\in X\cap B(\bar{x},\frac{\bar{r}}{2} )$. Since  $\bar{x}$ is a limit point of a subsequence of $(x^{k})$, then for all $k$ sufficiently large, $x^{k}\in X\cap B(\bar{x},\frac{\bar{r}}{2})$. Hence, $x\in X\cap B(x^{k},r^{k} )$ and it holds that
	\[
	\psi_{p}(x^{k})\leq\psi_{\mathcal{R}}^{t_{k}}(x^{k})\leq\psi_{\mathcal{R}}^{t_{k}}(x)
	\]
	since $x^{k}$ is an optimal solution of problem \eqref{KKT-RG} for $\mathcal{R}\in \left\{\mbox{S}, \mbox{LF}, \mbox{SU}, \mbox{KS}\right\}$. Given that  the function $x\mapsto\psi_{p}(x)$ is lower semicontinuous at  $\bar{x}$ and  $t\mapsto\psi_{\mathcal{R}}^{t}(x)$ is upper semicontinuous at $0^+$ for any $x \in \mathbb{R}^n$, it holds that 
	\begin{equation*}
	\psi_{p}(\bar{x})\leq\text{ }\underset{k\rightarrow\infty}{\lim\inf}\text{ }\psi
	_{p}(x^{k})\leq\text{ }\underset{k\rightarrow\infty}{\lim\inf}\text{ }\psi_{\mathcal{R}}^{t_{k}
	}(x)\leq\text{ }\underset{k\rightarrow\infty}{\lim\sup}\text{ }\psi_{\mathcal{R}}^{t_{k}
	}(x)\leq \psi_{\mathcal{R}}^{0
	}(x)=\text{ }\psi_{p}(x).
	\end{equation*} 
(b) For the KDB-relaxation,  similarly for any $x\in X\cap B(\bar{x},\frac{\bar{r}}{2} )$ and $k$ sufficiently large,  we have 
	\[
	\psi_{\text{KDB}}^{t_{k}}(x^{k})\leq\psi_{\text{KDB}}^{t_{k}}(x)
	\]
	since $x^{k}$ is an optimal solution of of problem \eqref{KKT-RG} for $\mathcal{R}=KDB$. Consequently from the assumptions, 
	\begin{equation*}
	\psi_{p}(\bar{x})\leq\text{ }\underset{k\rightarrow\infty}{\lim\inf}\text{ }\psi
	^{t_{k}}_{KDB}(x^{k})\leq\text{ }\underset{k\rightarrow\infty}{\lim\sup}\text{ }\psi_{KDB}^{t_{k}
	}(x)\text{ }\leq \psi_{KDB}^{0
	}(x)=\text{ }\psi_{p}(x),
	\end{equation*} 
 which concludes the proof of the result. \hfill \qed
\end{proof}

\begin{remark}
Note that Theorems \ref{fprop}, \ref{sg}, and \ref{sg1} rely on one or more of the following semicontinuity properties: the lower semicontinuity of the function $(t,u)\mapsto\psi_{\mathcal{R}}^{t}(u)$ (resp. $x \mapsto \psi_{p}(x)$) at $(0^+,x)$ for all $x\in \mathbb{R}^n$ (resp. at $\bar{x}$) and the upper semicontinuity of the function $t\mapsto\psi_{\mathcal{R}}^{t}(x)$  at $0^+$ for any $x\in \mathbb{R}^n$. A detailed analysis of these assumptions is conducted in \cite{BNZ} for $\mathcal{R}:=\mathcal{S}$ to establish problem data-based sufficient conditions ensuring the fulfillment of these conditions. This analysis can easily be extended to construct tractable sufficient conditions for these properties, which are tailored to the cases where $\mathcal{R}\in  \left\{\mbox{LF}, \, \mbox{KDB},\, \mbox{SU},\, \mbox{KS}\right\}$. 
\end{remark}

We conclude this section by analyzing the case where in Step 1 of Algorithm \ref{algorithm1-S},  stationary points of problem \eqref{KKT-RG} for $t:=t_k$, are computed. The aim is to construct a framework ensuring that the resulting sequence converges to a stationary point of problem \eqref{MPCC}. To proceed, we introduce the  necessary optimality conditions for problem \eqref{KKT-RG} for $\mathcal{R}\in \left\{\mbox{S}, \mbox{LF}, \mbox{KDB}, \mbox{SU}, \mbox{KS}\right\}$. For a fixed number $t>0$, a point $x^t$ is said to \emph{satisfy the necessary optimality conditions} (or \emph{be a stationary point}) for problem \eqref{KKT-RG} if there exists a vector $(y^{t},u^{t}, \alpha^{t}, \beta^{t}, \lambda^{t})$ such that the following relationships are satisfied:
\begin{eqnarray}
(x^{t},y^{t},u^{t})\in\mbox{gph}\mathcal{S}^{t}_\mathcal{R}, \label{Er0S*}\\
\nabla_{x}F(x^{t},y^{t})+\nabla G(x^{t})^{\top}\alpha^{t}-\nabla
_{x}\mathcal{L}(x^{t},y^{t},u^{t})^{\top}\beta^{t}-\sum\limits_{i=1}
^{q} \lambda_{i}^{t} \nabla_{x}
\phi^t_{i, \mathcal{R}}(x^{t},y^{t},u^{t})=0,\label{Er1S*}\\
\nabla_{y}F(x^{t},y^{t})-\nabla_{y}\mathcal{L}(x^{t},y^{t},u^{t})^{\top}
\beta^{t}-\sum\limits_{i=1}^{q}  \lambda_{i}^{t}  \nabla_{y}
\phi^t_{i, \mathcal{R}}(x^{t},y^{t},u^{t})=0,\\
\forall i=1, \ldots, q:\;\, -\nabla_{y}g_{i}(x^{t},y^{t})\beta^{t}+\nabla_{u}
\phi^t_{i, \mathcal{R}}(x^{t},y^{t},u^{t})=0,\\[1ex]
\forall j=1, \ldots, p: \;\; \alpha_{j}^{t} \geq 0,\;\; G_j(x^t)\leq 0, \;\; \alpha_{j}^{t}G_{j}(x^{t})=0,\label{Er011*}\\
\forall i=1, \ldots, q:\;\; \lambda^t_i\geq 0, \;\; \phi^t_{i, \mathcal{R}}(x^{t},y^{t},u^{t}) \leq 0, \;\; \lambda^t_i \phi^t_{i, \mathcal{R}}(x^{t},y^{t},u^{t})=0.\label{Er5S*}
\end{eqnarray}
\begin{theorem}[\cite{Zemkoho}]\label{Optimality conditions for RMPCC(t)-S}For a given $t>0$ and $\mathcal{R}\in \left\{\mbox{S}, \mbox{LF}, \mbox{KDB}, \mbox{SU}, \mbox{KS}\right\}$, let the point $x^{t}$ be an upper-level regular local optimal solution for problem \eqref{KKT-RG}.
	Assume that the set-valued map $\mathcal{S}^{t}_\mathcal{R}$ is inner
	semicontinuous at $(x^{t},y^{t},u^{t})\in\mathrm{gph}\,\mathcal{S}^{t}_\mathcal{R}$ and
	the following qualification condition holds at $(x^{t},y^{t},u^{t})$:
	\begin{equation}
	\left.
	\begin{array}
	[c]{r}
	\nabla_{y}\mathcal{L}(x^{t},y^{t},u^{t})^{\top}\beta^{t}+\sum\limits_{i=1}
	^{q} \lambda^t_{i}\nabla_{y}
\phi^t_{i, \mathcal{R}}(x^{t},y^{t},u^{t}) =0\\
	\nabla_{y}g_{i}(x^{t},y^{t})\beta^{t}-\lambda^t_{i}\nabla_{u}
\phi^t_{i, \mathcal{R}}(x^{t},y^{t},u^{t})=0\\
	\lambda^t_i\geq 0, \;\;  \lambda^t_i \phi^t_{i, \mathcal{R}}(x^{t},y^{t},u^{t}) =0
	\end{array}
	\right\}  \Longrightarrow\left\{
	\begin{array}
	[c]{l}
	\beta^t=0_{m},\\
	\lambda^t=0_{q}.
	\end{array}
	\right. \label{CQS}
	\end{equation}
	Then $x^t$ is a stationary point for problem \eqref{KKT-RG}; i.e., there exists a vector $(\alpha^{t}, \beta^{t}, \lambda^{t})$ such that the optimality conditions  \eqref{Er0S*}--\eqref{Er5S*} are satisfied.
	\end{theorem}
For a detailed analysis of this result and sufficient conditions ensuring the fulfillment of the required assumptions, similarly to the case of Theorem \ref{Th21Sect2}, interested readers are referred to \cite{D2,D3,DMZns2019,Zemkoho} and references therein.

Now, in contrary to Theorem \ref{sg} (resp. Theorem  \ref{sg1}), where we assume that in Step 1 of Algorithm \ref{algorithm1-S} we are computing a global (resp. local) optimal solution for problem \eqref{KKT-RG}, we present a convergence result in the case where the subproblem instead consists of calculating stationarity points (see Appendix \ref{B} for the proof). 

\begin{theorem}\label{ConvergenceResult}
	Let $(t_{k}) \downarrow 0$ and $(x^{k})$ be a sequence such that $x^k$ is a stationary point of problem \eqref{KKT-RG} for $t:=t^k$ with $(x^{k})\rightarrow \bar{x}$ and suppose that the following condition holds: 
 \begin{equation}
	\underset{k \rightarrow
		\infty}{\lim }e\left(\mathcal{S}^{t_{k}}_\mathcal{R}(x^{k}), \;\mathcal{S}_{p}(\bar{x})\right) = 0.\label{e1}
	\end{equation}
  Assume that $\bar{x}$ is upper-level regular with  $\mathcal{S}_{p}(\bar{x})$ 
	nonempty and compact. Furthermore, let the 
conditions $(A_{1}^{m})$ and $(A_{2}^{m})$ be satisfied if $\mathcal{R}\in  \left\{\mbox{S},\, \mbox{LF}, \, \mbox{SU}\right\} $  (resp. \eqref{CQMM-1} and \eqref{CQMM-2} hold if  $\mathcal{R}\in  \left\{\mbox{KDB}, \, \mbox{KS}\right\}$) hold at all $(\bar{x},y,u)$ $\in\mathrm{gph}\mathcal{S}_{p}$. 	
	Then $\bar{x}$ is a C (resp. M)-stationary point.
\end{theorem}
The next result gives some practical conditions ensuring that assumption \eqref{e1} is satisfied.
\begin{proposition} \label{cs}
Let the sequences $(t_{k})\downarrow0$ and $(x^{k}) \rightarrow \bar{x}$ be such that $\mathcal{D}(\bar{x})$ is
nonempty and compact. Let the set-valued mapping $(t,x) \rightrightarrows \mathcal{D}_\mathcal{R}^{t}(x)$ be  upper semicontinuous at $(0^+,\bar{x})$ and $x \rightrightarrows \mathcal{D}(x)$  be lower semicontinuous at $\bar{x}$. Then the property \eqref{e1} holds for any $\mathcal{R}\in \left\{\mbox{S}, \mbox{LF}, \mbox{SU}, \mbox{KS}\right\}$ and it is satisfied for $\mathcal{R} = \text{KDB}$ if in addition the following inequality holds: $\underset{k\rightarrow\infty}{\text{ }\lim\inf}\psi_{p}(x^{k})\leq \underset
{k\rightarrow\infty}{\text{ }\lim\inf}\psi_{\text{KDB}}^{t_{k}}(x^{k})$.
\end{proposition}
\begin{proof}
Assume that for a given $\mathcal{R}\in  \left\{\text{S}, \, \text{LF}, \,\text{KDB}, \, \text{SU},\, \text{KS}\right\}$, the property (\ref{e1}) does not hold. Thus,  there exist $\delta >0$ and a sequence $(y^{k},u^{k})_{k}$ such that $(y^{k},u^{k})\in\mathcal{S}_\mathcal{R}^{t_{k}}(x^{k})$ and
\begin{equation}\label{not}
d\left((y^{k},u^{k}),\;\mathcal{S}_p(\bar{x})\right) \geq \delta \;\mbox{ for all }\; k.
\end{equation}
Therefore, since $(y^{k},u^{k})\in \mathcal{D}_\mathcal{R}^{t_{k}}(x^{k})$, 
$
d\left((y^{k},u^{k}), \,\mathcal{D}(\bar{x})\right)\leq e\left(\mathcal{D}_\mathcal{R}^{t_{k}}(x^{k}
), \,\mathcal{D}(\bar{x})\right).
$
By the upper semicontinuity of the set-valued mapping $(t,x) \rightrightarrows \mathcal{D}_\mathcal{R}^{t}(x)$ at the point $\left(0^+,\bar{x}\right)$,  we have  
\[
\underset{k \rightarrow
		\infty}{\lim }e\left(\mathcal{D}^{t_{k}}_\mathcal{R}(x^{k}), \;\mathcal{D}(\bar{x})\right)=0.
\]
Hence, one can pick a sequence  $(z^{k},w^{k})_k$ in $\mathcal{D}(\bar{x})$ such that
\[
    \underset{k \rightarrow
		\infty}{\lim }\left\Vert y^{k}-z^{k}\right\Vert = 0 \;\, \text{ and } \;\,\underset{k \rightarrow
		\infty}{\lim }\left\Vert
	u^{k}-w^{k}\right\Vert = 0
\]
hold. Due to the compactness of $\mathcal{D}(\bar{x})$, this sequence (up to a subsequence) converges to a point $(\bar{y},\bar
{u})\in\mathcal{D}(\bar{x})$ and so does the (sub)sequence $(y^{k},u^{k})$.

Since $F(\cdot,\cdot)$ is continuous and $\mathcal{D}(\cdot)$ is lower
semicontinuous, the function $\psi_{p}(\cdot)$ is lower semicontinuous at
$\bar{x}$. Let us now prove that $(\bar{y},\bar{u})\in\mathcal{S}_{p}(\bar{x})$ for $\mathcal{R}\in  \left\{\mbox{S}, \, \mbox{LF}, \, \mbox{SU},\, \mbox{KS}\right\}$. Indeed,
from assertion (a) in Theorem \ref{fprop}, 
\begin{equation} \label{psiR}
    F(x^{k},y^{k})-\psi_{p}(x^{k})\geq F(x^{k},y^{k})-\psi_\mathcal{R}^{t_{k}}
(x^{k})\geq0
\end{equation}
as $(y^{k},u^{k})\in\mathcal{S}_\mathcal{R}^{t_{k}}(x^{k})$.  Hence,
\[
F(\bar{x},\bar{y})-\psi_{p}(\bar{x})\geq F(\bar{x},\bar{y})-\underset
{k\rightarrow\infty}{\text{ }\lim\inf}\psi_{p}(x^{k})\geq \underset
{k\rightarrow\infty}{\text{ }\lim\sup}(F(x^{k},y^{k})-\psi_\mathcal{R}^{t_{k}}(x^{k}
))\geq0.
\]
Consequently, $(\bar{y},\bar{u})\in\mathcal{S}_{p}(\bar{x})$ and from the inequality
\[
d\left((y^{k},u^{k}), \;\mathcal{S}_{p}(\bar{x})\right)\leq\left\Vert (y^{k},u^{k})-(\bar
{y},\bar{u})\right\Vert,
\]
$\underset{k\longrightarrow\infty}{\lim}d\left(\left(y^{k},u^{k}\right), \;\mathcal{S}_{p}
(\bar{x})\right)=0$, which is a contradiction to (\ref{not}).  The  poof  is similar for the KDB-relaxation using the additional assumption since the first inequality  \eqref{psiR} fails in this case. \hfill \qed
\end{proof}


\section{Numerical experiments}\label{Numerical experiments}
The focus of our experiments in this section will be to implement our Algorithm \ref{algorithm1-S} on each of the relaxations $\mathcal{R}\in \left\{\mbox{S}, \mbox{LF}, \mbox{KDB}, \mbox{SU}, \mbox{KS}\right\}$. 
To proceed, we first write the optimality conditions for an $\mathcal{R}$-relaxation problem \eqref{KKT-RG} for a fixed $t>0$. Considering Theorem \ref{Optimality conditions for RMPCC(t)-S}, it follows that under suitable assumptions, a local optimal solution $x$ of problem \eqref{KKT-RG} for a fixed $t>0$ satisfies the necessary optimality conditions
\begin{eqnarray}
\nabla_{x}F(x,y)+\nabla G(x)^{\top}\alpha-\nabla
_{x}\mathcal{L}(x,y,u)^{\top}\beta  -  \nabla_{x}
\Phi^{t}_{\mathcal{R}}(x,y,u,\gamma, \mu, \delta)=0,\label{Er1S1}\\
\nabla_{y}F(x,y)-\nabla_{y}\mathcal{L}(x,y,u)^{\top}
\beta -     \nabla_{y}
\Phi^{t}_{\mathcal{R}}(x,y,u,\gamma, \mu, \delta)=0,\\
 \nabla_{y}g_{i}(x,y)\beta -\nabla_{u}
\Phi^{t}_{\mathcal{R}}(x,y,u,\gamma, \mu, \delta)=0,\\
\mathcal{L}(x,y,u)  = 0, \\
\alpha \geq 0, \; \; G(x) \leq 0, \;\; \alpha^\top G(x) = 0, \label{Doctrine-1}\\
\gamma \geq 0, \; \; \nabla_{\gamma}\Phi_{\mathcal{R}}^{t}(x,y,u,\gamma,\mu,\delta) \leq 0, \; \; \gamma^\top \nabla_{\gamma}\Phi_{\mathcal{R}}^{t}(x,y,u,\gamma,\mu,\delta) = 0, \\
\mu \geq 0, \; \;  \nabla_{\mu}\Phi_{\mathcal{R}}^{t}(x,y,u,\gamma,\mu,\delta) \leq 0, \; \; \mu^\top \nabla_{\mu}\Phi_{\mathcal{R}}^{t}(x,y,u,\gamma,\mu,\delta) = 0, \\
\delta \geq 0, \; \; \nabla_{\delta}\Phi_{\mathcal{R}}^{t}(x,y,u,\gamma,\mu,\delta) \leq 0, \; \; \delta^\top \nabla_{\delta}\Phi_{\mathcal{R}}^{t}(x,y,u,\gamma,\mu,\delta) = 0, \label{Er1S2}
\end{eqnarray}
for some  $(y,u, \alpha, \beta, \delta, \gamma, \mu)$. Note that for $t>0$ and  $i=1, \ldots, q$, we have 
\begin{equation*}
    \Phi^{t}_{i, \mathcal{R}}(x, y, u,\gamma,\mu,\delta):=\left\{ 
    \begin{array}{lll}
                 \gamma_{i} g_i(x,y)  -\mu_{i}u_i    - \delta_{i}(u_i g_i(x,y)+t)
        &  \mbox{ if } & \mathcal{R}:=\text{S},\\[1ex]
           -\gamma_{i}(u_{i}g_{i}(x,y)+t^{2}) -\delta_{i}((u_{i}+t)(-g_{i}(x,y)+t)+t^{2})
         &  \mbox{ if } & \mathcal{R}:=\text{LF},\\[2.5ex]
           \gamma_{i}( g_i(x,y) -t)  - \mu_{i}(u_i + t) - \delta_{i}(u_i -t)(g_i(x,y) + t)
         &  \mbox{ if } & \mathcal{R}:=\text{KDB},\\[2.5ex]
            \gamma_{i} g_i(x,y) -\mu_{i}u_i +   \delta_{i}\varphi^t_{i, SU}(x,y,u)
       &  \mbox{ if } & \mathcal{R}:= \text{SU},\\[2.5ex]
            \gamma_{i} g_i(x,y) -\mu_{i}u_i + \delta_{i} \varphi^t_{i, KS}(x,y,u)
        &  \mbox{ if } & \mathcal{R}:=\text{KS}
    \end{array} 
    \right.
\end{equation*}
with $\Phi^{t}_{\mathcal{R}}(x, y, u,\gamma,\mu,\delta):=\left(\Phi^{t}_{i, \mathcal{R}}(x, y, u,\gamma,\mu,\delta)\right)^q_{i=1}$.  
Observe that the system \eqref{Er1S1}--\eqref{Er1S2} is obtained from  \eqref{Er0S*}--\eqref{Er5S*} by expanding \eqref{Er0S*}. 
In order to compute points satisfying the system \eqref{Er1S1}--\eqref{Er1S2}, we can first approximately write it  as a system of   {smooth} equations by applying the  smoothing Fischer-Burmeister function $\theta^\epsilon :\mathbb{R}\times \mathbb{R} \rightarrow \mathbb{R}$, which is defined by 
\[
\theta^{\epsilon}(a, b):=\sqrt{a^2 + b^2 +   {2\epsilon}} -(a + b),
\]
where the smoothing parameter $\epsilon>0$ helps to guarantee the  differentiability of the function at points $(a, b)$ such that $a=b=0$. It is well known that the following equivalence holds:
\begin{equation}\label{perb}
    \theta^{\epsilon}(a, b) = 0 \;\;\Longleftrightarrow \;\; [a>0, \; b>0,\; ab = \epsilon ].
\end{equation}
Note that in the sequel, if $a$ and $b$ are vectors, writing $\theta^{\epsilon}(a, b)$ should be understood componentwise. Hence, the system of  optimality conditions \eqref{Er1S1}--\eqref{Er1S2} can be approximately written as follows: 
\begin{align}\label{EquFinal}
    \Psi^{\epsilon,t}_{\mathcal{R}}(\zeta) : = \left( \begin{array}{c}
      \nabla_{x}F(x,y)+\nabla G(x)^{\top}\alpha-\nabla
_{x}\mathcal{L}(x,y,u)^{\top}\beta  -  \nabla_{x}
\Phi^{t}_{\mathcal{R}}(x,y,u,\gamma, \mu, \delta)\\   
       \nabla_{y}F(x,y)-\nabla_{y}\mathcal{L}(x,y,u)^{\top}
\beta -     \nabla_{y}
\Phi^{t}_{\mathcal{R}}(x,y,u,\gamma, \mu, \delta) \\
 \nabla_{y}g_{i}(x,y)\beta -\nabla_{u}
\Phi^{t}_{\mathcal{R}}(x,y,u,\gamma, \mu, \delta) \\
\mathcal{L}(x,y,u) \\
\theta^{\epsilon}(\alpha,\,G(x)) \\
\theta^{\epsilon}(\gamma,\,\nabla_{\gamma}\Phi_{\mathcal{R}}^{t}(x,y,u,\gamma,\mu,\delta))\\
\theta^{\epsilon}(\delta, \,\nabla_{\delta}\Phi_{\mathcal{R}}^{t}(x,y,u,\gamma,\mu,\delta)) \\
\theta^{\epsilon}(\mu,\,\nabla_{\mu}\Phi_{\mathcal{R}}^{t}(x,y,u,\gamma,\mu,\delta))
    \end{array}
    \right)=0
\end{align}
with $\zeta:=(x, y,u, \alpha, \beta, \delta, \gamma, \mu)$.  Observe that if $\epsilon =0$, then the resulting system in \eqref{EquFinal} is a nonsmooth system of equations that is equivalent to the optimality conditions in \eqref{Er1S1}--\eqref{Er1S2}. The introduction of the perturbation $\epsilon$ helps to obtain a smooth system of equations. As we aim to use a solver for smooth equation systems, our focus in the remainder of this section will be on \eqref{EquFinal}. 
Thus, as the functions $F$ and $G$ (resp. $f$ and $g$) are assumed to be twice (resp. thrice) continuously differentiable (cf. Section \ref{Introduction}), for any given $\epsilon >0$, $t>0$, and $\mathcal{R}\in \left\{\mbox{S}, \mbox{LF}, \mbox{KDB}, \mbox{SU}, \mbox{KS}\right\}$, the function $\Psi^{\epsilon,t}_{\mathcal{R}}$ is continuously differentiable. 
\begin{figure}[h]
    \centering
    \begin{subfigure} 
  \centering
 \includegraphics[width=0.3\textwidth]{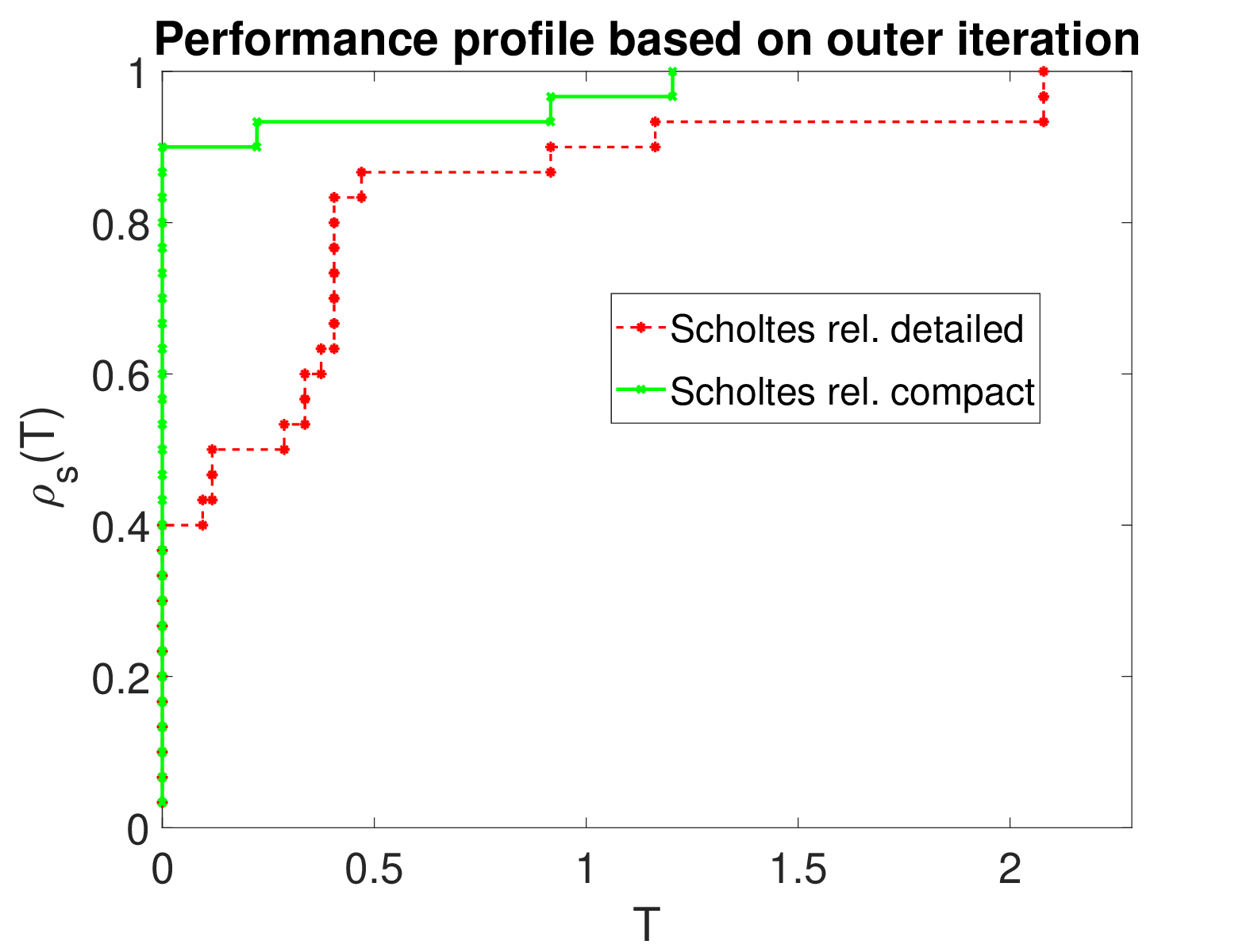}
     \end{subfigure}
 \hfill
     \begin{subfigure} 
         \centering
         \includegraphics[width=0.3\textwidth]{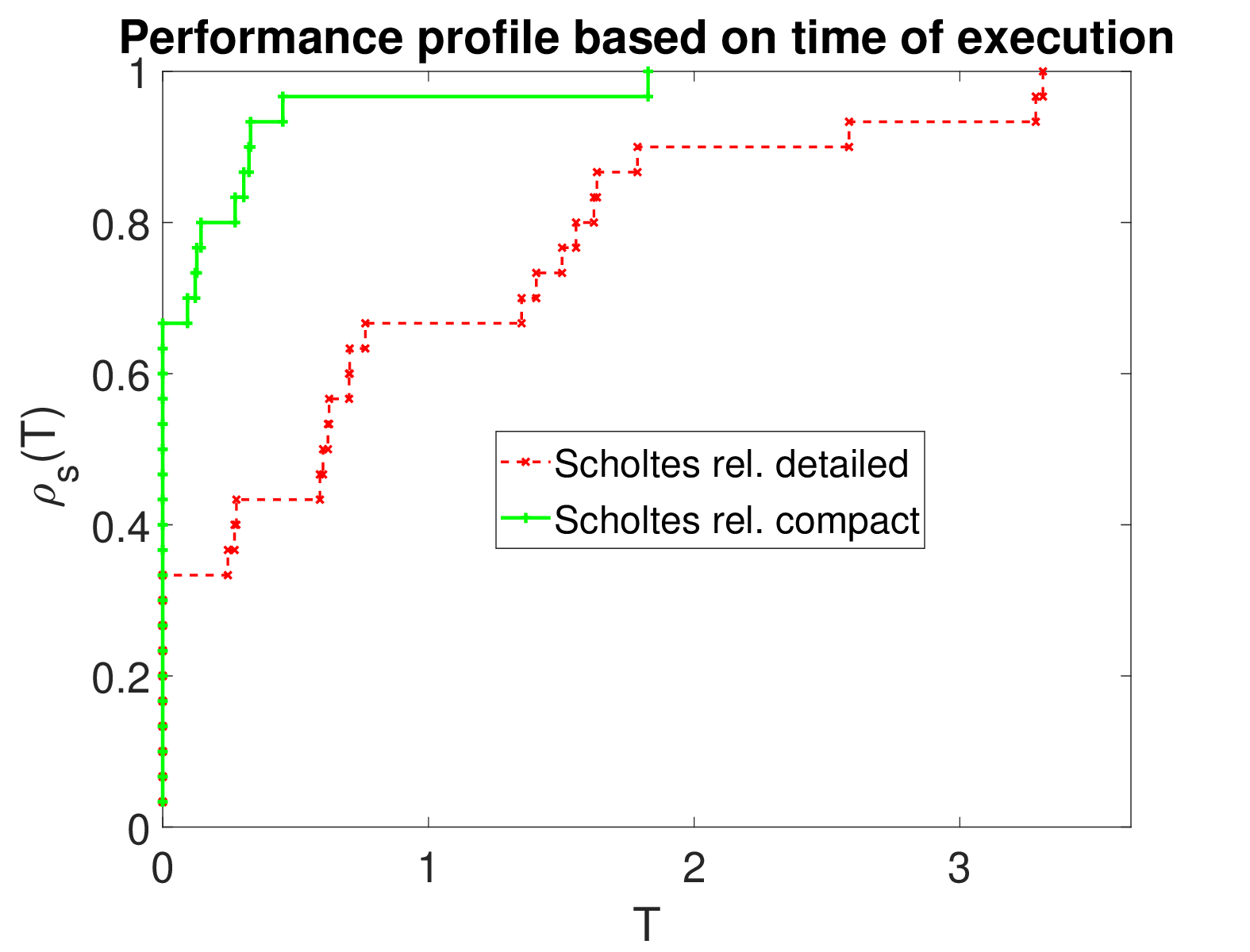}
     \end{subfigure} 
  \hfill
     \begin{subfigure} 
         \centering
         \includegraphics[width=0.3\textwidth]{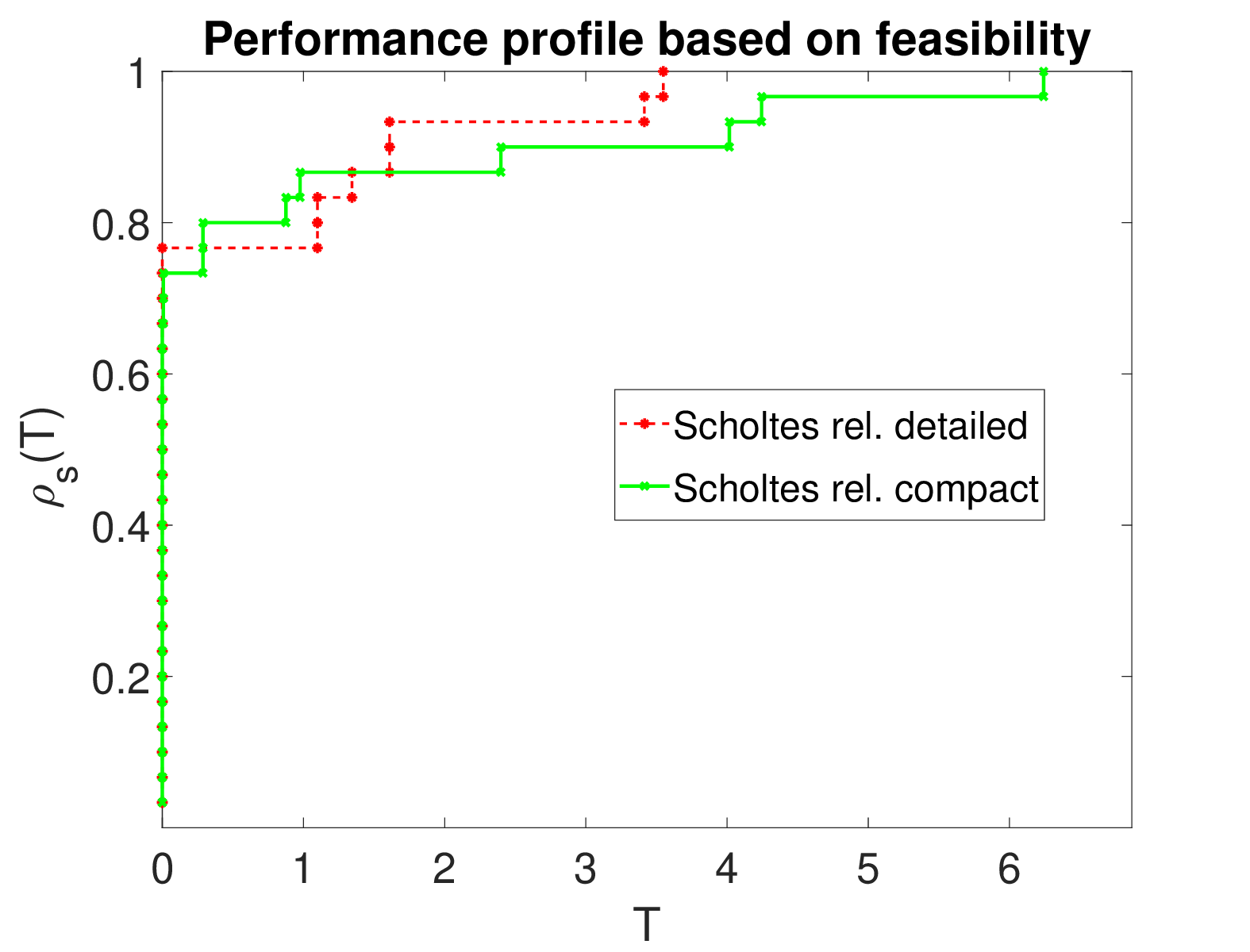}
     \end{subfigure}   
       \centering
    \begin{subfigure} 
  \centering
 \includegraphics[width=0.3\textwidth]{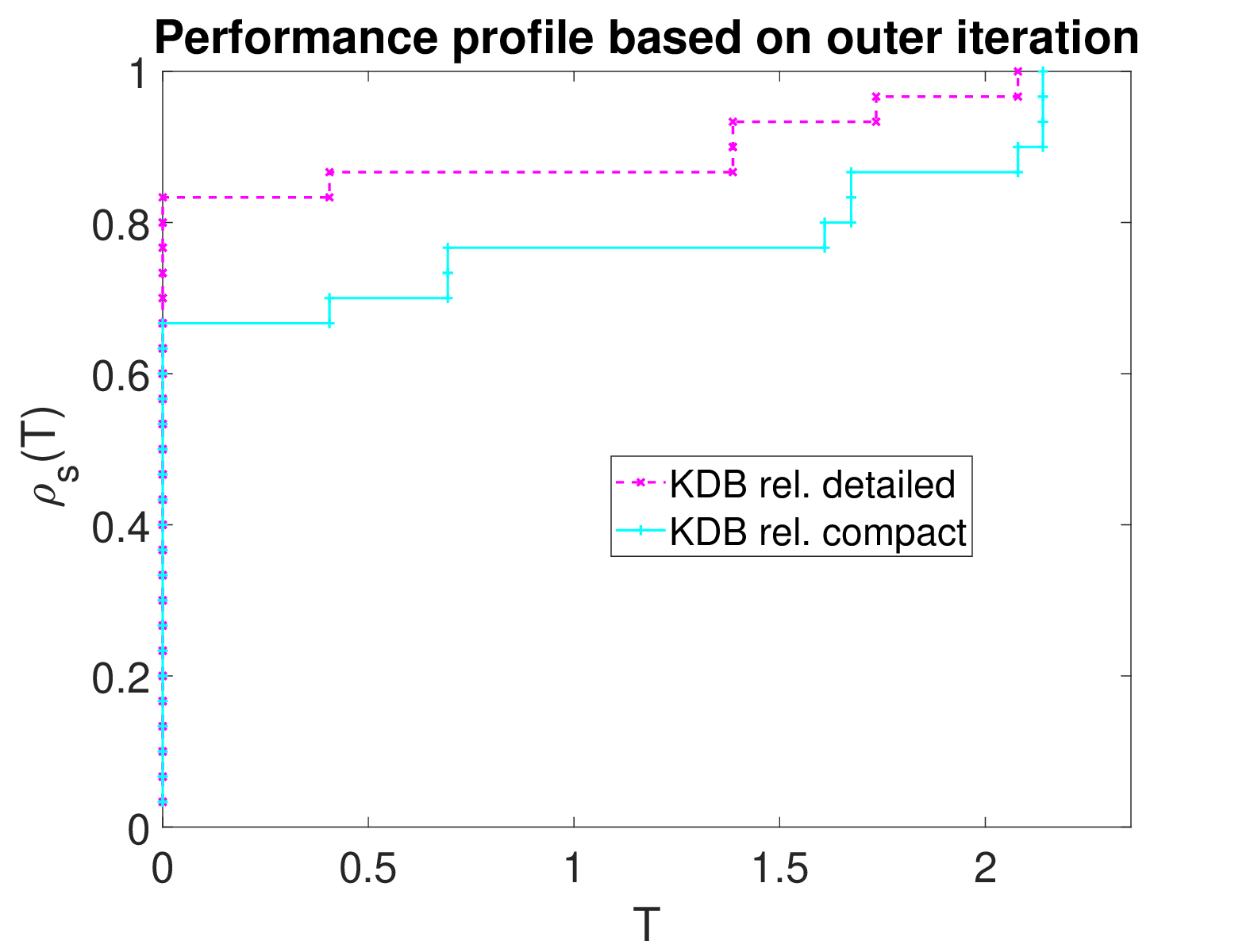}
     \end{subfigure}
 \hfill
     \begin{subfigure} 
         \centering
         \includegraphics[width=0.3\textwidth]{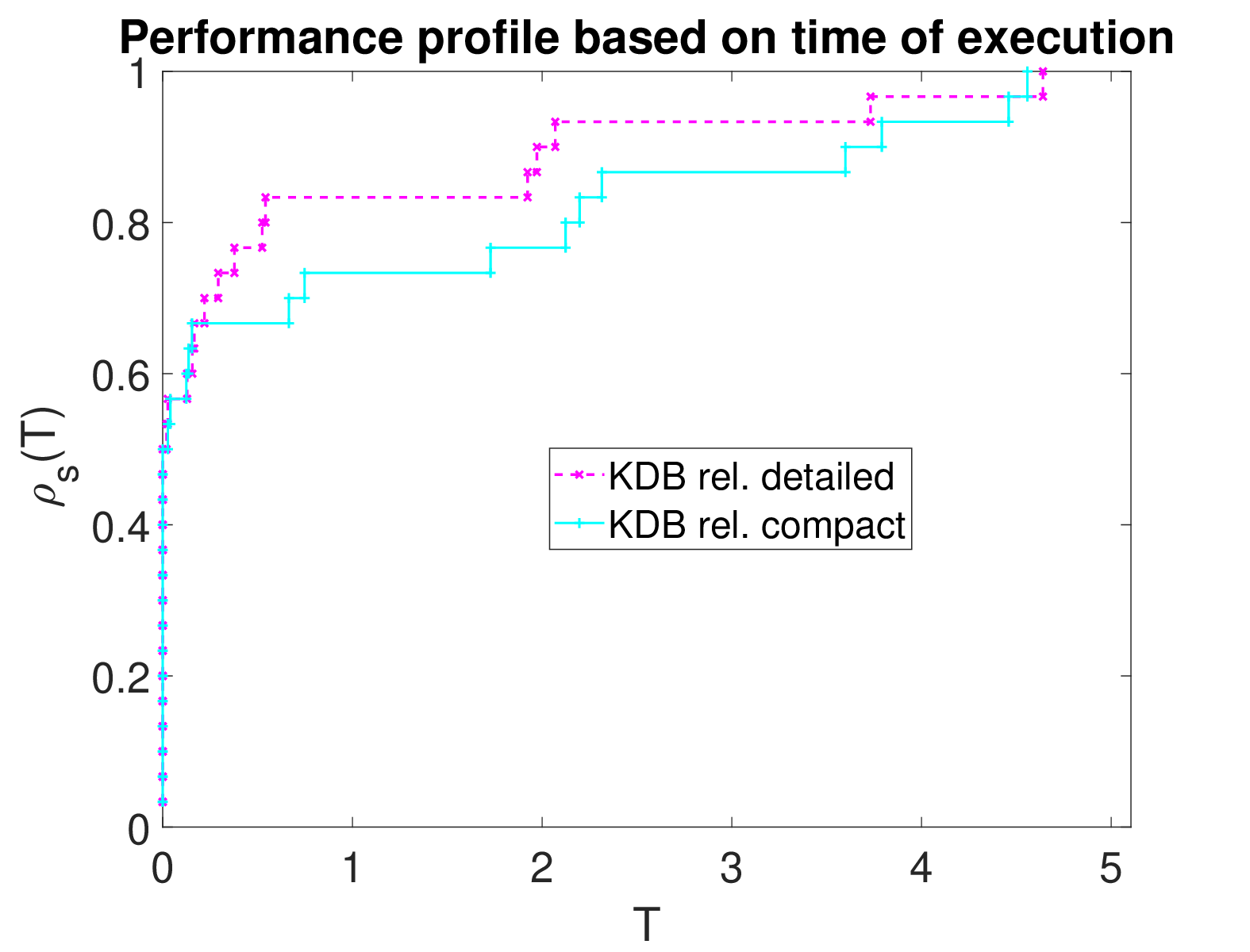}
     \end{subfigure} 
  \hfill
     \begin{subfigure} 
         \centering
         \includegraphics[width=0.3\textwidth]{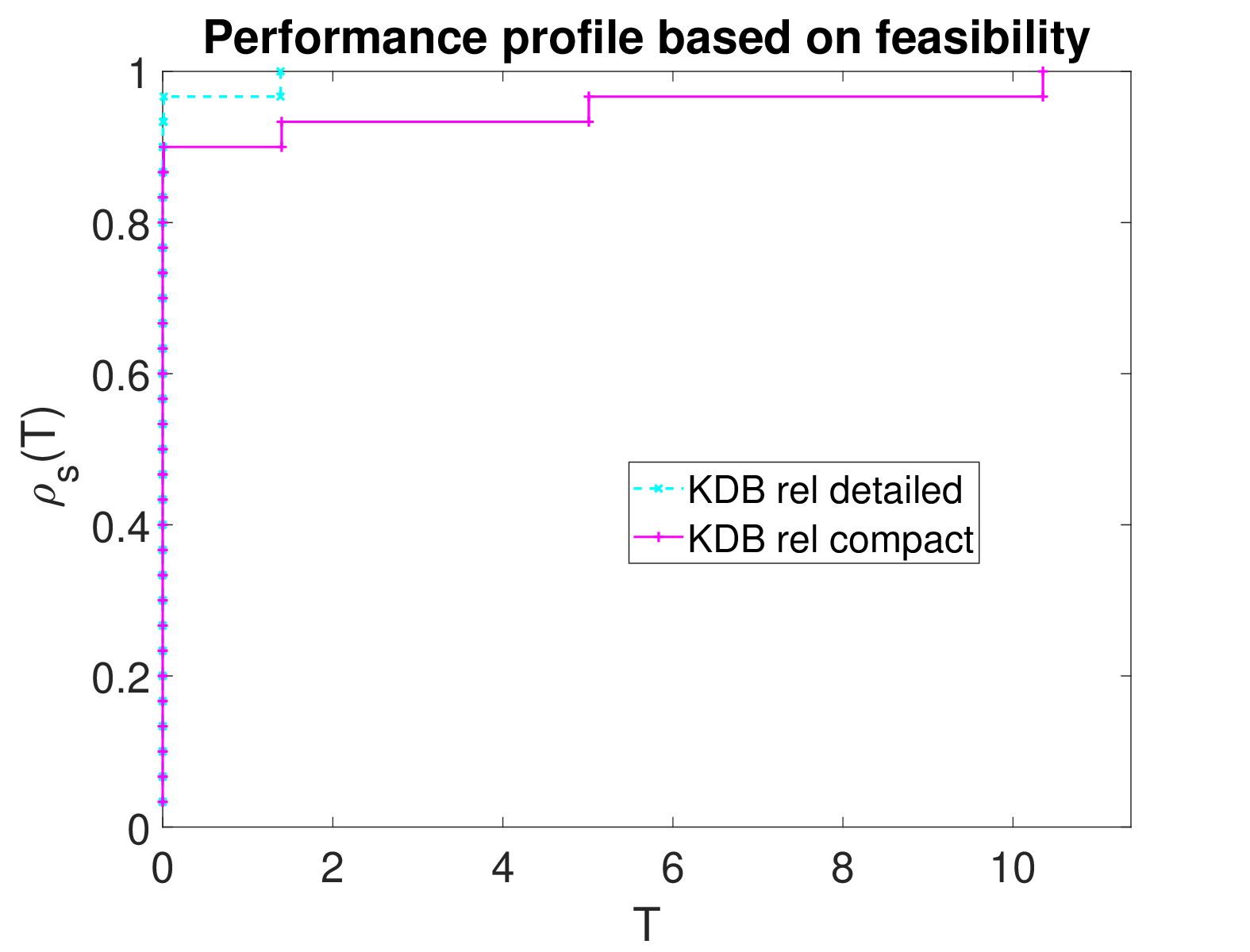}
     \end{subfigure}   
        \centering
    \begin{subfigure} 
  \centering
 \includegraphics[width=0.3\textwidth]{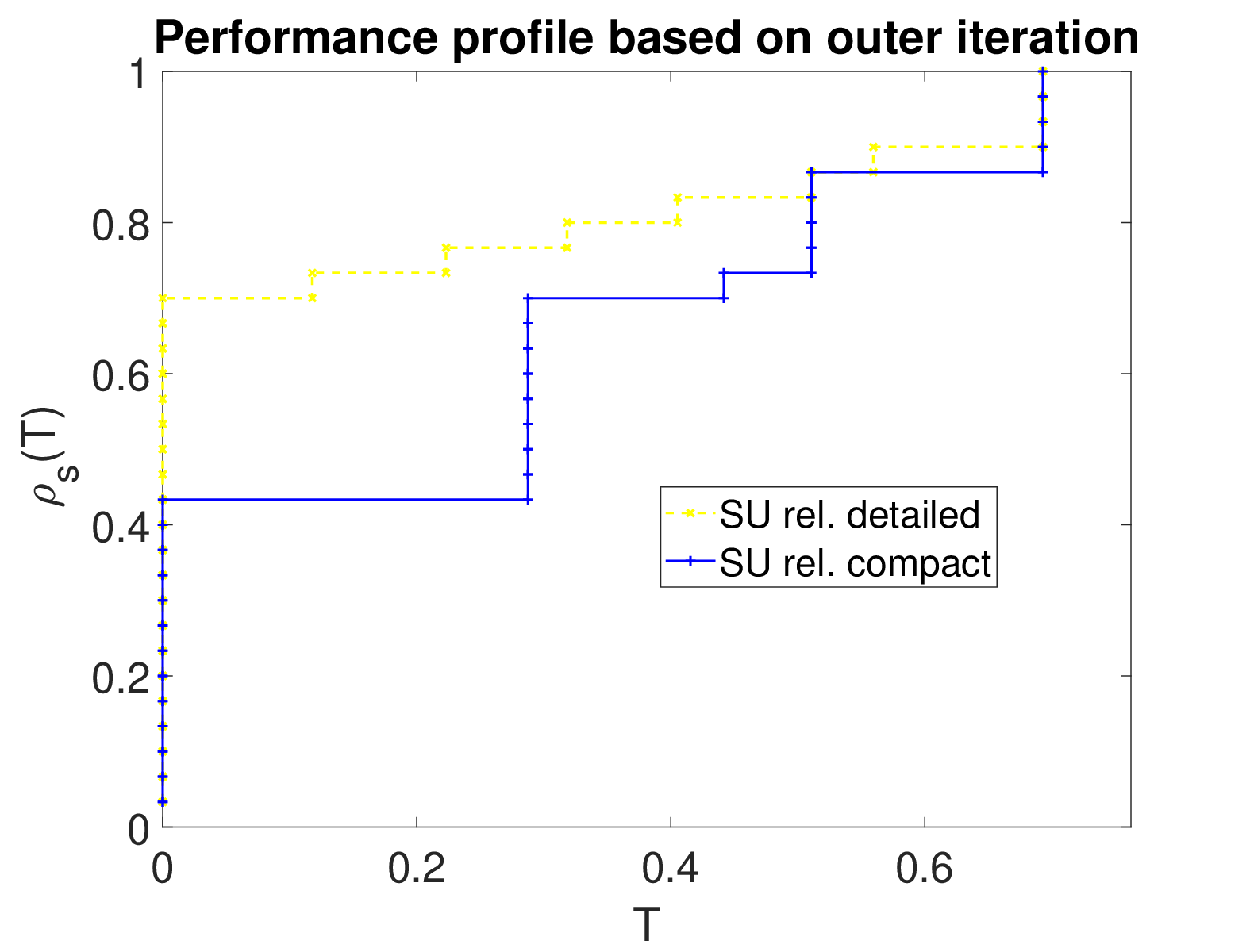}
     \end{subfigure}
 \hfill
     \begin{subfigure} 
         \centering
         \includegraphics[width=0.3\textwidth]{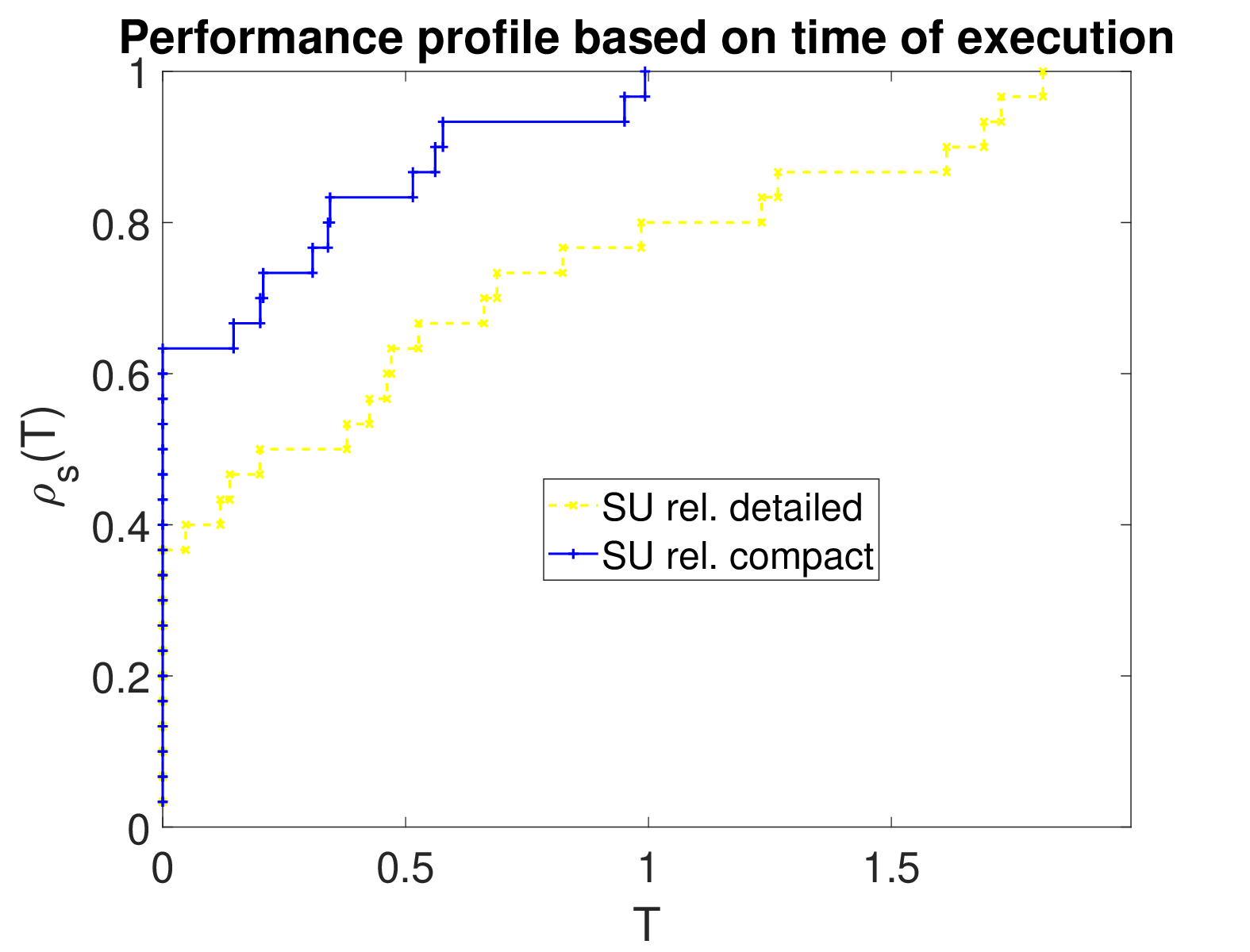}
     \end{subfigure} 
  \hfill
     \begin{subfigure} 
         \centering
         \includegraphics[width=0.3\textwidth]{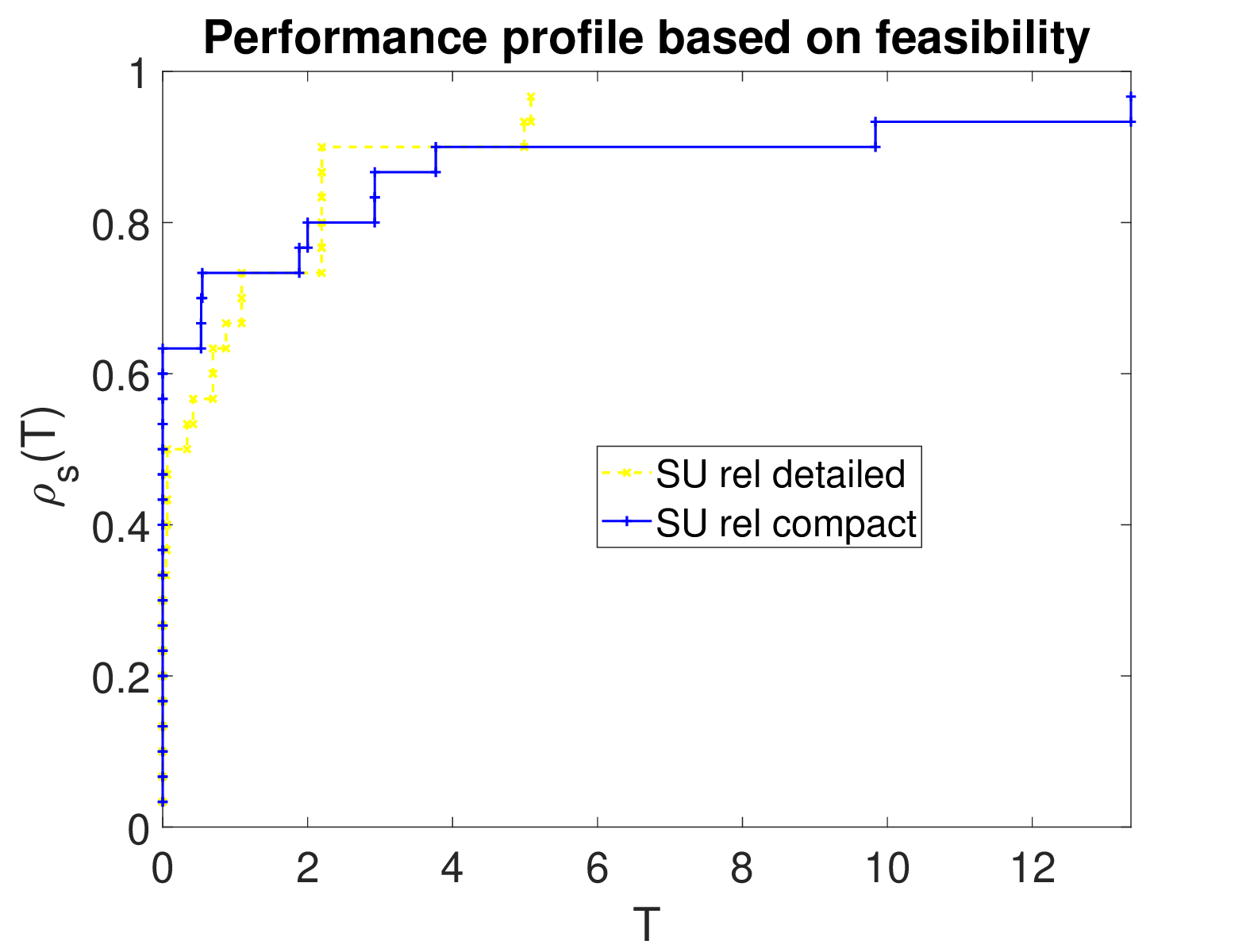}
     \end{subfigure}   
       \centering
    \begin{subfigure} 
  \centering
 \includegraphics[width=0.3\textwidth]{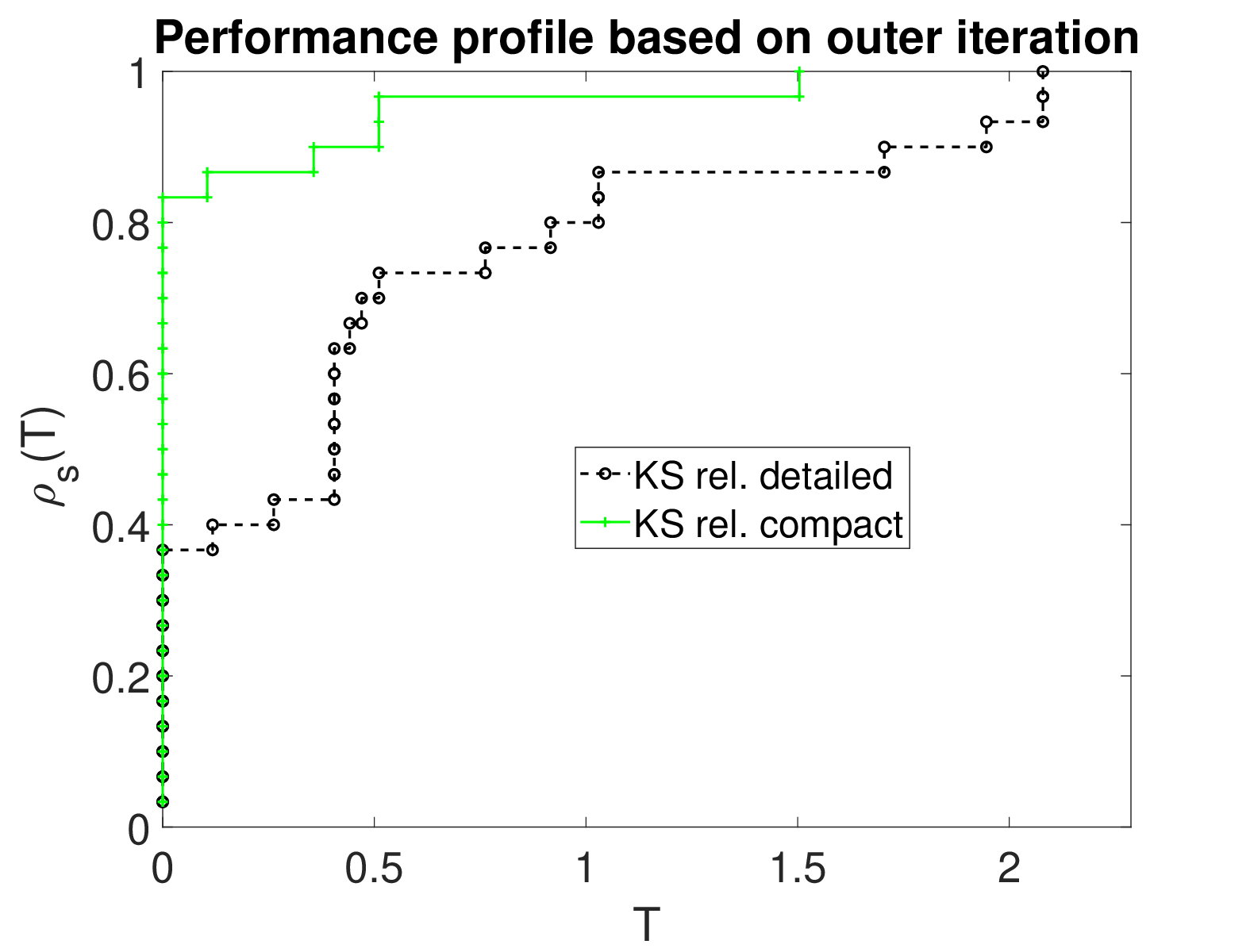}
     \end{subfigure}
 \hfill
     \begin{subfigure} 
         \centering
         \includegraphics[width=0.3\textwidth]{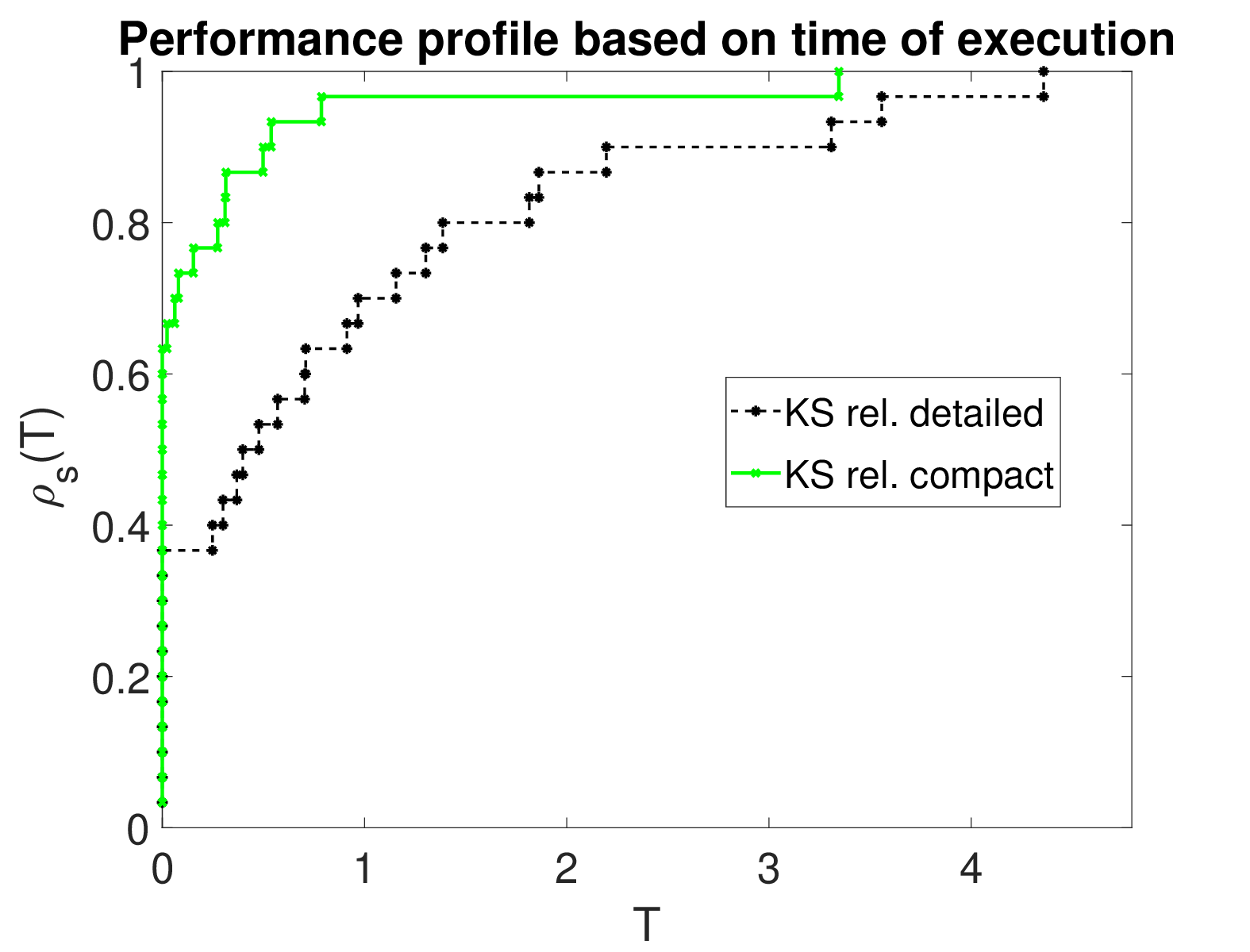}
     \end{subfigure} 
  \hfill
     \begin{subfigure} 
         \centering
         \includegraphics[width=0.3\textwidth]{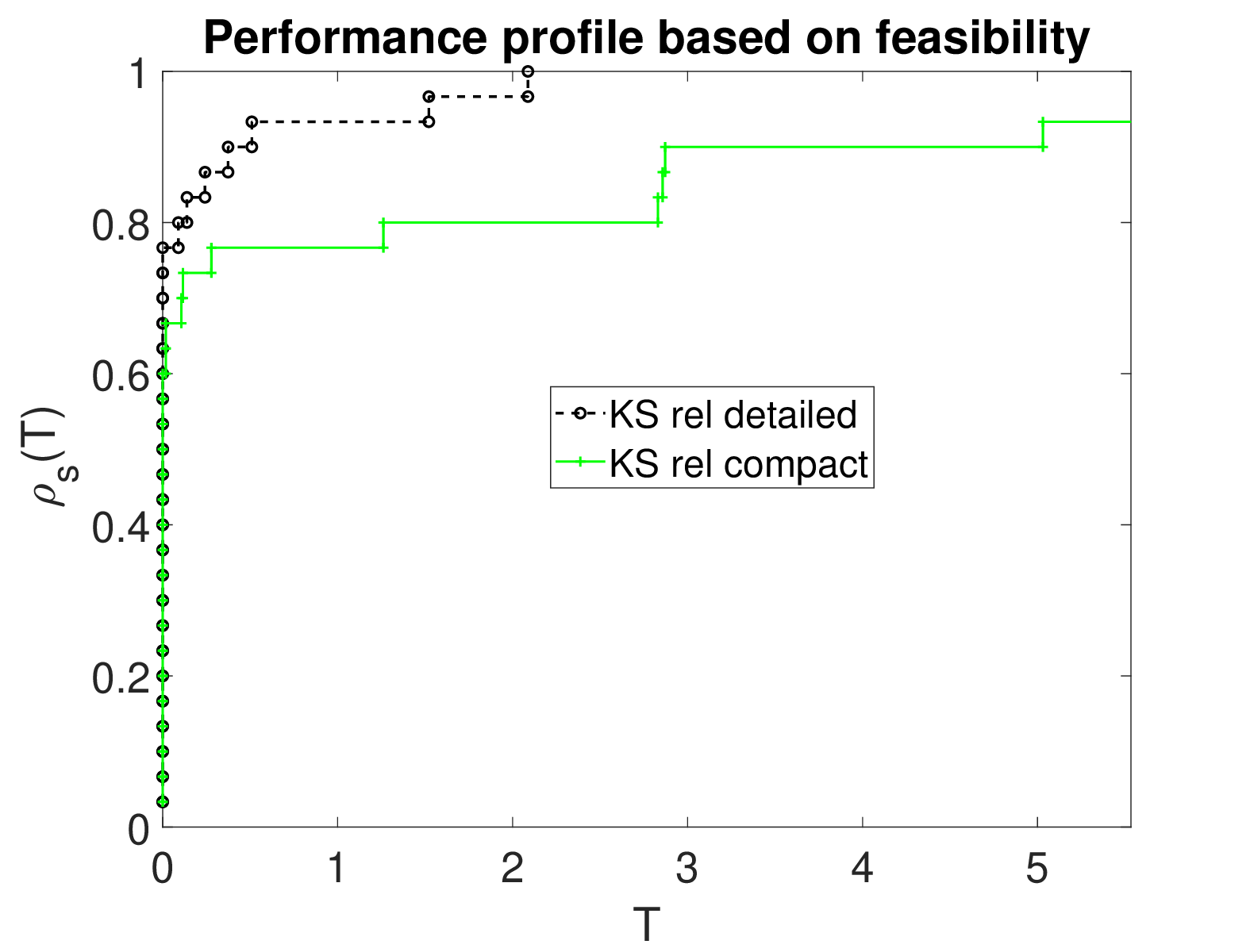}
     \end{subfigure}
        \caption{Comparison of the detailed and compact form of the relaxation methods with examples satisfying convexity and regularity conditions. From Top: Scholtes rel(axation), KDB rel(axation), SU rel(axation) and KS methods rel(axation).} 
       \label{figure1}   
\end{figure}

We use examples of pessimistic bilevel programs for which the optimal solutions have been reported in the literature; see \cite{Mitsos2006,WiesemannTsoukalasKleniatiRustem2013}. These problems are organized in two categories: the ones with convex (labeled as {mb\_1\_1\_06}, {mb\_1\_1\_10}, and {mb\_1\_1\_17}) and the those with nonconvex lower-level problems (i.e., precisely, {mb\_1\_1\_03}, {mb\_1\_1\_04}, {mb\_1\_1\_05}, {mb\_1\_1\_07}, {mb\_1\_1\_08}, {mb\_1\_1\_09}, {mb\_1\_1\_11}, {mb\_1\_1\_12}, {mb\_1\_1\_13}, and {mb\_1\_1\_14}). Subsequently, the evaluation of the performance analysis of the relaxation methods is split along the lines of these two categories, as well as in terms the \textit{detailed form} and \textit{compact form} (see \cite{BNZ} for more details on these forms) of each relaxation, in order to ascertain their impact on the results. It is important to recall that the compact form involves a reduction in the number of variables and the dimension of the system \eqref{EquFinal}. 
\begin{figure}
     \centering
     \begin{subfigure} 
         \centering
         \includegraphics[width=0.3\textwidth]{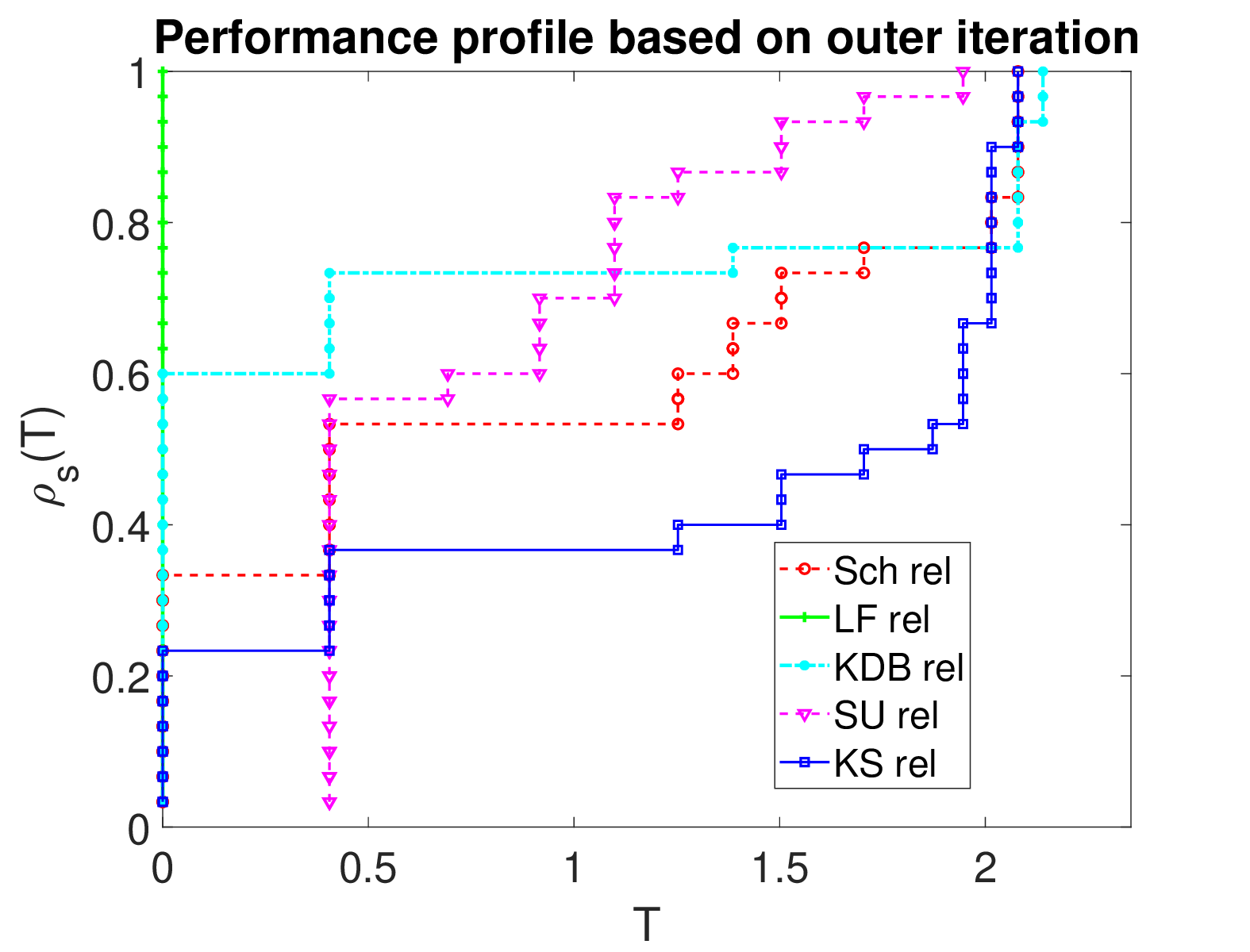}
     \end{subfigure}
     \hfill
     \begin{subfigure} 
         \centering
         \includegraphics[width=0.3\textwidth]{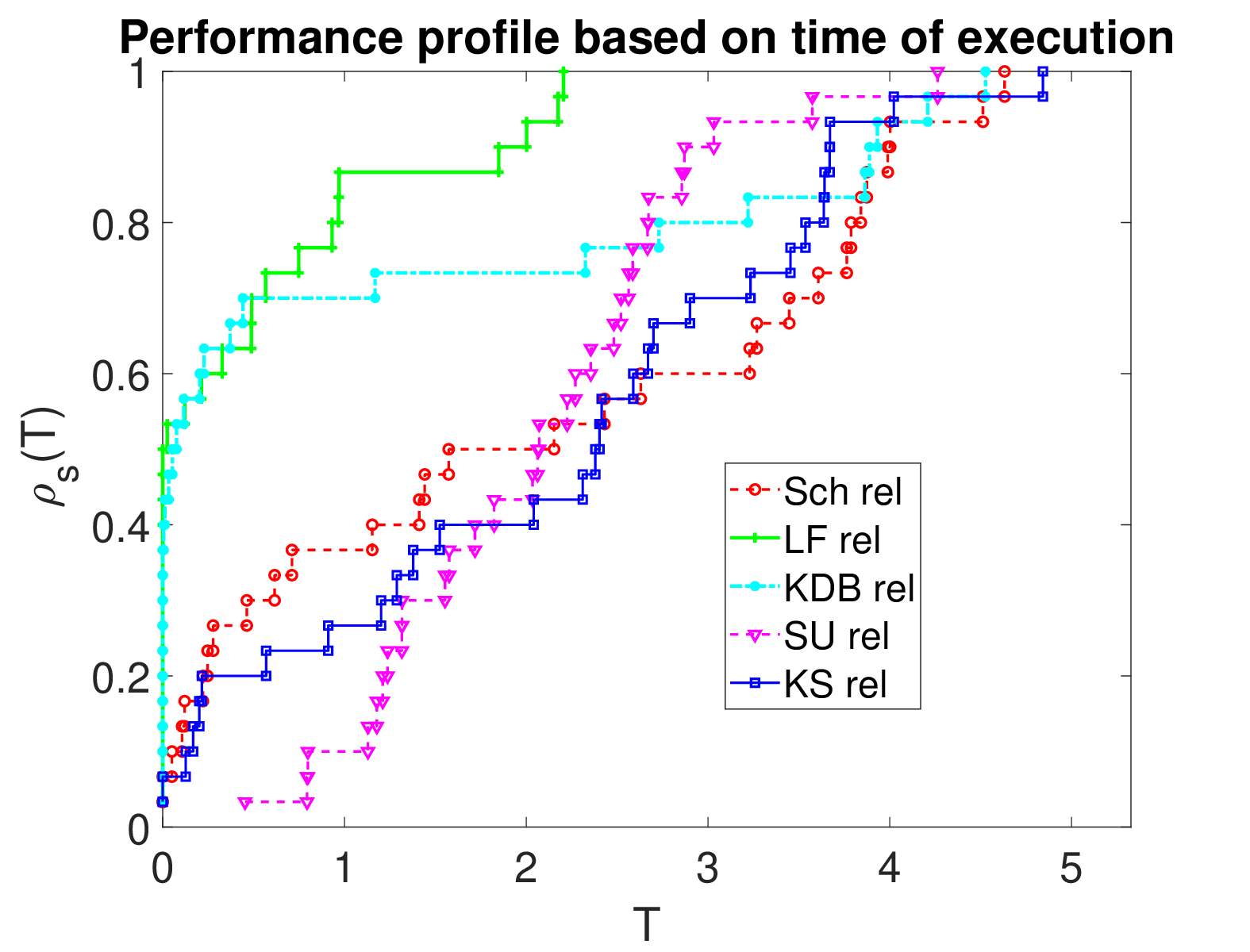}
     \end{subfigure}
     \hfill
          \begin{subfigure} 
         \centering
         \includegraphics[width=0.3\textwidth]{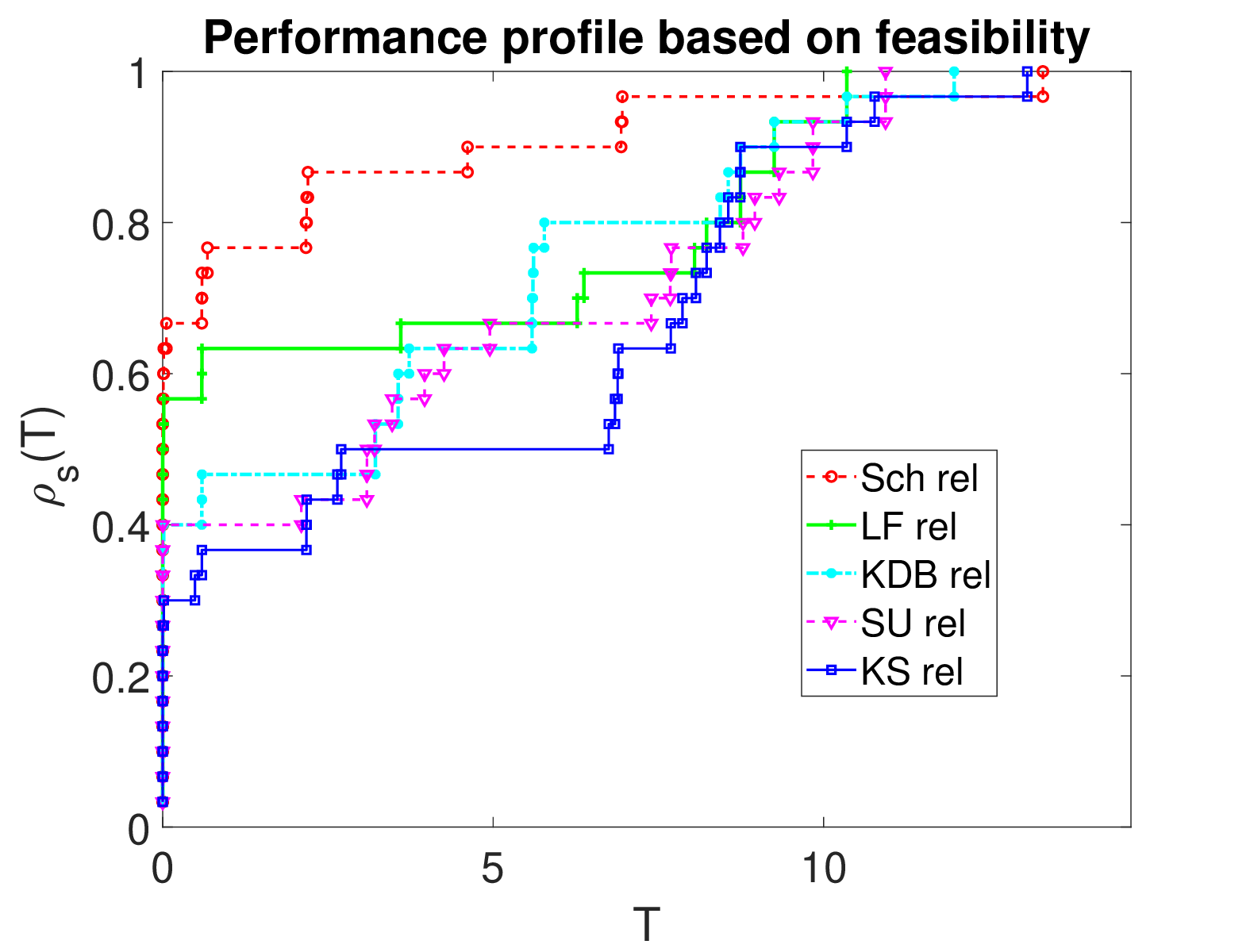}
     \end{subfigure}
          \centering
     \begin{subfigure} 
         \centering
         \includegraphics[width=0.3\textwidth]{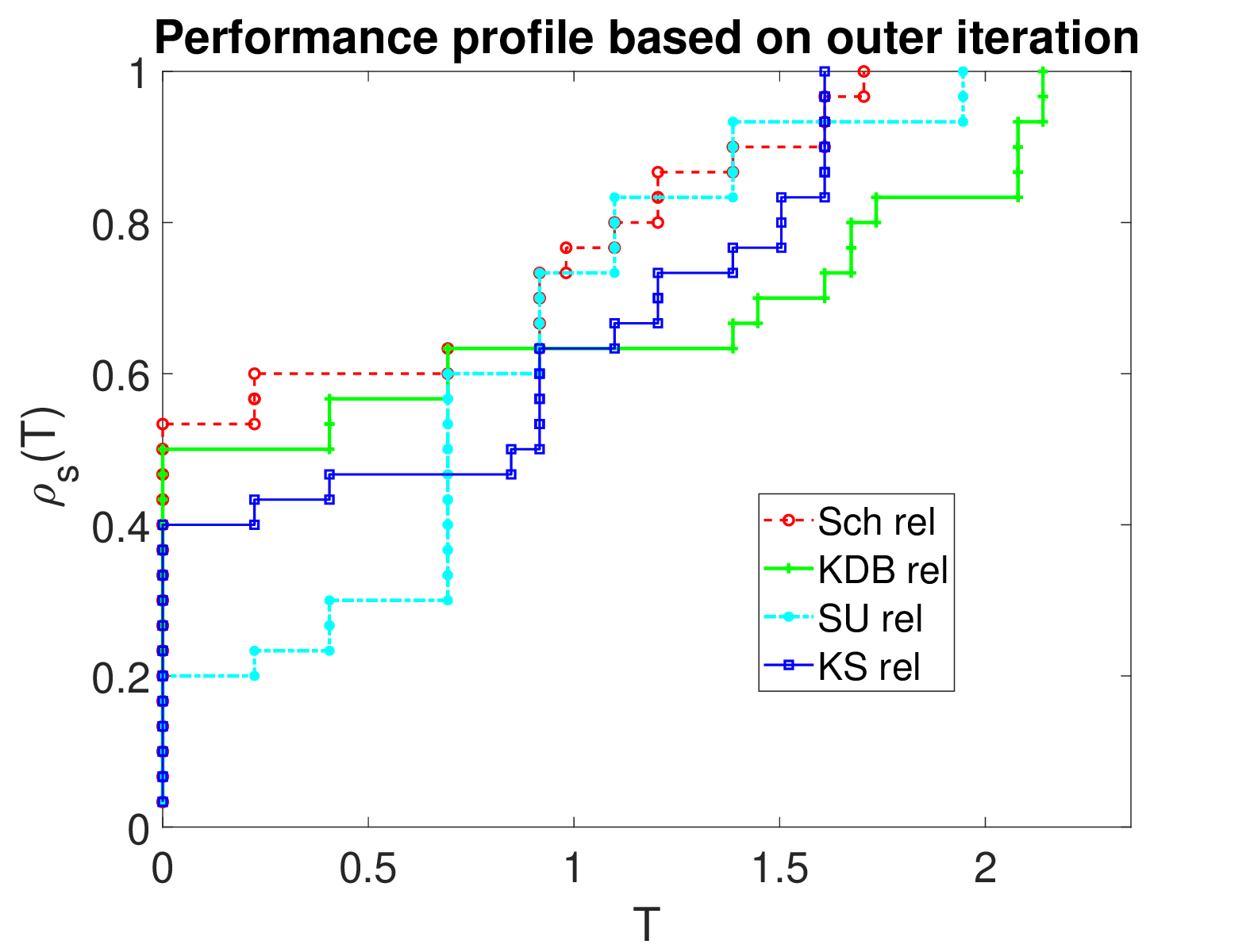}
     \end{subfigure}
     \hfill
     \begin{subfigure} 
         \centering
         \includegraphics[width=0.3\textwidth]{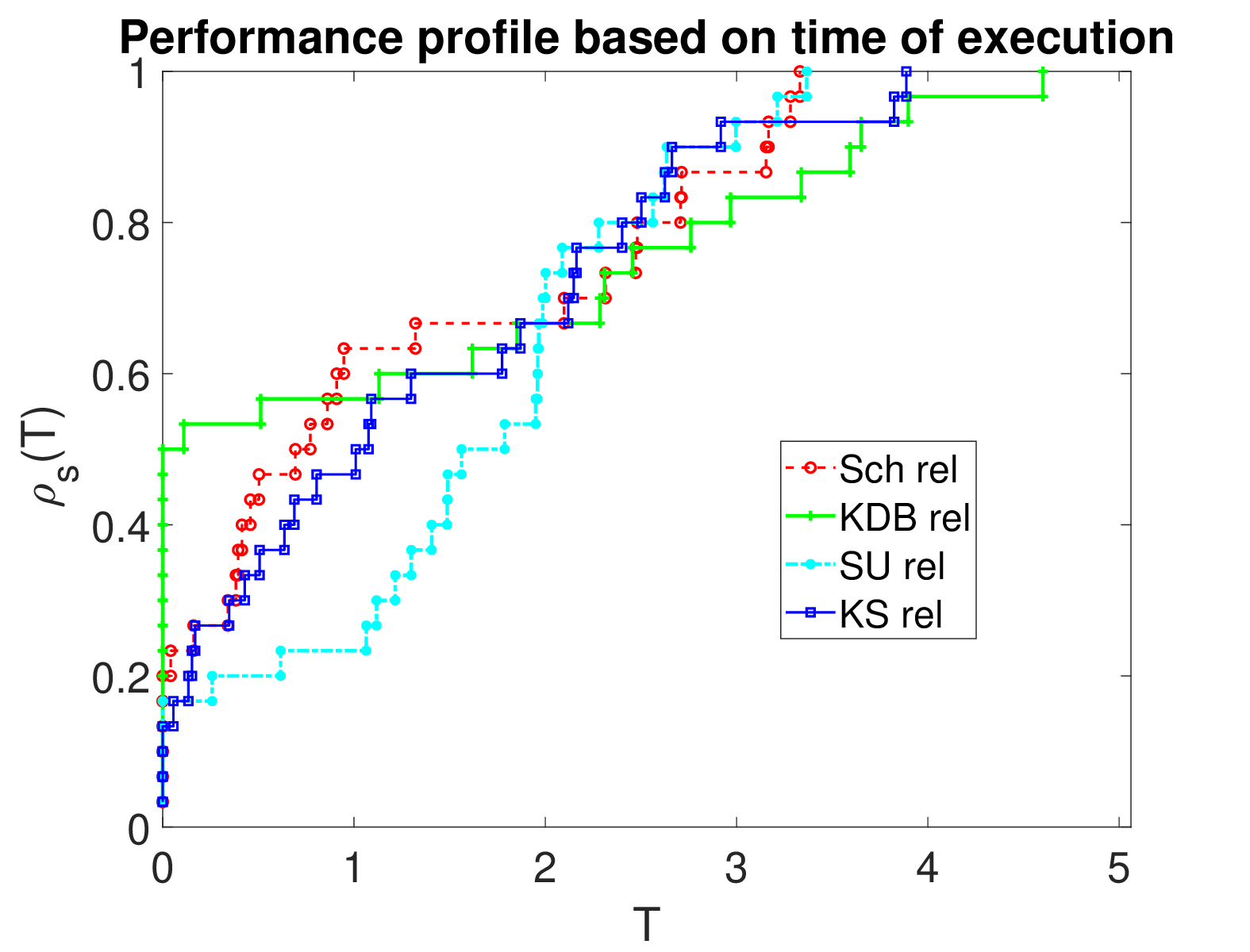}
     \end{subfigure}
     \hfill
     \begin{subfigure} 
         \centering
         \includegraphics[width=0.3\textwidth]{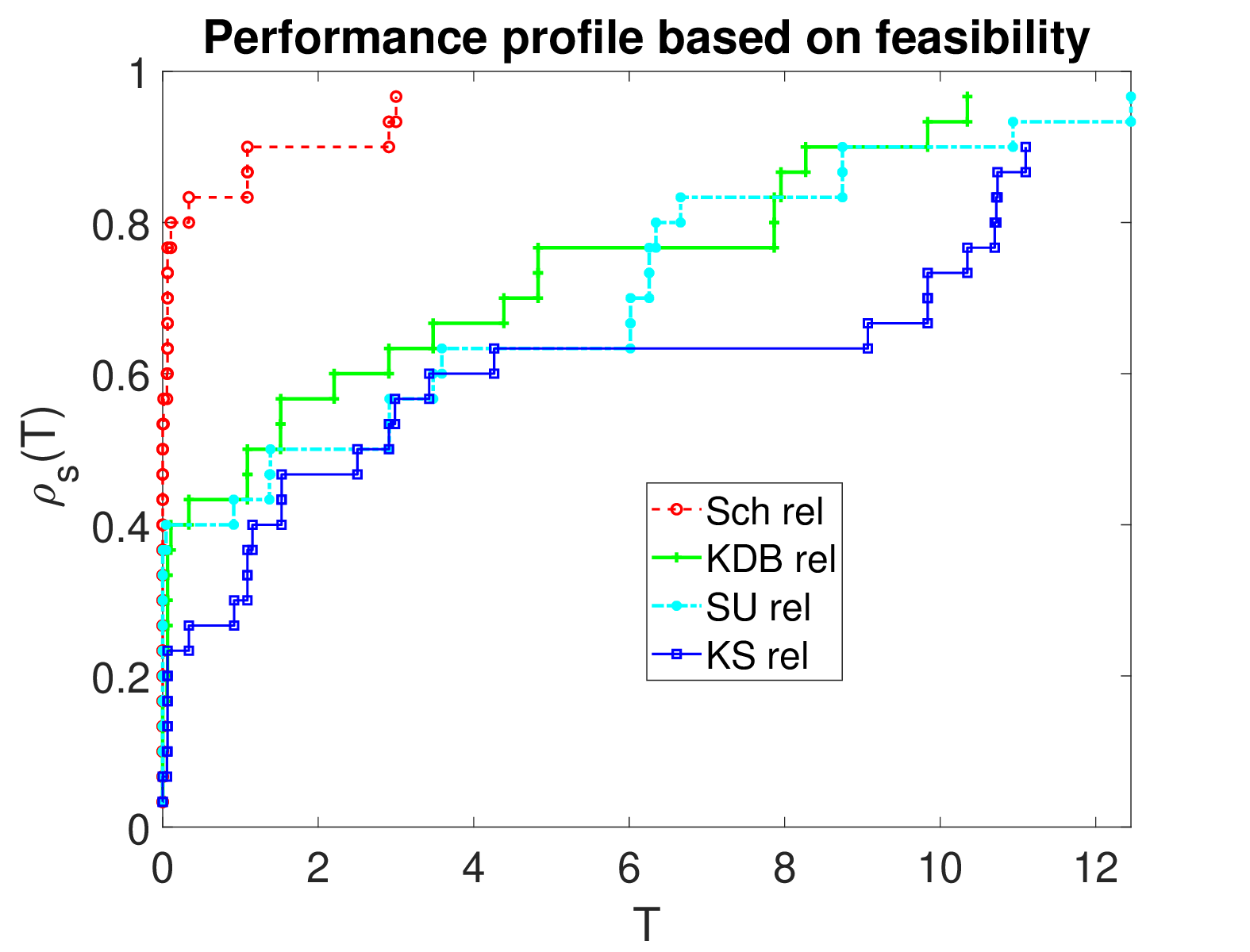}
     \end{subfigure}
      \caption{Comparison of all the relaxation methods on examples satisfying convexity and regularity conditions.} 
       \label{figure2} 
\end{figure}

We conduct a comparison of method performance across all relaxation methods for both detailed and compact systems. It is worth noting that the LF method lacks a compact version. For each example, we employ 10 random starting points, resulting in 30 instances in the first experiment and 100 instances in the second experiment for the comparative analysis. The performance of the methods is then evaluated based on several metrics, including the \textit{number of outer iterations} (i.e., not counting the iterations necessary to   {solve} the subproblem in Step 1 of Algorithm \ref{algorithm1-S}), execution time, the count of inner iterations (representing the number of iterations to solve the system \eqref{EquFinal} for fixed $\epsilon >0$ and $t>0$), and the experimental order of convergence (EOC), which is defined by 
\[
\text{EOC}  := \max \left\{ \frac{\log\|\Psi^{\epsilon,t}_{\mathcal{R}}(\zeta^{K-1})\|}{\log\|\Psi^{\epsilon,t}_{\mathcal{R}}(\zeta^{K-2})\|}, \frac{\log\|\Psi^{\epsilon,t}_{\mathcal{R}}(\zeta^{K})\|}{\log\|\Psi^{\epsilon,t}_{\mathcal{R}}(\zeta^{K-1})\|} \right\}.
\]
We also assess the number of instances in which the algorithm achieves a C-stationary point, along with the accuracy of the obtained solution. The accuracy is  determined by calculating $|F_{pes}(x, y) - F(x^*, y^*)|$, where $F_{pes}(x, y)$ represents the upper-level objective value of the optimal solution to the pessimistic bilevel problem reported in \cite{WiesemannTsoukalasKleniatiRustem2013}, and $F(x^*, y^*)$ corresponds to the optimal solution resulting from the algorithm. 

To verify the feasibility of a point obtained from Algorithm \ref{algorithm1-S}, we examine whether the solution satisfies the conditions $u_j \geq 0$ and $-g_j(x, y) \geq 0$ for $j=1, \dots, q$, using a tolerance level of $10^{-4}$, taking into account that $\mathcal{L}(x, y, u)$ is already involved  in the optimality condition from \eqref{KKT system}. The primary stopping criterion for the algorithm is $\|\Phi^{t_k}_{\mathcal{R}}(\zeta^k)\| < \varepsilon$ in the context of Step 1. 

\begin{table}[h]
   \centering
    \begin{tabular}{ccccccccccc}
         \toprule
          & \multicolumn{2}{c}{Scholtes} 
          & LF 
          & \multicolumn{2}{c}{KDB}
          & \multicolumn{2}{c}{SU}
          & \multicolumn{2}{c}{KS} \\
          \cline{2-3} \cline{5-6} \cline{7-8} \cline{9-10} \\
          & Detail & Comp. & Detail & Detail & Comp. & Detail & Comp. & Detail & Comp. \\
   \midrule 
   Average outer iter & 6.9 & 4.7 & 2.0 & 5.67 & 7.5 & 4.8 & 5.23 & 9.47 & 5.73 \\
   Average time & 0.64 & 0.22 & 0.08 & 0.39 & 0.41 & 0.25 & 0.15 & 0.60 & 0.19  \\
   Average inner iter & 505.23 & 189.27 & 64.4 & 291.73 & 319.63 & 172.83 & 101.7 & 463.5 & 129.43 \\
   Average accuracy & 0.40 & 0.43 & 0.46 & 0.55 & 0.57 & 0.82 & 0.69 & 0.71 & 1.34\\
   C-stationarity & 6 & 13 & 0 & 0 & 0 & 0 & 0 & 0 & 0 \\
   Feasibility (\%) & 43.3 & 83.3 & 83.3& 66.7 & 73.3 & 70& 6& 43.3 & 59.7 \\
   EOC $\leq$ 1 & 22 & 29 & 30 & 20& 20& 23 & 22& 26& 29 \\
   EOC $>$ 1& 8 & 1 & 0 & 10 & 10 & 7 & 8 & 4 & 1 \\
   \bottomrule
    \end{tabular}
    \caption{Numerical results for examples satisfying convexity and regularity conditions}
    \label{tab1}
\end{table}

However, an additional condition has been incorporated to serve as a safeguard for the algorithm when dealing with potential ill-behavior.
 In practice, we set $\varepsilon = 10^{-7}$ and stops the algorithm if
\[
\|\Psi^{\epsilon,t}_{\mathcal{R}}(\zeta^k)\| < \varepsilon \quad \mbox{or} \quad | \|\Psi^{\epsilon,t}_{\mathcal{R}}(\zeta^{k-1})\| - \|\Psi^{\epsilon,t}_{\mathcal{R}}(\zeta^{k})\| | < 10^{-9}.
\]
For initialization, we take $t^{0} = 10^{-3}$, $\theta = 0.05$, $u^0 = 1_q$, $\delta^0 = 1_p$, $\alpha^0 = 1_p$, $\beta^0 = 1_m$, $\mu^0 = 1_q$, while $x^0 \in \mathbb{R}^n$ and $y^0 \in \mathbb{R}^m$ are chosen randomly, as described above. We also take $\epsilon = 10^{-3}$ for the perturbation in \eqref{perb}.  
    
The experimental results are analyzed using the performance profile framework introduced in \cite{Dolan&More}. Let $t_{s,i} > 0$ denote the performance measure for solver $s \in \mathcal{S}$ when solving instance $i \in \mathcal{I}$. Here, $\mathcal{S}$ represents the set of solvers, and $\mathcal{I}$ denotes the set of problem instances corresponding to various starting points. The performance ratio for solver $s$ on instance $i$ is defined as 
\[
 r_{s,i} := \frac{t_{s,i}}{\min_{s' \in \mathcal{S}} t_{s',i}} \;\,\mbox{ for }\;  s \in \mathcal{S},\;\, i \in \mathcal{I}. 
\]
This ratio compares the performance of solver $s$ to the best performance achieved by any solver in $\mathcal{S}$ for the same problem instance $i$.
To evaluate the solver performance across all problem instances, we use the cumulative distribution function $\rho_s$ for solver $s \in \mathcal{S}$ that can be written as 
\[
\rho_s(T) := \frac{|\{i \in \mathcal{I} \mid r_{s,i} \leq T\}|}{|\mathcal{I}|} \;\,\mbox{ for }\;  T \in [1, \infty).
\]
This function $\rho_s(T)$ represents the proportion of instances where the performance of solver $s$ is within a factor $T$ of the best solver's performance on those instances.

The performance profile is a plot of $\rho_s(T)$ for all solvers $s \in \mathcal{S}$. The value $\rho_s(1)$ represents the percentage of problem instances where solver $s$ achieves the best performance, while for arbitrary $T \geq 1$, $\rho_s(T)$ indicates the percentage of instances where solver $s$ shows at most the $T$-fold of the best performance.

\begin{table}[h]
    \centering
    \begin{tabular}{cccccccccc}
         \toprule
          & \multicolumn{2}{c}{Scholtes} 
          & LF 
          & \multicolumn{2}{c}{KDB}
          & \multicolumn{2}{c}{SU}
          & \multicolumn{2}{c}{KS} \\
          \cline{2-3} \cline{5-6} \cline{7-8} \cline{9-10} \\
          & Detail & Comp. & Detail & Detail & Comp. & Detail & Comp. & Detail & Comp. \\
   \midrule 
   Average outer iter & 6.12 & 5.56 & 2.0 & 8.73 & 7.98 & 4.37 & 4.85 & 8.38 & 6.84 \\
   Average time & 0.56 & 0.49 & 0.11 & 1.16 & 0.69 & 0.25 & 0.19 & 0.63 & 0.62  \\
   Average inner iter & 370.9 & 368.4 & 78.28 & 835.65 &552.19 & 159.86 & 127.9 & 409.97 & 455.46 \\
   Average accuracy &  0.48 & 0.48 & 0.88 & 1.18 & 1.19 & 1.46 & 1.10 & 1.09 & 0.66\\
   C-stationarity & 8 & 40 & 0 & 0 & 0 & 0 & 0 & 0 & 0  \\
   Feasibility (\%) & 18 & 51 & 78 & 65 & 64 & 67 & 3 & 30 & 56.9 \\ 
   EOC $\leq$ 1 & 85 & 86 & 100 & 85 & 80 & 62 & 62 & 88 & 90 \\
   EOC $>$ 1& 15 & 14 & 0 & 15 & 20 & 38 & 38 & 12 & 10 \\
   \bottomrule
    \end{tabular}
    \caption{Numerical results for examples that do not satisfy convexity or regularity conditions}
    \label{tab2}
\end{table}
\begin{figure}[h]\label{fig1}
     \centering
     \begin{subfigure} 
         \centering
         \includegraphics[width=0.3\textwidth]{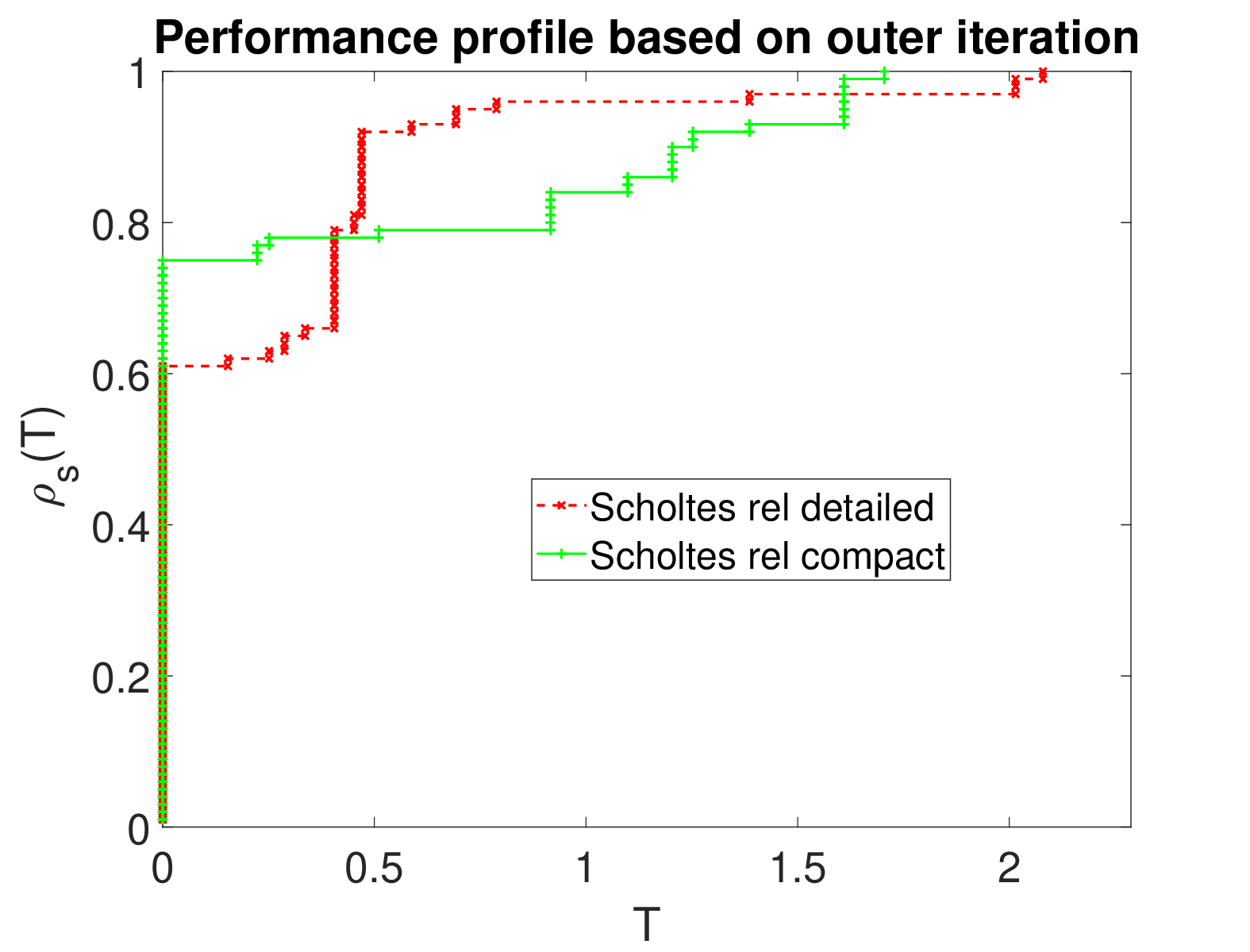}
     \end{subfigure}
\hfill
     \begin{subfigure} 
         \centering
         \includegraphics[width=0.3\textwidth]{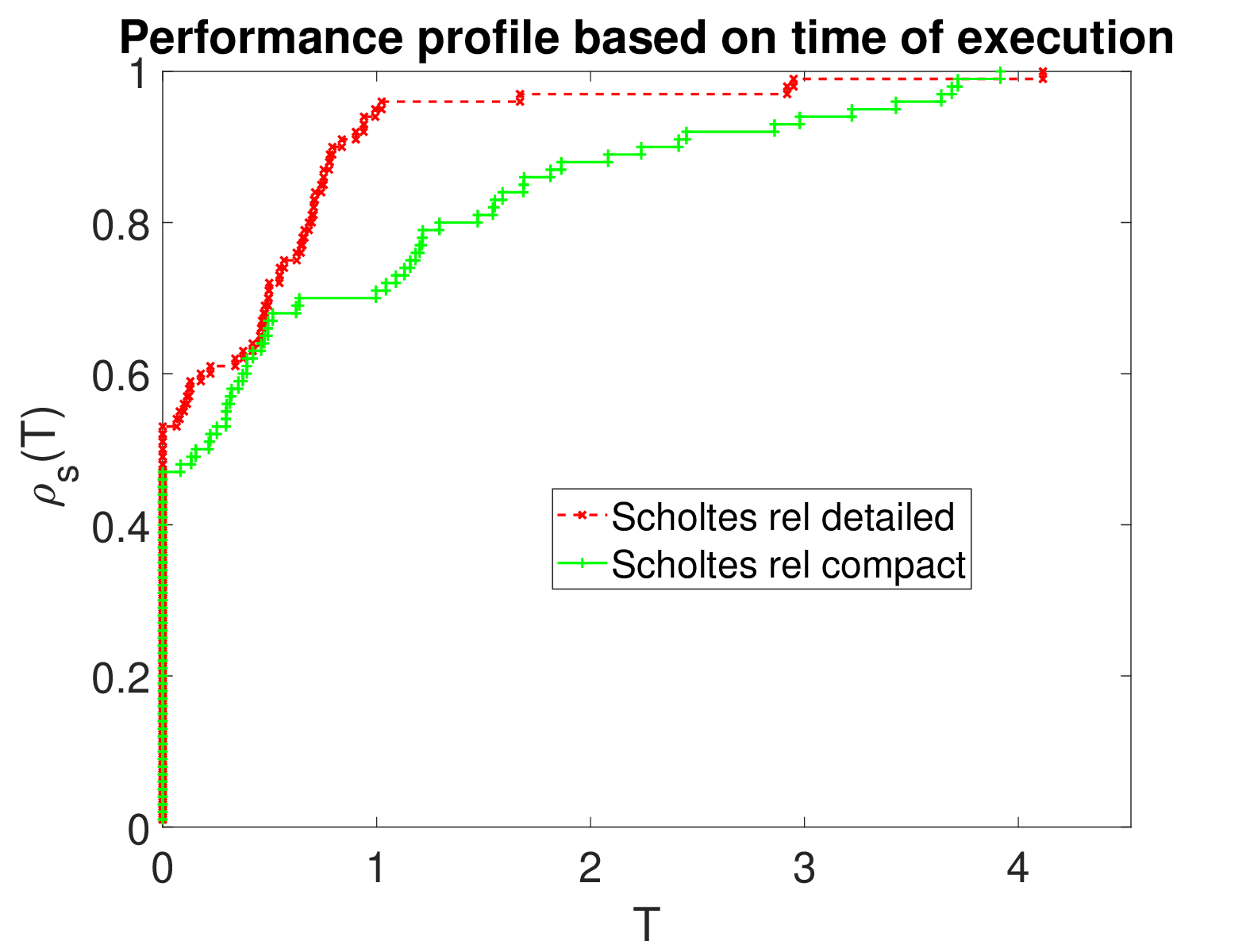}
     \end{subfigure}
  \hfill
     \begin{subfigure} 
         \centering
         \includegraphics[width=0.3\textwidth]{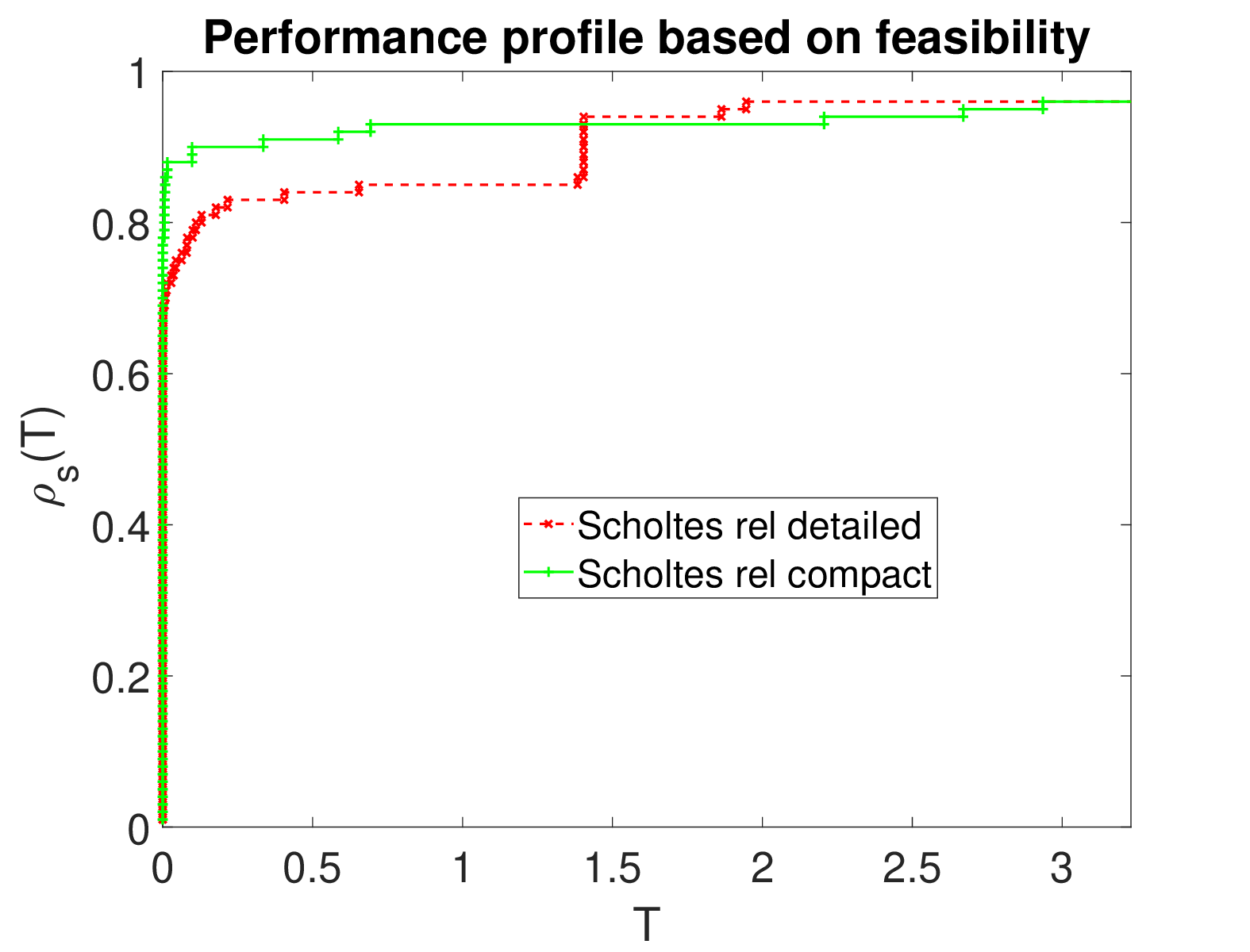}
     \end{subfigure}     
 \hfill     
    \begin{subfigure} 
         \centering
         \includegraphics[width=0.3\textwidth]{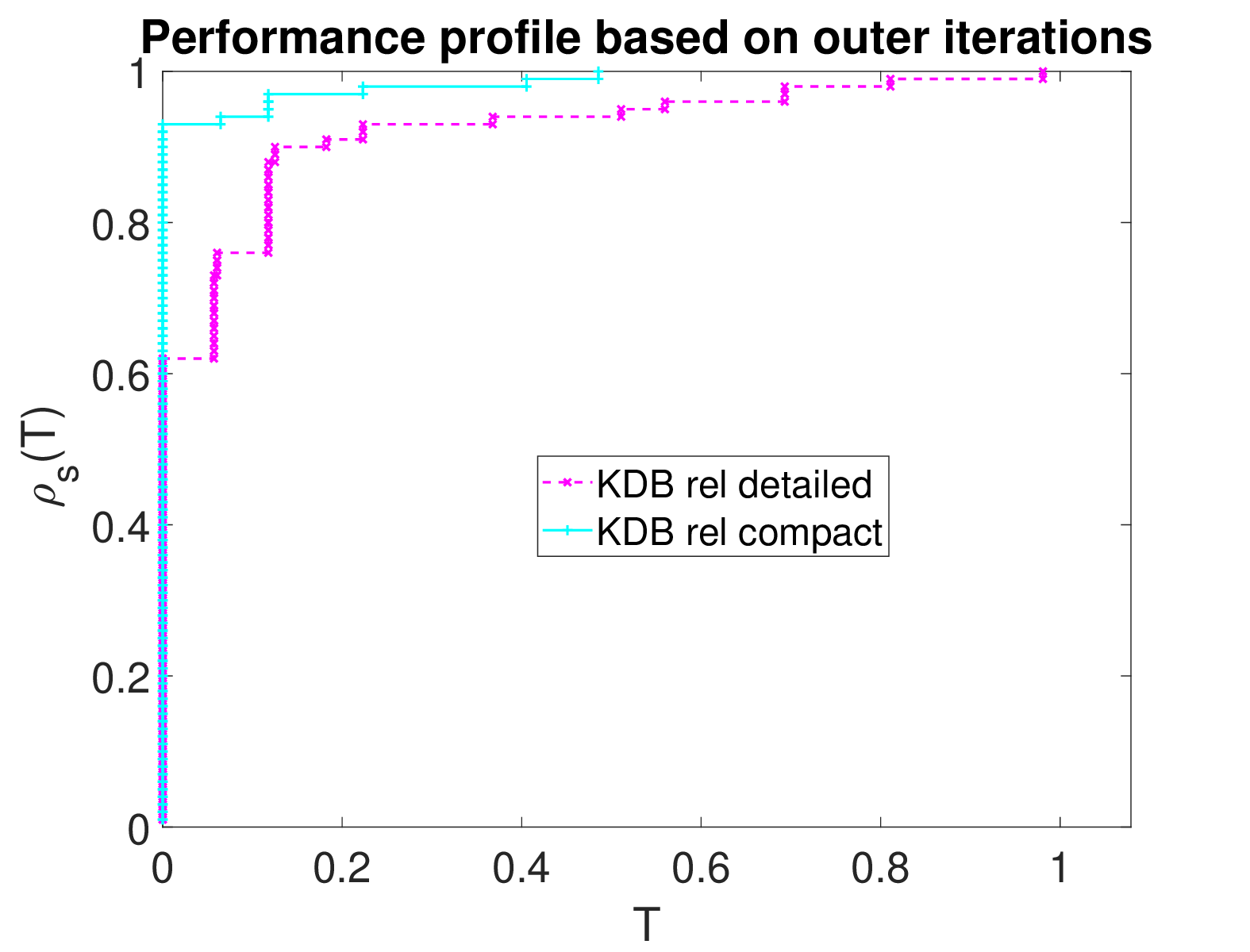}
     \end{subfigure}
\hfill
     \begin{subfigure} 
         \centering
         \includegraphics[width=0.3\textwidth]{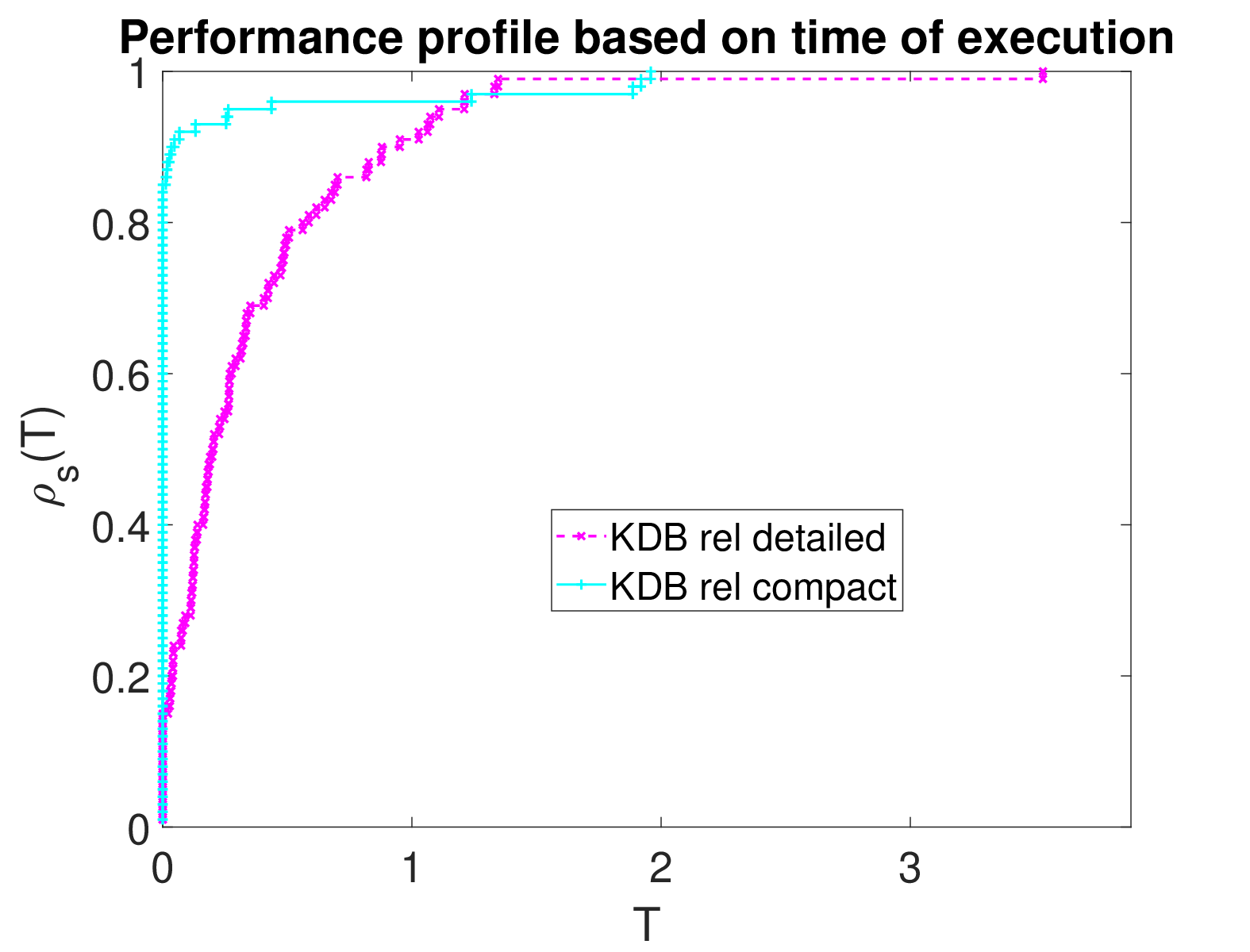}
     \end{subfigure}
  \hfill
     \begin{subfigure} 
         \centering
         \includegraphics[width=0.3\textwidth]{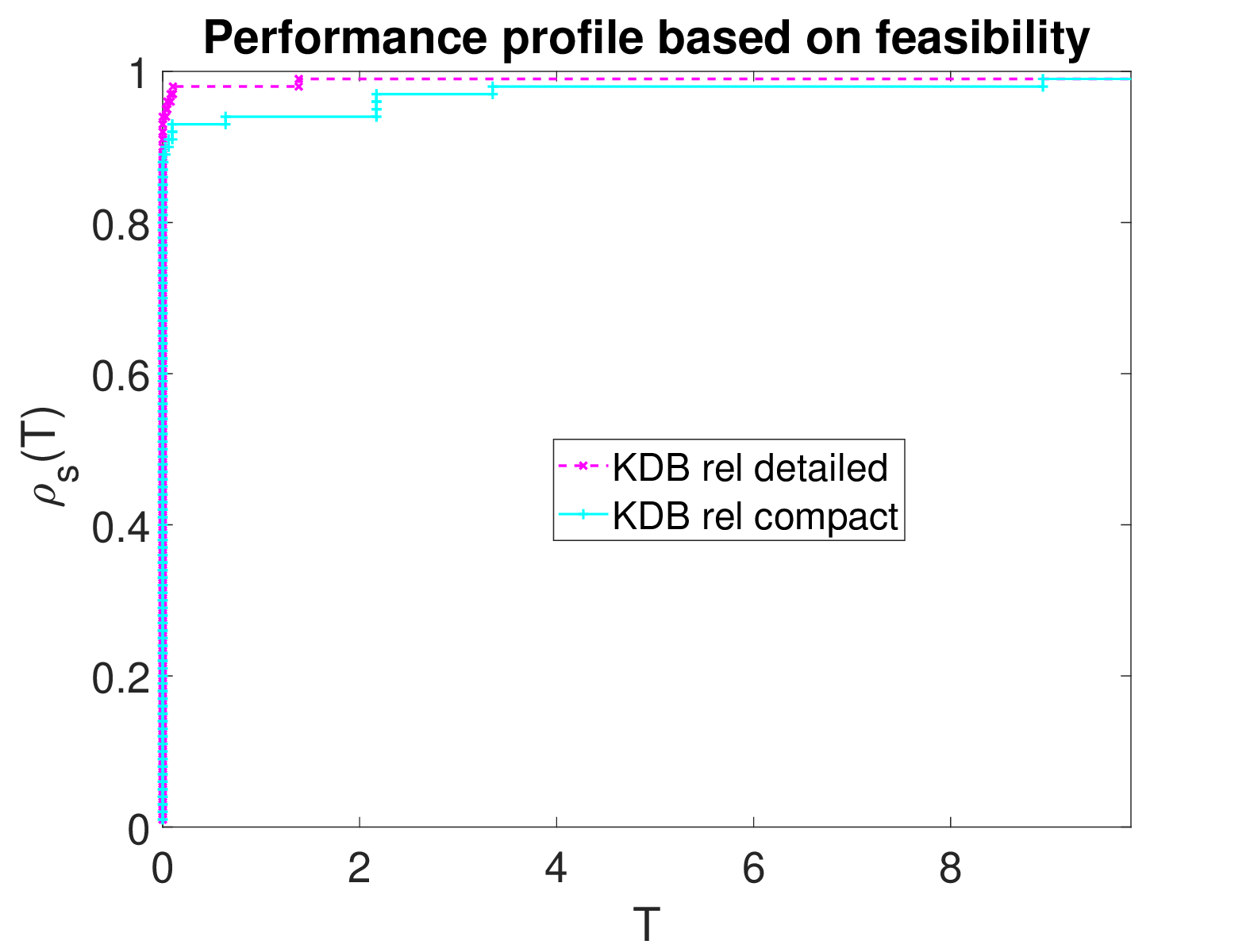}
     \end{subfigure}  
  \hfill    
       \begin{subfigure} 
         \centering
         \includegraphics[width=0.3\textwidth]{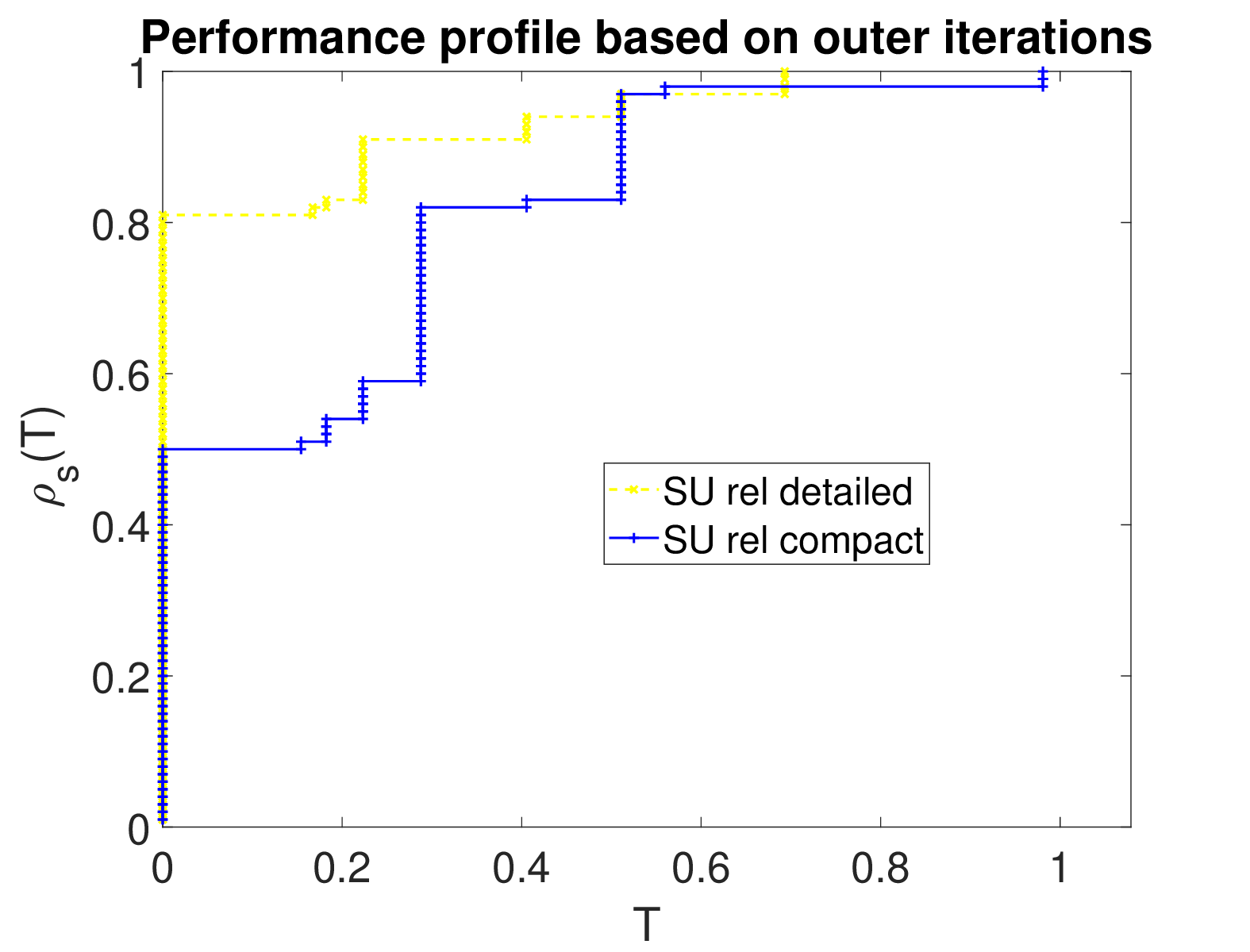}
     \end{subfigure}     
 \hfill     
    \begin{subfigure} 
         \centering
         \includegraphics[width=0.3\textwidth]{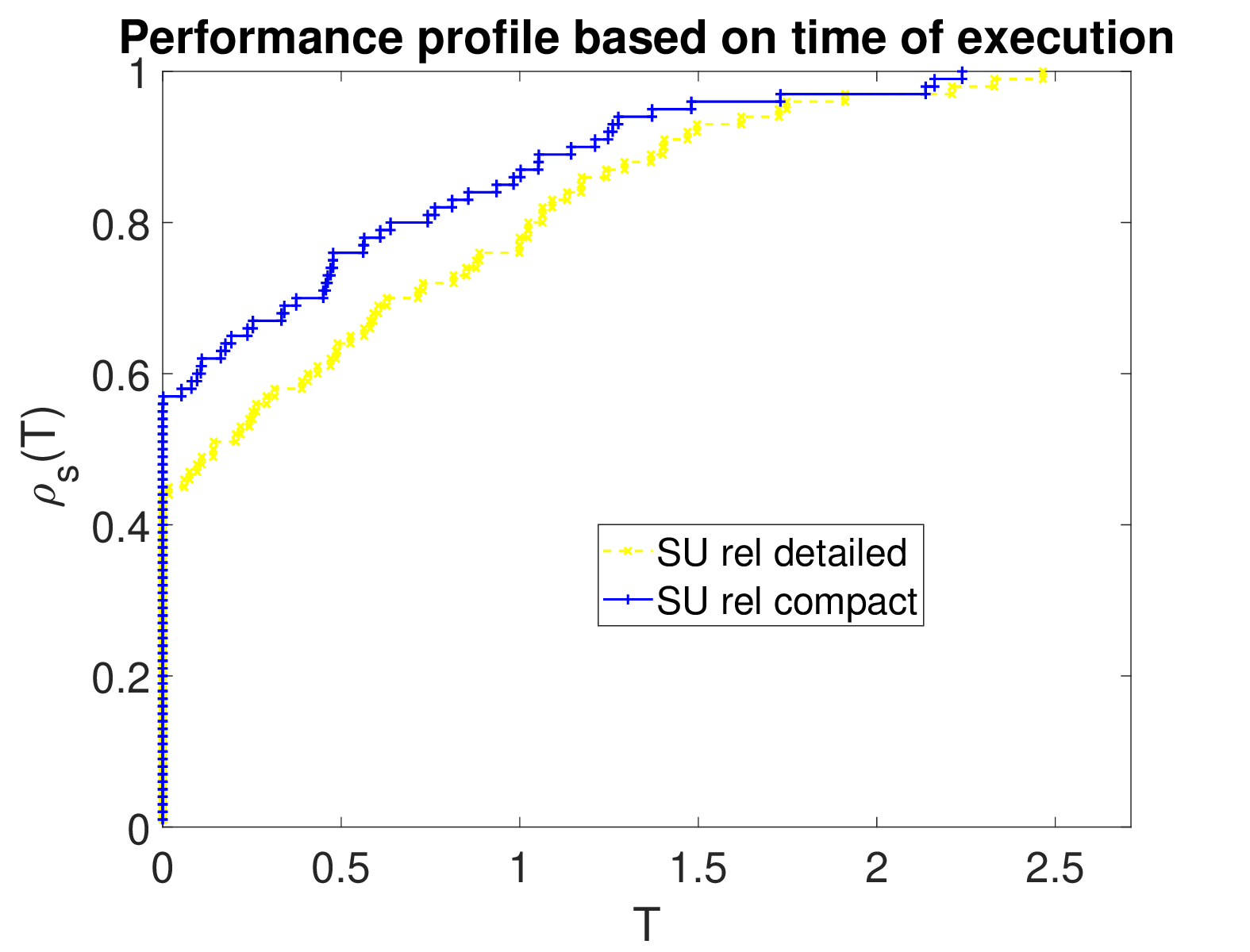}
     \end{subfigure}
\hfill
     \begin{subfigure} 
         \centering
         \includegraphics[width=0.3\textwidth]{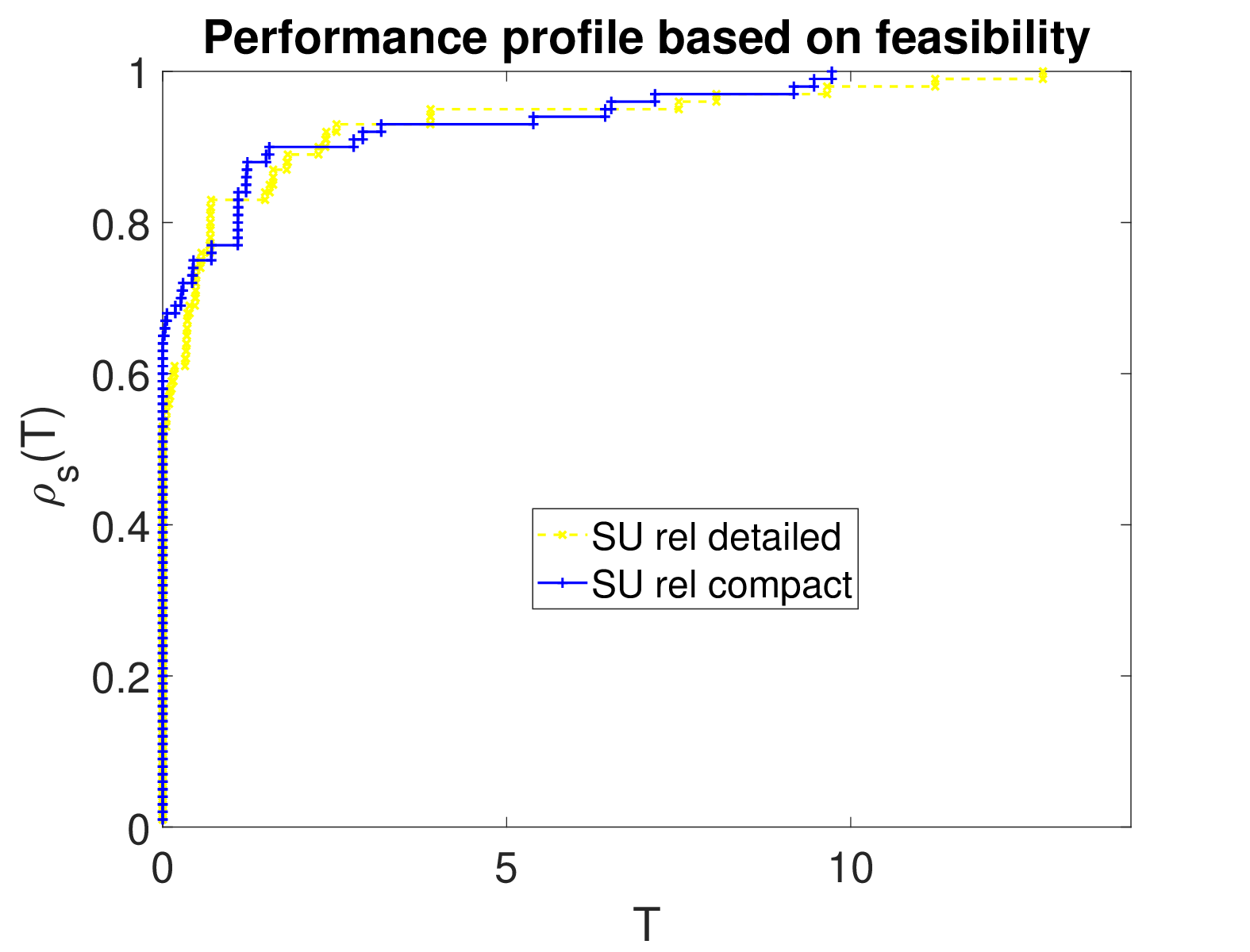}
     \end{subfigure}
  \hfill
     \begin{subfigure} 
         \centering
         \includegraphics[width=0.3\textwidth]{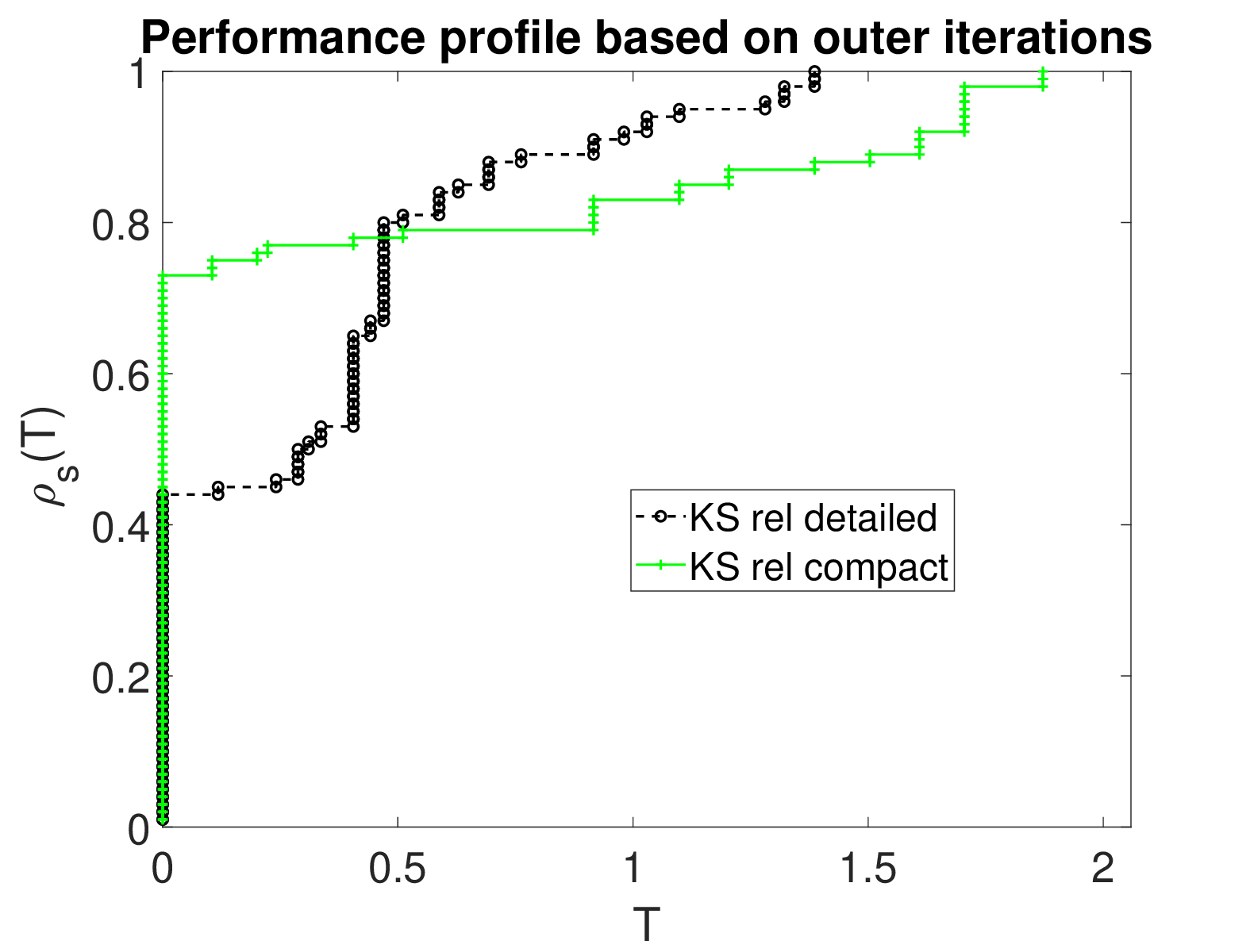}
     \end{subfigure}  
     \hfill
     \begin{subfigure} 
         \centering
         \includegraphics[width=0.3\textwidth]{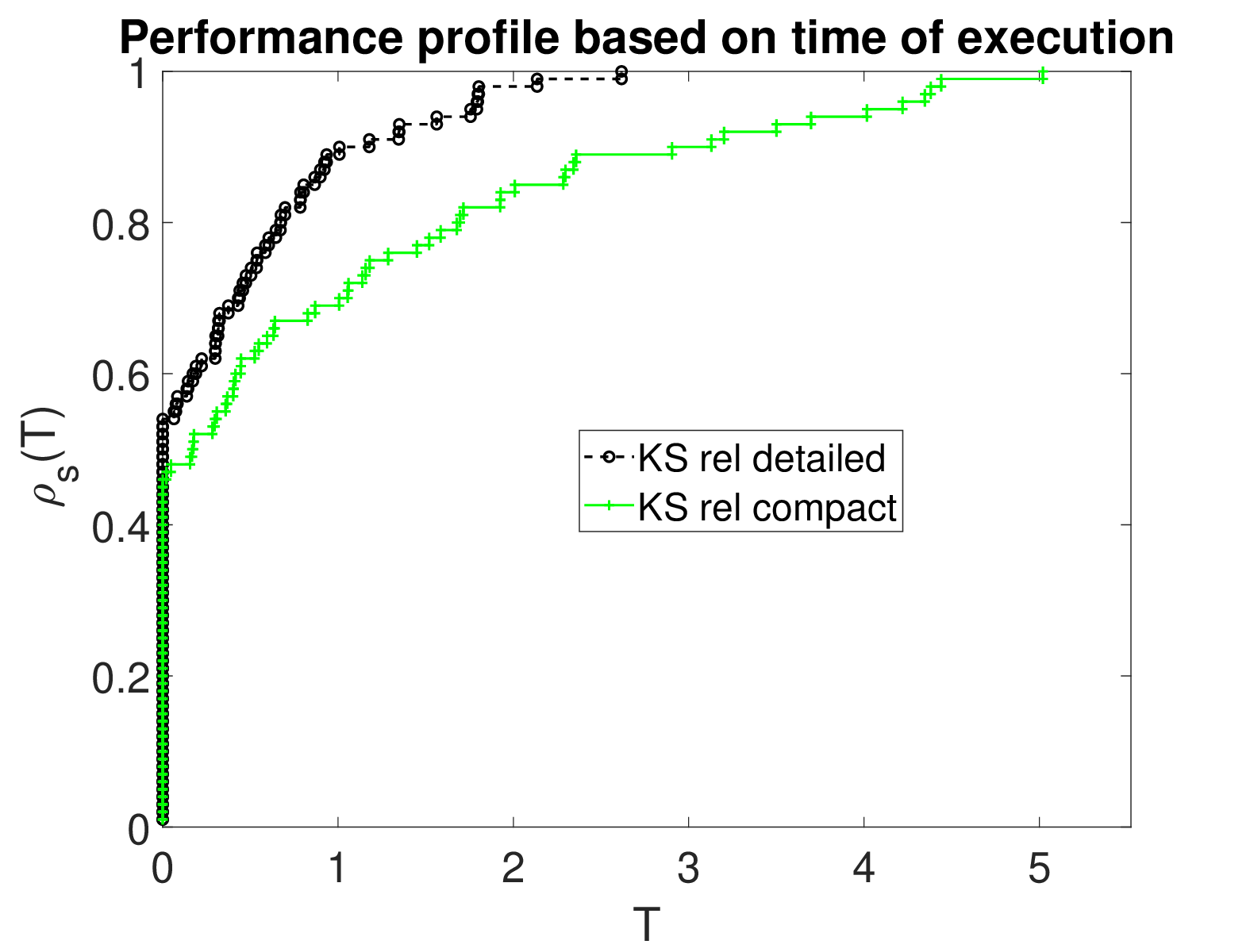}
     \end{subfigure}
  \hfill
     \begin{subfigure} 
         \centering
         \includegraphics[width=0.3\textwidth]{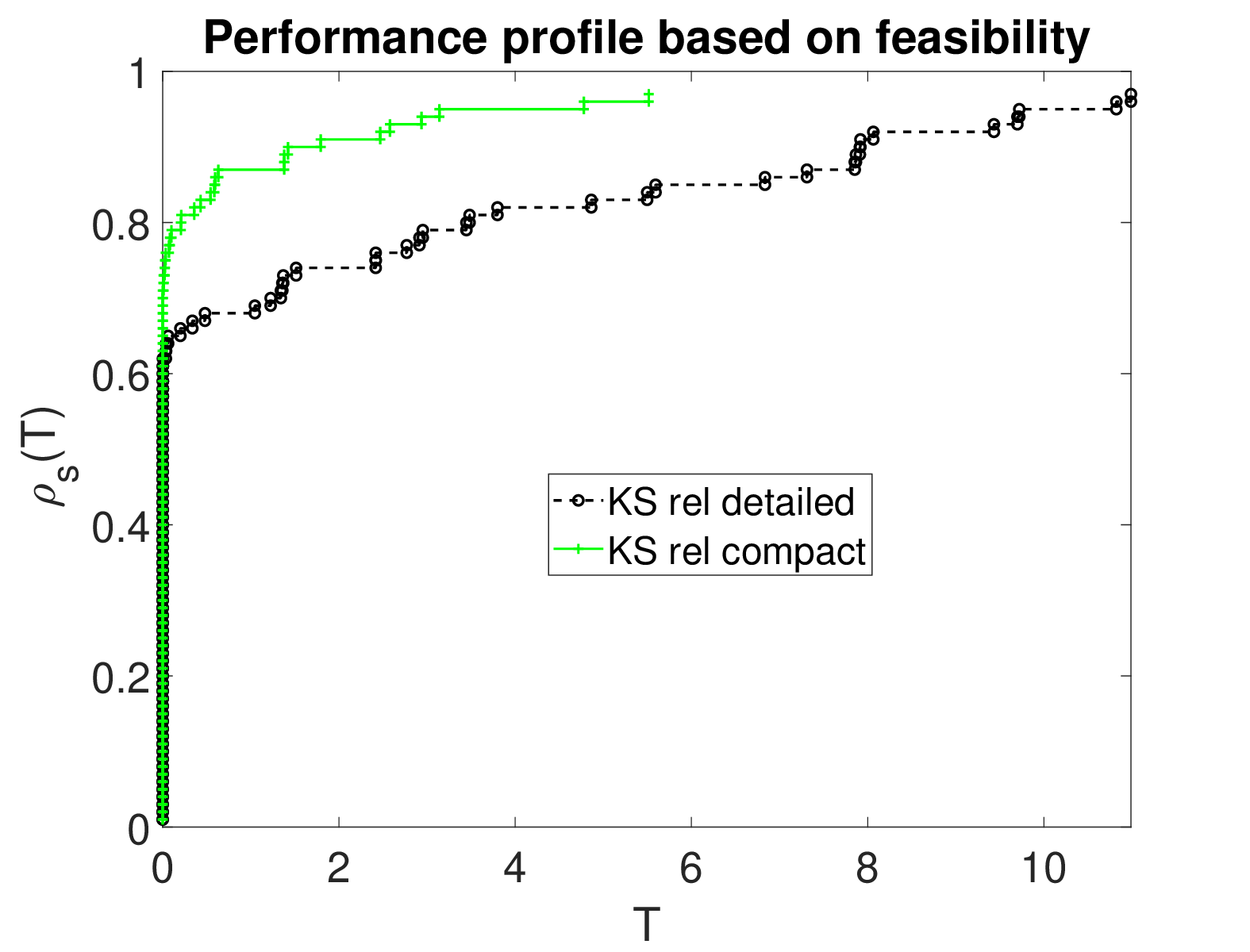}
     \end{subfigure} 
      \caption{Comparison of the detailed and compact form of the relaxation methods with examples that do not satisfy convexity and regularity conditions. From Top: From Top: Scholtes rel(axation), KDB rel(axation), SU rel(axation) and KS methods rel(axation).} 
       \label{figure3} 
\end{figure}
\subsection{Experiment I: Problems  satisfying convexity and linearity constraints }
In this experiment, we utilized 10 random starting points for each of the three examples ({mb\_1\_1\_06}, {mb\_1\_1\_10}, and {mb\_1\_1\_17}), resulting in a total of 30 instances. The graphical comparison can be seen in Figures \ref{figure1} and \ref{figure2}. 
In the evaluation across all five relaxations for the detailed systems, note that the LF method exhibits falsely superior performance compared to the other methods. This is attributed to the fact that the LF method terminated prematurely due to the additional stopping criterion included in the implementation. Excluding the LF method, it is noteworthy that KDB has the smallest number of outer iterations in approximately 60\% of the instances, followed by Scholtes (38\%), KS (22\%), and SU, which exhibited the highest number of iterations for the majority of instances. A comparable pattern is also evident in the execution time across the methods.  Furthermore, the Scholtes relaxation successfully computes solutions satisfying C-stationarity conditions in 6 instances, whereas the points computed by the other relaxations do not meet C-stationarity conditions in any of the instances. Additionally, the number of instances for which each algorithm successfully computes a C-stationary point is detailed in Table \ref{tab1}. In terms of accuracy, the Scholtes relaxation  successfully generates the known solution for approximately 70\% of the instances, followed by the LF method at 54\%. KDB and SU each achieves a success rate of 40\%, while KS reaches 30\%. Notably, each method computes the known optimistic solution for some instances, with Scholtes at about 58\%, LF at 30\%, KDB at 23\%, SU at 30\%, and KS at 30\%. In conclusion, despite the Scholtes relaxation method's relatively lower performance in terms of iteration number and execution time, it consistently provides better solutions compared to other relaxations (i.e., LF, KDB, KS, and SU). Furthermore, the KS relaxation exhibits the poorest performance in terms of iteration number and execution time and often fails to yield good solutions to the problem.

For the compact system, the Scholtes relaxation exhibits the best performance in approximately 55\% of instances, followed by KDB at 50\%, KS at 40\%, and SU at 20\%. Regarding accuracy, it is evident that the Scholtes method computes the known solution for approximately 80\% of the instances, while KDB and SU achieves the known solution for about 40\% each, followed by KS at 28\%. These results indicate that, although the detailed system showed better performance in terms of the number of iterations and execution time for the S, KDB, and SU relexations, the compact versions of the systems exhibit superior performance regarding the practicality and precision. Additionally, across all five relaxations, the S one consistently showcases superior performance compared to the remaining ones. 

\begin{figure}
     \centering
     \begin{subfigure} 
         \centering
         \includegraphics[width=0.3\textwidth]{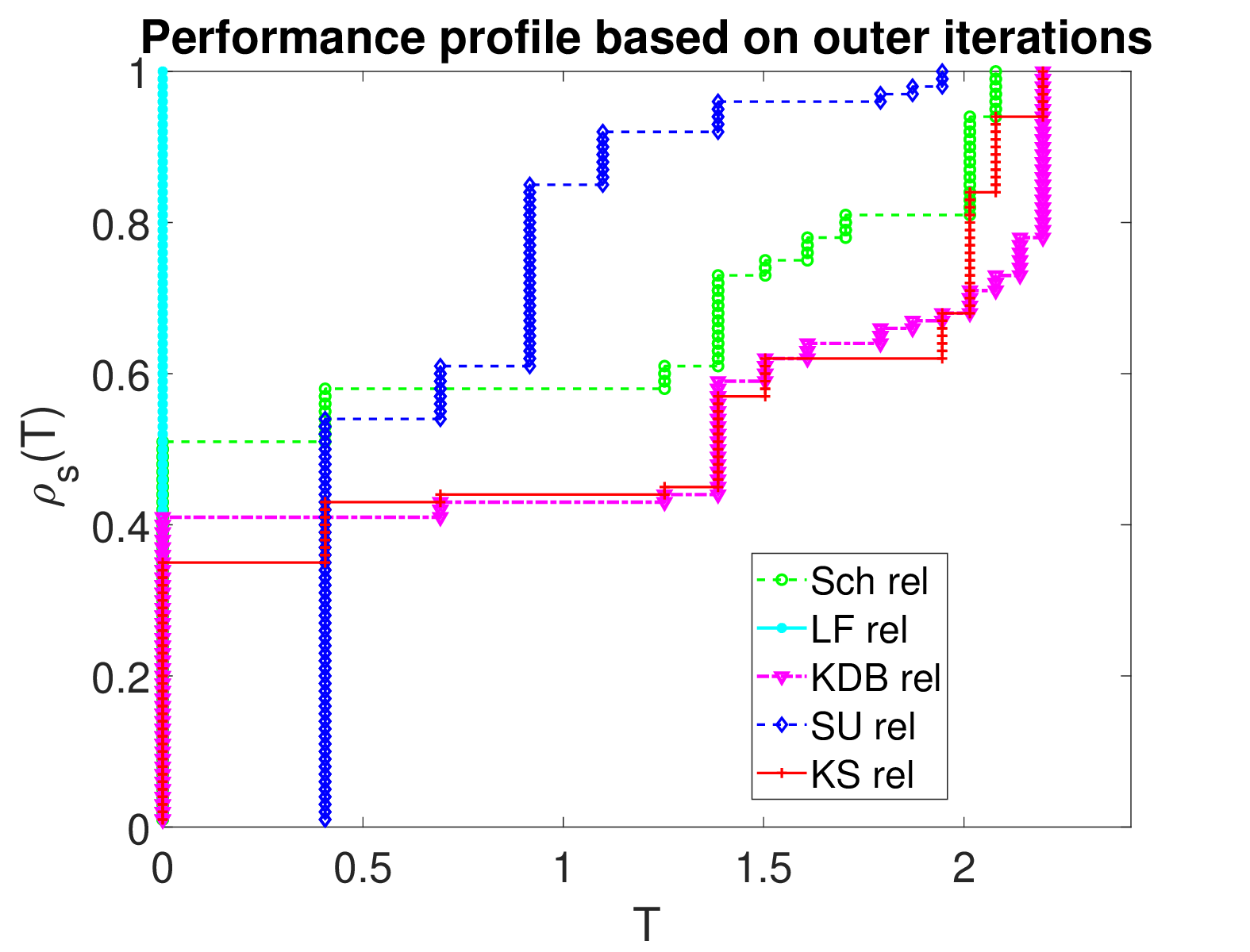}
     \end{subfigure}
     \hfill
  \begin{subfigure} 
         \centering
         \includegraphics[width=0.3\textwidth]{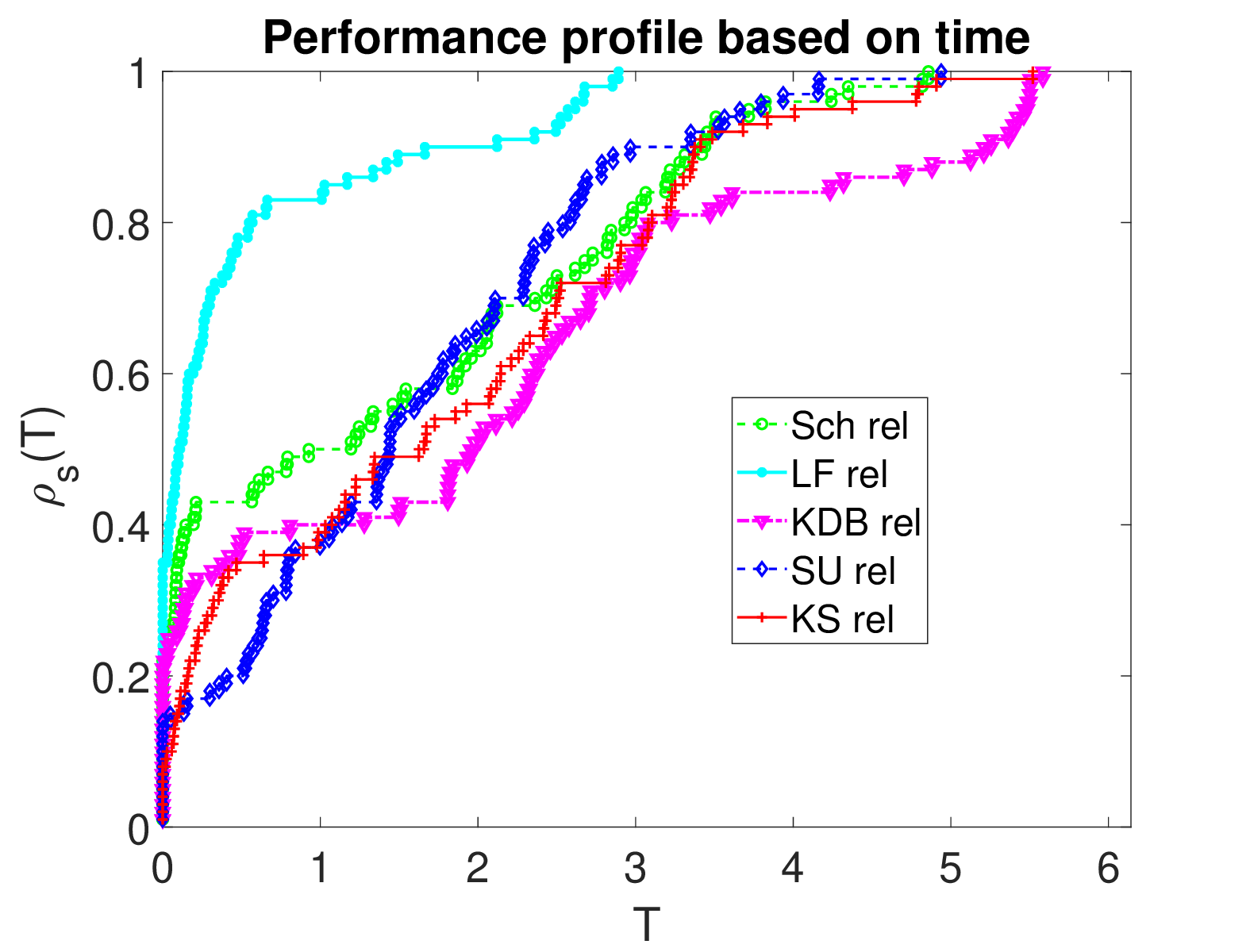}
     \end{subfigure}
     \hfill
     \begin{subfigure} 
         \centering
         \includegraphics[width=0.3\textwidth]{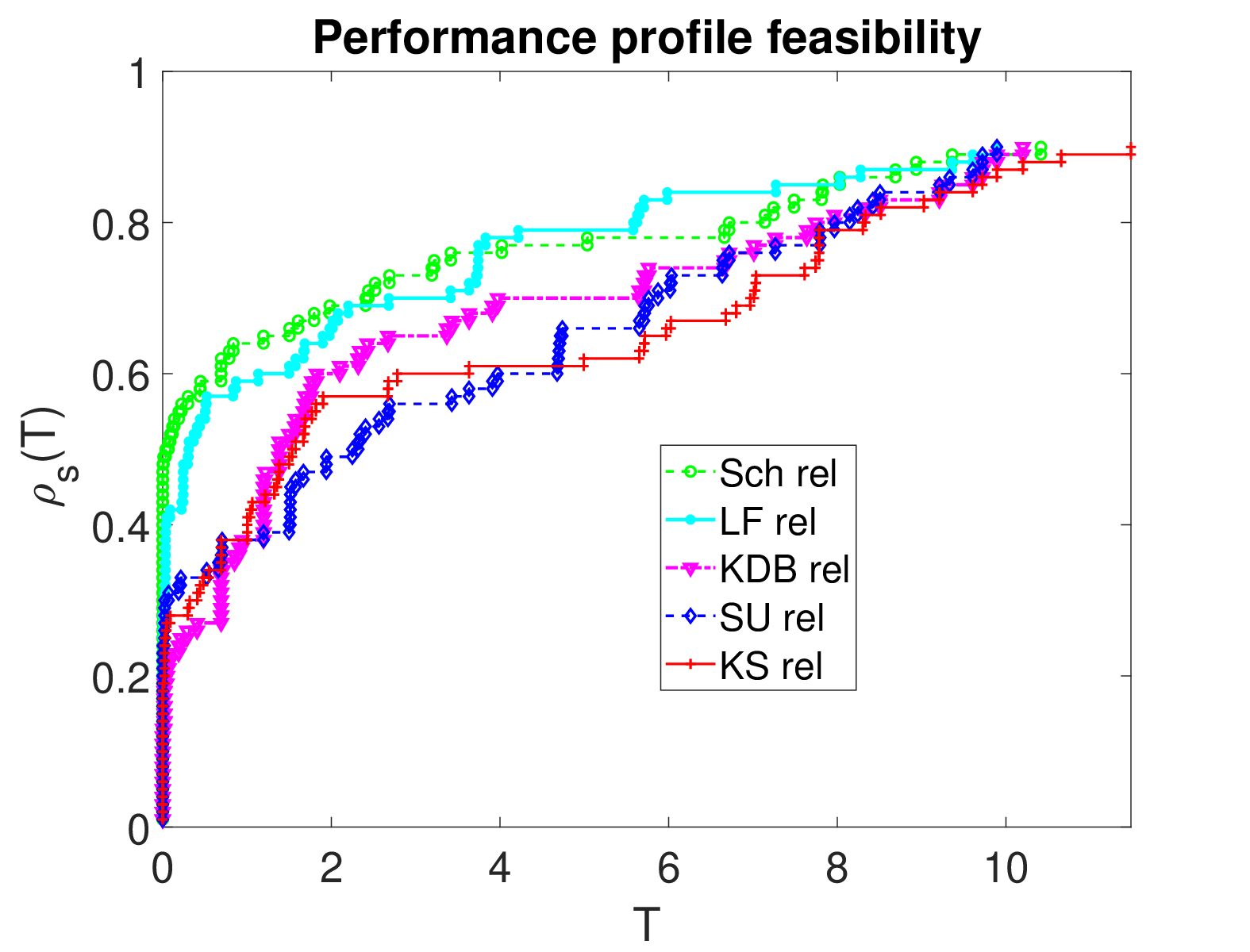}
     \end{subfigure}
     \hfill
       \begin{subfigure} 
         \centering
         \includegraphics[width=0.3\textwidth]{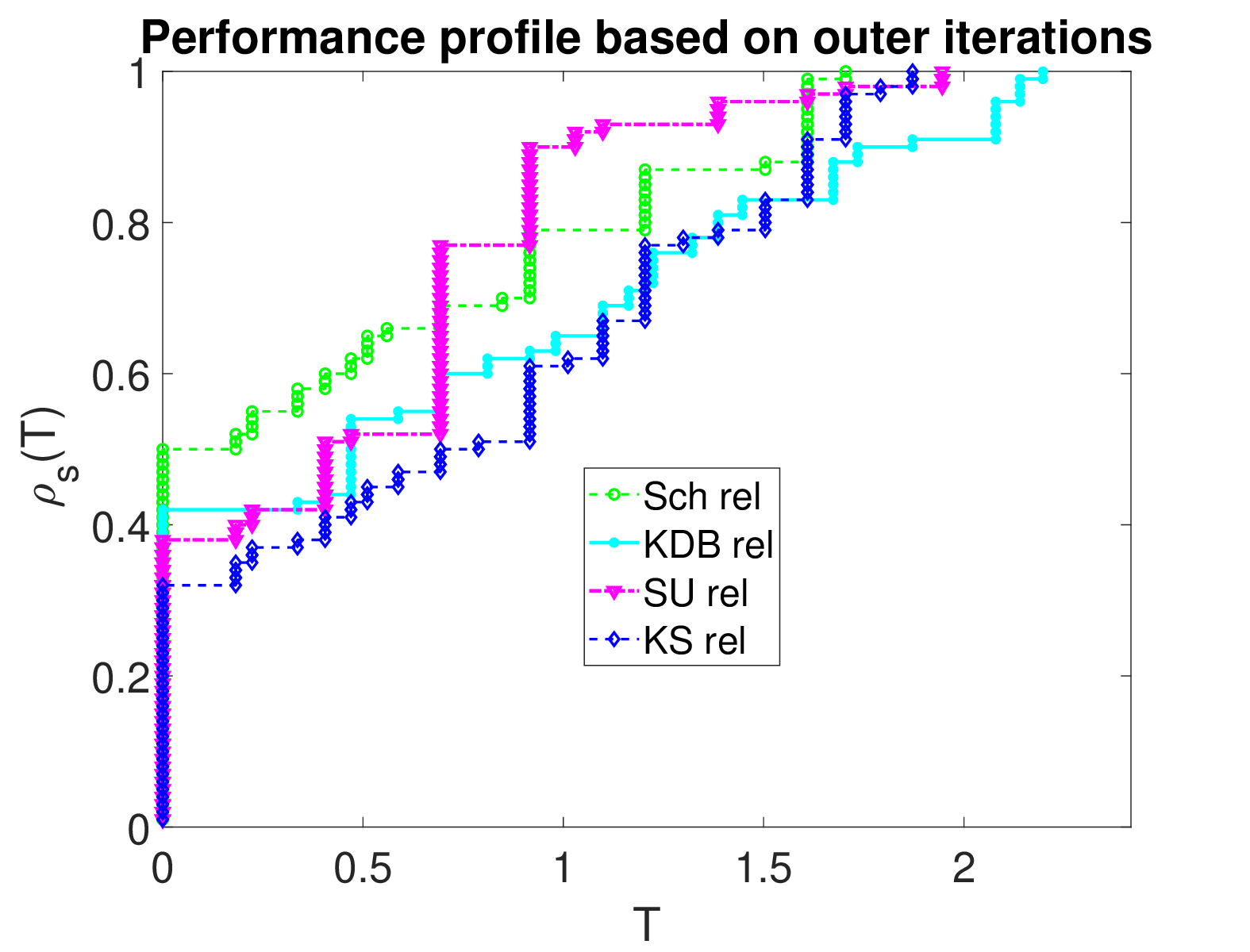}
     \end{subfigure}
     \hfill
      \begin{subfigure} 
         \centering
         \includegraphics[width=0.3\textwidth]{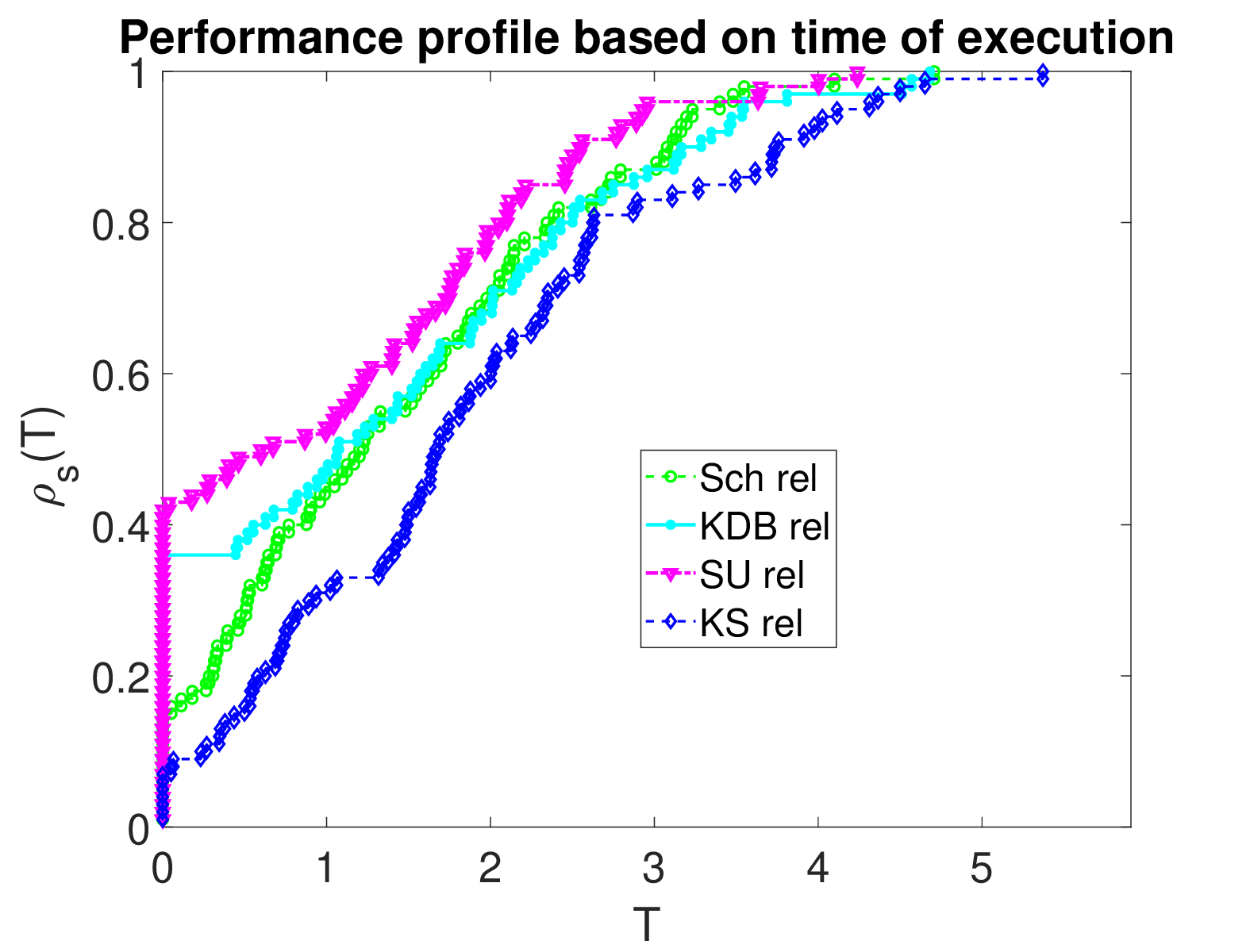}
     \end{subfigure}
     \hfill
       \begin{subfigure} 
         \centering
         \includegraphics[width=0.3\textwidth]{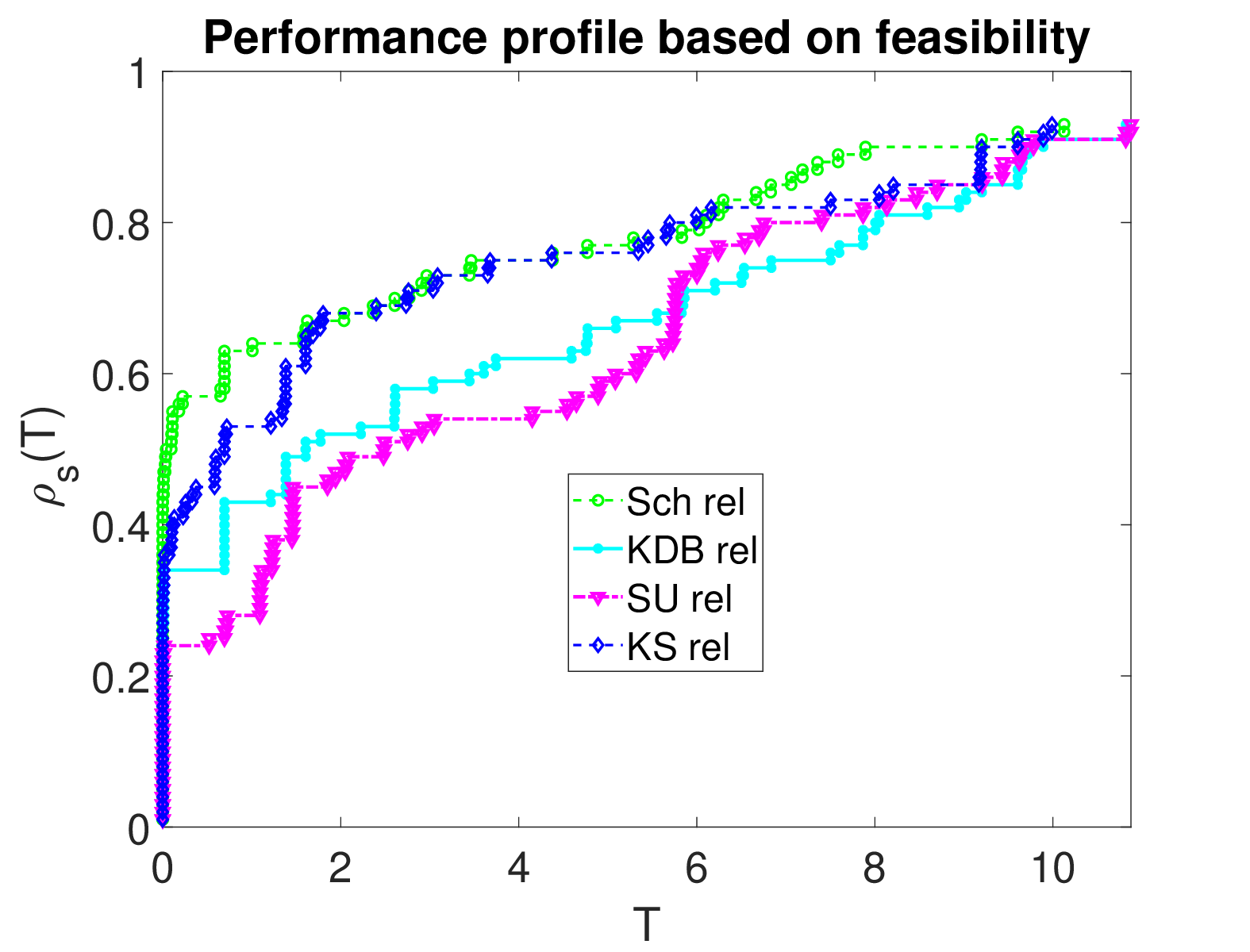}
     \end{subfigure}
     Numerical results for examples that do not satisfy convexity and regularity conditions across all relaxation methods.
     \label{Figure4}
\end{figure} 

\subsection{Experiment II: Problems that do not satisfy  convexity and linearity constraints}
In this experiment, we utilize nonconvex bilevel problem examples, specifically, \textrm{mb\_1\_1\_03}, {mb\_1\_1\_04}, {mb\_1\_1\_05}, {mb\_1\_1\_07}, {mb\_1\_1\_08}, {mb\_1\_1\_09}, {mb\_1\_1\_11}, {mb\_1\_1\_12}, {mb\_1\_1\_13}, and {mb\_1\_1\_14} from \cite{Mitsos2006}. The primary objective of this experiment is to analyze the performance of each relaxation technique when the conditions required for the KKT reformulation \eqref{MPCC} are not satisfied (see Section \ref{Introduction}). We employ 10 random starting points for each problem, resulting in 100 instances used for the comparison. Similar to the previous observations, it is noted that the LF relaxation exhibits falsely better performance due to the second stopping criterion used in the implementation. Excluding the LF method, the Scholtes method also demonstrates the best performance in about 54\% of the instances, followed by KDB at 40\%, KS at 32\%, while SU exhibits  the largest number of iterations for the majority of instances. With regards to accuracy, it is important to highlight that the relaxation methods successfully computed both the known solutions of the pessimistic and optimistic problems in several instances. Therefore, we can conclude that despite the failure to meet the required assumptions, the methods still produce satisfactory solutions for the problems.


\section*{Acknowledgments}
The work of KN is supported by the University of Oran 1 under the PRFU grant with
reference 48/S.D.R.F/2023, while that of AZ is partly supported by the EPSRC grant EP/V049038/1.

\section*{Data availability statement}
The test problems used for the experiments in this paper can be found in \cite{Mitsos2006,WiesemannTsoukalasKleniatiRustem2013}. As for the codes used for
the experiments, they are based on MATLAB’s fsolve and can be requested by an email to the authors.

\section*{Conflict of interest statement}
The authors have no conflicts of interest to declare.

\bibliographystyle{informs2014}

\appendix

\section{Proof of Proposition \ref{lem}}\label{A}
(a) is automatically satisfied given that the functions $f$ \eqref{S(x)} and $g$ \eqref{Y(x)} describing the lower-level problem are continuously differentiable, while the functions $\varphi^t_{i, SU}$ and $\varphi^t_{i, KS}$ are continuous, for $i=1, \ldots, q$. 
Observe also that assertion (d) derives from assertion (c) since for any $t>0$, it holds that $\mathcal{D}(x) \subset \mathcal{D}^{t}_{\mathcal{R}}(x)$. 
Subsequently, to prove assertion (e), it suffices to show that $\underset{k\rightarrow\infty}{\lim\sup}
\mathcal{D}_{\mathcal{R}}^{t_{k}}(x)\subset\mathcal{D}(x)$ for any $(t_{k})\downarrow0$ and all $\mathcal{R}\in \left\{\mbox{S}, \mbox{LF}, \mbox{SU}, \mbox{KS}\right\}$. For remaining properties, see \cite[Proposition 2.1]{BNZ} for $\mathcal{R}=\mbox{S}$ and see the proofs of the cases $\mathcal{R}\in \left\{\mbox{LF}, \mbox{SU}, \mbox{KS}\right\}$ below, respectively.   {Let us prove  assertions (c) (resp. (e))  in the context of $\mathcal{R}=\mbox{KDB}$. Observe  that  inclusion $\underset{t>0}{\cap}\mathcal{D}_{KDB}^{t}(x)\subset\mathcal{D}(x)$ is obvious. Given $(t_{k})\downarrow0$ and
$(y,u)\in\mathcal{D}(x)$ with $(u_{i},g_{i}(x,y))\neq(0,0)$ for all $i=1, \ldots, q$, we have $u_{i}+t_{k}\geq0$ and $g_{i}(x,y)-t_{k}\leq0$ for $k\in \mathbb{N}$. Furthermore, since $(t_{k})\downarrow0$ and $u_{i}-g_{i}(x,y)>0$ for $i=1, \ldots, q$,  there exists
$k_{0}\in\mathbb{N}$ such that 
\[
t_{k}\in\left]0,\;\;\, 
\underset{1\leq i\leq q}{\min}\left\{u_{i}-g_{i}(x,y)\right\}\right] \;\, \mbox{ for all } \;\, k\geq k_{0}.
\]
Hence, for all $i=1, \ldots, q$ and  $k\geq k_{0}$, 
$
(u_{i}-t_{k})(g_{i}
(x,y)+t_{k})=-t_{k}^{2}+(u_{i}-g_{i}(x,y))t_{k}\geq0
$ which leads to 
$(y,u)\in\underset{k\geq k_{0}}{\cap}\mathcal{D}_{KDB}^{t_{k}}(x).$  Now for assertion (e), take $(y,u) \in \underset{(t,x')\rightarrow (0^+,x)}{\lim\sup}
\mathcal{D}_{KDB}^{t}(x')$ so there exists sequences $(t_{k})\downarrow0$ and $(x^{k})
\rightarrow x$ along with a sequence of points $(y^{k},u^{k}
)\in\mathcal{D}_{KDB}^{t_{k}} $  which converges (up to a subsequence)
to $(y,u)$. Now as for any $k$ 
$\mathcal{L}(x^{k},y^{k},u^{k}) =0$ and 
\begin{align*}
-u_{i}^{k}-t_{k}  \leq0, \;\; 
g_{i}(x^{k},y^{k})-t_{k}   \leq0, \;\;
-(u_{i}^{k}-t_{k})(g_{i}(x^{k},y^{k})+t_{k})  & \leq0,
\end{align*}
 taking the limit as $k\rightarrow\infty$, we get $(y,u)\in\mathcal{D} (x)$
since all functions are continuous.}\\[3ex]
\textbf{\underline{Case $\mathcal{R}$=LF}.} Start by recalling that the relaxation $\mathcal{R}$=LF consists of  replacing the complementarity conditions in $\mathcal{D}(x)$, for a given $x\in X$, by 
\begin{equation}\label{LFc}
    \begin{array}
	[c]{l}
	\;\; u_{i}g_{i}(x,y)+t^{2}\geq0,\text{ } \;\;
	(u_{i}+t)(-g_{i}(x,y)+t)-t^{2}\geq 0 , \;\; i=1,\ldots,q.
	\end{array}
\end{equation}
Hence, by taking $t=0$, we have that $\mathcal{D}(x)\subset \mathcal{D}_{LF}^{0}(x) $ for any $x$.
    	 For item (c), take $(y,u)\in \underset{t>0}{\cap}\mathcal{D}_{LF}^{t}(x)\ $ so that for all $t>0$, \eqref{LFc} is satisfied. Thus, for $i=1, \ldots, q$,  $u_{i}g_{i}(x,y)=0$ and $u_{i}\geq g_{i}(x,y)$. This leads to $(y,u)\in \mathcal{D}(x)$.
	To prove the Painlev\'{e}--Kuratowski convergence (assertion (e)), it suffices to show that $\underset{t\downarrow0}{\lim\sup}\,\mathcal{D}^{t}_{LF}(x) \subset \mathcal{D}(x)$. Let $(y,u)\in$ $\underset{t\downarrow0}{\lim\sup}\,\mathcal{D}^{t}_{LF}(x)$, it follows that
	for a sequence $(t_{k})\downarrow0$, the point $(y,u)$ is a cluster point of a sequence
	in $\mathcal{D}^{t_{k}}_{LF}(x)$; i.e., there exists a sequence $(y^{k}
	,u^{k})_{k}$ with $(y^{k},u^{k})\in\mathcal{D}^{t_{k}}_{LF}(x)$, which converges
	to $(y,u)$. Thus, for any $k$, we have $\mathcal{L}(x,y^{k},u^{k})=0$ and
	\[
	\begin{array}
	[c]{l}
	\;\; u_{i}^{k}g_{i}(x,y^{k})+t_{k}^{2}\geq0,\text{ } \;\;
	(u_{i}^{k}+t_{k})(-g_{i}(x,y^{k})+t_{k})-t^{2}_{k}\geq 0 , \;\; i=1,\ldots,q.
	\end{array}
	\]
	Thus, 
	$-(u_{i}^{k}-g_{i}(x,y^{k}))t_{k}\leq u_{i}^{k}(-g_{i}(x,y^{k}))\leq t_{k}^{2}$ and $t_{k}+(u_{i}^{k}-g_{i}(x,y^{k})\geq 0$, $i=1,\ldots,q$.
	Taking the limit as $k\rightarrow{+\infty}$, we get $\mathcal{L}(x,y,u)=0$ and
    $u_{i}(-g_{i}(x,y))=0$, $u_{i}-g_{i}(x,y)\geq 0$, $i=1,\ldots,q$. Hence, $(y,u)\in\mathcal{D}(x)$, and therefore,  $\underset
	{t\rightarrow0^{+}}{\lim\sup}\,\mathcal{D}^{t}_{LF}(x)\subset\,\mathcal{D}(x).$ 

 \medskip
\noindent\textbf{\underline{Case $\mathcal{R}$=SU}.} In this case, the properties of the regularization function $\theta(\cdot)$ and the corresponding relaxation function $\varphi_{i,SU}^{t}$ for $i=1, \ldots, q$ as given in \cite{SU} 
are used.  Let us start by proving that $\mathcal{D}_{SU}^{t_{1}}(x)\subset\mathcal{D}_{SU}^{t_{2}}(x)$ for
		$0<t_{1}<t_{2}$. It suffices to show that 
$\varphi_{i,SU}^{t_2}(x,y,u)\leq0 $ whenever $\varphi_{i,SU}^{t_1}(x,y,u)\leq0 $  for $i=1,\ldots,q$. Let $i=1,\ldots,q$, then for  $\left\vert
u_{i}+g_{i}(x,y)\right\vert \geq t_{2}$, it holds that
$\varphi_{i,SU}^{t_2}(x,y,u) = \varphi_{i,SU}^{t_1}(x,y,u).$ So in this case, if $(y,u)\in \mathcal{D}_{SU}^{t_{1}}(x)$, then $(y,u) \in \mathcal{D}_{SU}^{t_{2}}(x)$. Now, let $i=1,\ldots,q$, it holds that for  $ t_{1} \leq  \left\vert
u_{i}+g_{i}(x,y)\right\vert <t_{2}$, we have from Lemma 3.1 in \cite{SU}  (see too \cite[Lemma 4.4 (a)]{Hoheisel2010}) that
\[
\varphi_{i,SU}^{t_2}(x,y,u)= u_{i}-g_{i}(x,y)-t_{2}\theta\left(\dfrac{u_{i}+g_{i}(x,y)}{t_{2}}\right) <u_{i}-g_{i}(x,y)-\left\vert
u_{i}+g_{i}(x,y)\right\vert =\varphi_{i,SU}^{t_1}(x,y,u).
\]
Also in this case $(y,u)\in\mathcal{D}_{SU}^{t_{1}}(x)$  implies $(y,u) \in \mathcal{D}_{SU}^{t_{1}}(x)$. 
Finally, for $i=1, \ldots, q$ such that $\left\vert
 u_{i}+g_{i}(x,y)\right\vert <t_{1}$, it follows gain from Lemma 3.1 in \cite{SU} that the function $t\rightarrow\varphi_{i,SU}^{t}(x,y,u)$ is strictly
    monotonically decreasing on  $ [t_1,+\infty[$. Hence, we have $ \varphi_{i,SU}^{t_2}(x,y,u) \leq \varphi_{i,SU}^{t_1}(x,y,u)$ and assertion (b) follows. 
For assertion (c), let $(y,u) \in \underset{t>0}{\cap}\mathcal{D}_{SU}^{t}(x)$. For any $t>0$, 
\[
 \mathcal{L}(x,y,u)=0, \;\; u_{i}\geq0,\text{ \ }-g_{i}(x,y)\geq0,\text{ \ } 
\varphi_{i,SU}^{t}(x,y,u) \leq0\text{ \ } \forall i=1,\ldots,q.
\]
Taking the limit as $t$ goes to $0^+$ and considering the fact that 
\[
\underset{t\rightarrow0^+}{\lim}\varphi_{i,SU}^{t}(x,y,u)=u_{i}-g_{i}(x,y)-\left\vert u_{i}+g_{i}(x,y)\right\vert
\]
(see property (c) of Lemma 4.4 in \cite{Hoheisel2010}),  
we obtain
$(y,u) \in \mathcal{D}(x)$. Let us now prove assertion (e) in which we only need to show that $\underset{t\rightarrow0^+}{\limsup} \, \mathcal{D}_{SU}^{t}(x)\subset \mathcal{D}(x).$
Let $(y,u)\in$ $\underset{t\downarrow0}{\lim\sup}\,\mathcal{D}^{t}_{SU}(x)$, that is 
	for a sequence $(t_{k})\downarrow0$, there exists a sequence $(y^{k}
	,u^{k})_{k}$ with $(y^{k},u^{k})\in\mathcal{D}^{t_{k}}_{SU}(x)$ which converges
	to $(y,u)$. Thus, for any $k$,  $\mathcal{L}(x,y^{k},u^{k})=0$ as well as 
	$u_{i}^{k}\geq 0$, $g_{i}(x,y^{k})\leq0$ and $\varphi_{i,SU}^{t_k}(x,y^{k},u^{k})\leq 0$ for all $i=1,\ldots,q$.
		Taking the limit of this system $k\rightarrow +\infty$, we get $\mathcal{L}(x,y,u)=0$ with  $u_{i}\geq 0$, $g_{i}(x,y) \leq 0$ and also by considering Lemma 4.4 (c) in \cite{Hoheisel2010}, we have 
\[
 u_{i}-g_{i}(x,y)-\left\vert u_{i}+g_{i}(x,y)\right\vert  \; \leq 0 \;\,\mbox{ for all } \;\,i=1,\ldots,q,
 \]
  which leads to $(y,u)\in\mathcal{D}(x)$. Hence,  $\underset
	{t\rightarrow0^{+}}{\lim\sup}\,\mathcal{D}^{t}_{SU}(x)\subset\,\mathcal{D}(x).$
 \medskip
${}$\\
\noindent\textbf{\underline{Case $\mathcal{R}$=KS}.} Let us show the second assertion. Take  $t_{2}\geq t_{1} \geq 0$ and
$(y,u)$ an arbitrary element in $ \mathcal{D}_{KS}^{t_{1}}(x)$. Thus by taking into account the fact that for  any $i$, the function $\mathbb{R}^{+}\ni
t\longmapsto\phi^{t}_{i,KS}(x,y,u) $ decreases, 
we have
$$ \phi^{t_2}_{i,KS}(x,y,u) \leq \phi^{t_1}_{i,KS}(x,y,u)\leq 0$$
since $(y,u)\in \mathcal{D}_{KS}^{t_{1}}(x)$.
For assertion (c), it suffices to show that for any $(y,u)\in \underset{t>0}{\cap}\mathcal{D}_{KS}^{t}(x)$, it holds for all $i=1,\ldots,q$, $u_{i}g_{i}(x,y)=0$ 
or equivalently $\varphi(u_{i},-g_{i}(x,y))\leq 0$  where $\varphi$ is the NCP-function used in \cite{KS} (namely, $\varphi(a,b)=ab$ if $a+b\geq0$, $\varphi(a,b)=-\dfrac{1}{2}(a^{2}+b^{2})$ if $a+b<0$). 
Assume that there is an  $i=1,\ldots,q$ such that $\varphi(u_{i},-g_{i}(x,y))> 0$. Considering Lemma 3.1(c) in \cite{KS},
\[
\left\{
\begin{array}[c]{l}
\varphi(a,b)>0\text{\ \ \ if \ \ }a>0\text{ and }b>0,\\
\varphi(a,b)<0\text{\ \ \ if \ \ }a<0\text{ or }b<0.
\end{array}
\right.
\]
Therefore, $\varphi(u_{i},-g_{i}(x,y))> 0$ if $u_{i}>0$ and $g_{i}(x,y)<0$. So for $0<t^{*}< \min(u_{i},-g_{i}(x,y))$, we have $ u_{i}-g_{i}(x,y)\geq 2t^{*}$ and hence, 
\[
\phi^{t^{*}}_{i,KS}(x,y,u) =(u_{i}-t^{*})(-g_{i}(x,y)-t^{*})\geq 0.
\]
Consequently,  $(y,u)\notin \mathcal{D}_{KS}^{t^{*}}(x)$ which a contradiction with the assumption that $(y,u)\in \underset{t>0}{\cap}\mathcal{D}_{KS}^{t}(x).$
For assertion (e), take a sequence $(t_{k})\downarrow0$ and $(y,u)$ the  limit of a sequence $(y^{k},u^{k})_{k}$
such that $(y^{k},u^{k})\in\mathcal{D}_{KS}^{t_{k}
}(x)$.
Thus, for all $k$ and $i=1,\ldots,q$, 
 \[
 \mathcal{L}(x,y^{k},u^{k})=0,\,\,
u_{i}^{k}\geq0,\,\, g_{i}(x,y^{k})\leq0 
\] with $(u_{i}^{k}-t_{k})(-g_{i}(x,y^{k})-t_{k})\leq0$ for all $i$ such that $u_{i}^{k}-g_{i}(x,y^{k})\geq2t_{k}$, given that when $u_{i}^{k}-g_{i}(x,y^{k})<2t_{k}$, 
\[
\phi^{t_{k}}_{i,KS}(x,y^{k},u^{k}
)=-\dfrac{1}{2}((u_{i}^{k}-t_{k})^{2}+(-g_{i}(x,y^{k})-t_{k})^{2})\leq0.
\]
Taking now the limit as $k\rightarrow +\infty$, we get by continuity that $\mathcal{L}(x,y,u)=0$, $u_{i}\geq0$, $g_{i}(x,y)\leq0$, $u_{i}
g_{i}(x,y)\geq0$ for all $i=1,\ldots,q$, meaning that $(y,u)\in\mathcal{D}
(x)$  and assertion (e) follows.\\

\section{Proof of Theorem \ref{ConvergenceResult}}\label{B}
Since $x^{k}$ is a stationary point of  \eqref{KKT-RG} for $t:=t^k$, there exist a vector
	$(y^{k},u^{k})\in\mathcal{S}_\mathcal{R}^{t_{k}}(x^{k})$ and multipliers $(\alpha
	^{k},\beta^{k},\lambda^{k})$ such that the relationships (\ref{Er1S*})-(\ref{Er5S*}) are satisfied. On the other hand, from (\ref{e1}), there exists $(z^{k},w^{k})\in
	\mathcal{S}_{p}(\bar{x})$ such that
\begin{equation}\label{Tr1-1}
    \underset{k \rightarrow
		\infty}{\lim }\left\Vert y^{k}-z^{k}\right\Vert = 0 \;\, \text{ and } \;\,\underset{k \rightarrow
		\infty}{\lim }\left\Vert
	u^{k}-w^{k}\right\Vert = 0
\end{equation}
	and since $\mathcal{S}_{p}(\bar{x})$ is compact, the sequence $(y^{k}
	,u^{k})_{k}$ (up to a subsequence) converges to some  point $(\bar{y},\bar
	{u})\in\mathcal{S}_{p}(\bar{x})$. For the rest of the proof, see the cases below for $\mathcal{R}\in \{\mbox{LF}, \mbox{SU}, \mbox{KDB}, \mbox{KS}\}$ and \cite{BNZ} for $\mathcal{R}=\mbox{S}$.\\[3ex]
\noindent \textbf{\underline{Case $ \mathcal{R}$= LF}:} To proceed, consider the following index sets defined for $(x,(y,u))\in X\times\mathcal{D}_{LF}^{t}(x)$ with $t>0$: 
\[
\begin{array}{rll}
I_{G}(x) & := & \left\{ i\in \{1,\ldots,p\}:\;\;G_{i}(x)=0\right\} , \\
I^{+}(x,y,u,t) & := & \left\{ i\in
\{1,\ldots,q\}:\;\;u_{i}g_{i}(x,y)+t^{2}=0\right\} , \\
I^{-}(x,y,u,t) & := & \left\{ i\in
\{1,\ldots,q\}:\;\;(u_{i}+t)(-g_{i}(x,y)+t)-t^{2}=0\right\}. \\
&  &
\end{array}
\]
Clearly,
$I^{+}(x,y,u,t)\subset \left\{ i\in
\{1,\ldots,q\}:\,u_{i}>0,\,\,g_{i}(x,y)<0\right\} $
and
$$\theta(x,y,u) \subset I^{-}(x,y,u,t)\subset
\left\{ i\in \{1,\ldots,q\}:\,u_{i}+t>0,\,\,-g_{i}(x,y)+t>0\right\}. $$
	We have for some $(y^{k},u^{k})\in\mathcal{S}_{LF}^{t_{k}}(x^{k})$ and multipliers $(\alpha
	^{k},\beta^{k},\mu^{k},\gamma^{k},\delta^{k})$, the relationships
	\begin{equation}\label{S1}
	\left\{
	\begin{array}{l}
	\nabla_{x}F(x^{k},y^{k})+\nabla G(x^{k})^{\top}\alpha^{k}-\nabla_{x}\mathcal{L}(x^{k},y^{k},u^{k})^{\top}\beta^{k}\\
    \qquad \qquad \qquad +\sum\limits_{i=1}^{q}\left(\gamma_{i}^{k} u_{i}^{k}-(u_{i}^{k}+t_{k})\delta_{i}^{k}\right)  \nabla_{x}g_{i}(x^{k},y^{k})=0, \\
	\nabla_{y}F(x^{k},y^{k})-\nabla_{y}\mathcal{L}(x^{k},y^{k},u^{k})^{\top}
	\beta^{k}+\sum\limits_{i=1}^{q}\left( \gamma_{i}^{k} u_{i}^{k}-(u_{i}^{k}+t_{k})\delta_{i}^{k}\right)  \nabla_{y}g_{i}(x^{k},y^{k})=0,\\
	\forall i=1, \ldots, q:\;\; -\nabla _{y}g_{i}(x^{k},y^{k})\beta ^{k}+\gamma
	_{i}^{k}g_{i}(x^{k},y^{k})+\delta _{i}^{k}(-g_{i}(x^{k},y^{k})+t_{k})=0, \\[1ex]
	\end{array}
	\right.
	\end{equation}
	together with the following conditions:
	\begin{equation}\label{S2}
	\left\{\begin{array}{l}
	\forall i=1, \ldots, p:\;\;\alpha_{i}^{k}\geq0,\;\; G_{i}(x^{k})\leq0,\;\; \alpha_{i}^{k}G_{i}(x^{k})=0, \\[1ex]
	\forall i=1, \ldots, q:\;\;\gamma _{i}^{k}\geq 0,\,\,\, \delta _{i}^{k}\geq
	0,\;\;\\[1ex]
	\forall i=1, \ldots, q:\,\, \gamma
	_{i}^{k}(u_{i}^{k}g_{i}(x^{k},y^{k})+t_{k}^{2})=0, \;\; \delta
	_{i}^{k}((u_{i}^{k}+t_{k})(-g_{i}(x^{k},y^{k})+t_{k})-t_{k}^{2})=0.
	\end{array}
	\right.
	\end{equation}
	Hence, it follows that for all $k$,
	\begin{equation}
	\left\{
	\begin{array}{lll}
	\nabla_{y}g_{i}(x^{k},y^{k})\beta^{k}=\gamma_{i}^{k}g_{i}(x^{k},y^{k}) & \mbox{ for } & i\in\mathrm{supp}\gamma^{k},\\[1ex]
	\nabla_{y}g_{i}(x^{k},y^{k})\beta^{k}=\delta_{i}^{k}(-g_{i}(x^{k},y^{k})+t_{k}) & \mbox{ for } & i\in\mathrm{supp}\delta^{k},\\[1ex]
	\nabla_{y}g_{i}(x^{k},y^{k})\beta^{k}=0 & \mbox{ for } & i\notin (\mathrm{supp}\gamma^{k}\cup\mathrm{supp}\delta^{k}),
	\end{array}
	\right. \label{S2b}
	\end{equation}
	and
	\begin{align}
	\alpha_{i}^{k} &  \geq0,\text{ \ }\mathrm{supp}\alpha^{k}\subset I_{G}
	(x^{k}),\medskip\label{S3}\\
	\gamma_{i}^{k} &  \geq0,\text{ \ }\mathrm{supp}\gamma^{k}\subset I^{+}
	(x^{k},y^{k},u^{k},t_{k}),\medskip\label{S5}\\
	\delta_{i}^{k} &  \geq0,\text{ \ }\mathrm{supp}\delta^{k}\subset I^{-}
	(x^{k},y^{k},u^{k},t_{k}).\label{S6}
	\end{align}
		Thus, $\mathrm{supp}\gamma^{k}\cap\mathrm{supp}\delta^{k}=\emptyset $ with $\mathrm{supp}(z) := \left\lbrace i |\,\, z_i \neq 0\right\rbrace $ being  the support of a vector $z \in \mathbb{R}^n.$ Next, we can then define the following new Lagrange multipliers
	\begin{equation}\label{S8}
	\tilde{\gamma}_{i}^{k}:=\left\{
	\begin{array}{lll}
	\gamma_{i}^{k}u_{i}^{k} & \mbox{ if } \,\,i\in\mathrm{supp}\gamma^{k}\backslash\eta,&\\
	-(u_{i}^{k}+t_{k})\delta_{i}^{k} & \mbox{ if }\,\, i\in\mathrm{supp}\delta^{k}\backslash\eta,&\\
	0& \mbox{ else. }&
	\end{array}
	\right.
	\end{equation}
		Hence, by setting $\tilde{\beta}^{k}:=-\beta^{k}$, the system of optimality conditions in (\ref{S1}) becomes
	\begin{equation}
	\left\{
	\begin{array}{l}
	\nabla_{x}F(x^{k},y^{k})+\nabla G(x^{k})^{T}\alpha^{k}+\nabla_{x}\mathcal{L}(x^{k},y^{k},u^{k})^{T}\tilde{\beta}^{k}\\[0.5ex]
	\qquad \qquad \qquad \qquad +\, \sum\limits_{i=1}^{q}\tilde{\gamma}_{i}
	^{k}\nabla_{x}g_{i}(x^{k},y^{k}) + \sum\limits_{i\in\eta
	}\left( \gamma_{i}^{k} u_{i}^{k}-(u_{i}^{k}+t_{k})\delta_{i}^{k}\right)   \nabla_{x}g_{i}(x^{k},y^{k})=0,\\[1ex]
	\nabla_{y}F(x^{k},y^{k})+\nabla_{y}\mathcal{L}(x^{k},y^{k},u^{k})^{T}\tilde{\beta}^{k}\\
        \qquad \qquad \qquad \qquad +\,\sum\limits_{i=1}^{q}\tilde{\gamma}_{i}^{k}\nabla_{y}g_{i}(x^{k},y^{k})+\sum\limits_{i\in\eta}\left( \gamma_{i}^{k} u_{i}^{k}-(u_{i}^{k}+t_{k})\delta_{i}^{k}\right)  \nabla_{x}g_{i}(x^{k},y^{k}) =0,
	\end{array}
	\right. \label{S10}
	\end{equation}
	with
	\begin{equation}
	\nabla_{y}g_{i}(x^{k},y^{k})\tilde{\beta}^{k}=\left\{
	\begin{array}{lll}
	-\gamma_{i}^{k}g_{i}(x^{k},y^{k})& \mbox{ if }\,\, i\in\mathrm{supp}\gamma^{k}\backslash\upsilon,&\\
	-\delta _{i}^{k}(-g_{i}(x^{k},y^{k})+t_{k})& \mbox{ if }\,\, i\in\mathrm{supp}\delta^{k}\backslash\upsilon,& \\
	-\gamma
	_{i}^{k}g_{i}(x^{k},y^{k})-\delta _{i}^{k}(-g_{i}(x^{k},y^{k})+t_{k})  & \mbox{ if }\,\,  i\in \upsilon,&  \\
	0 & \mbox{ else. } &
	\end{array}
	\right. \label{S11}
	\end{equation}
	 Next, we are going to show that the sequence
	$(\chi_{k}):=
	\begin{pmatrix}
	\alpha^{k},\tilde{\beta}^{k},\tilde{\gamma}^{k},\gamma
	_{\eta\cup\upsilon}^{k},\delta
	_{\eta\cup\upsilon}^{k}
	\end{pmatrix}
	_{k}$ is bounded. Assume that this is not the case and consider the sequence
	$\left(\dfrac{\chi_{k}}{\left\Vert \chi_{k}\right\Vert }\right)_{k}$ that converges to a
	nonvanishing vector $\chi:=(\alpha,\tilde{\beta},\tilde{\gamma
	},\gamma_{\eta\cup\upsilon},\delta_{\eta\cup\upsilon})$ (otherwise, a suitable subsequence is chosen).
	Dividing (\ref{S10}) by $\left\Vert \chi_{k}\right\Vert $ and taking the limit
	$k\rightarrow\infty$ while taking into account the continuous differentiability of all
	involved functions and the fact that $\bar{u}_{i}=0$ for $i\in\eta$, we come up with
	\begin{equation}\label{CQ}%
	\nabla G(\bar{x})^{T}\alpha+\nabla_{x}\mathcal{L}(\bar{x},\bar{y},\bar{u})^{T}\tilde{\beta}+\nabla_{x}g^{T}(\bar{x},\bar{y})\tilde{\gamma}=0
	\mbox{ and }
	\nabla_{y}\mathcal{L}(\bar{x},\bar{y},\bar{u})^{T}\tilde{\beta}+\nabla_{y}g^{T}(\bar{x},\bar{y})\tilde{\gamma}=0.
	\end{equation}
		Let us show now that we have the following inclusions for any $k$ sufficiently large:
	\begin{equation}
	\mathrm{supp}\alpha\subset I_{G}(x^{k})\subset I_{G},\label{S12}%
	\end{equation}%
	\begin{equation}
	\mathrm{supp}\tilde{\gamma}\subset I^{+}(x^{k},y^{k},u^{k},t_{k})\cup I^{-}(x^{k},y^{k},u^{k},t_{k})\backslash\eta\subset\theta\cup\upsilon,\label{S13}%
	\end{equation}
	\begin{equation}
	\mathrm{supp}\tilde{\delta}\subset I^{+}(x^{k},y^{k},u^{k},t_{k})\cup I^{-}(x^{k},y^{k},u^{k},t_{k})\backslash\upsilon\subset\theta\cup \eta.\label{S14}%
	\end{equation}
	Indeed, for (\ref{S12}), let $i\in\mathrm{supp}\alpha$ such that $\alpha_{i}>0$.
	For any $k$ large enough, $\alpha_{i}^{k}>0$ that is, $i\in \mathrm{supp}\alpha^{k}\subset I_{G}(x^{k})$ by (\ref{S3}). Let now $i$ be such that for any $k\geq k_{0},$ $G_{i}(x^{k})=0$. Passing to
	the limit as $k\rightarrow\infty$, we get $G_{i}(\bar{x})=0$ and so $i\in I_{G}(\bar{x}).$ For (\ref{S13}): take $i\in
	\mathrm{supp}\tilde{\gamma}$ thus for any $k$ sufficiently large,  $i\in\mathrm{supp}\tilde{\gamma}^{k}.$ Then for all $k$
	sufficiently large, $i\in\mathrm{supp}\gamma^{k}\cup\mathrm{supp}\delta
	^{k}\backslash\eta.$ Hence, by (\ref{S5}) and (\ref{S6}), $i\in
	I^{+}(x^{k},y^{k},u^{k},t_{k})\cup I^{-}(x^{k},y^{k},u^{k},t_{k})\backslash\eta$.  Now, if
	$u_{i}^{k}g_{i}(x^{k},y^{k})+t_{k}=0$ (resp. $(u_{i}^{k}+t_{k})(-g_{i}(x^{k},y^{k})+t_{k})-t_{k}^{2})=0$) for $k$ large enough, with
	$i\notin\eta.$ Taking the limit, it follows that $\bar{u}_{i}g_{i}(\bar
	{x},\bar{y})=0$ so $i\in\theta\cup\upsilon.$ Similarly, we get (\ref{S14}). On the
	other hand, 
	\begin{align*}
    \alpha_{j} \geq0, \,G_{j}(\bar{x})\leq0, \,\alpha_{j}G_{j}(\bar{x}
	) & =0  \mbox{ for }\;\, j=1,\ldots,p,\\
	\delta_{\eta\cup\upsilon} &  =\lim\delta_{\eta\cup\upsilon}^{k}\geq0,\\
	\tilde{\gamma}_{i} &  =\lim\tilde{\gamma}_{i}^{k}=0\text{ \ \ }\forall i\in\eta,\\
	\nabla_{y}g_{\upsilon}(\bar{x},\bar{y})\tilde{\beta} & = -\lim\left( \gamma_{\upsilon}^{k}g_{\upsilon}(x^{k},y^{k})+\delta _{\upsilon}^{k}(-g_{\upsilon}(x^{k},y^{k})+t_{k})\right)
	=-\left( \gamma_{\upsilon}-\delta_{\upsilon}\right) g_{\upsilon}(\bar{x},\bar{y})=0.
	\end{align*}
	Let us now prove that
	$
	(\tilde{\gamma}_{i}<0\wedge\nabla_{y}g_{i}(\bar{x},\bar{y})\tilde{\beta
	}<0)\vee\tilde{\gamma}_{i}\nabla_{y}g_{i}(\bar{x},\bar{y})\tilde{\beta}=0
	$
	for all $i\in\theta.$ Assume that $\tilde{\gamma}_{i}>0$ or $\nabla_{y}
	g_{i}(\bar{x},\bar{y})\tilde{\beta}>0$ for some $i\in\theta$ with
	$\tilde{\gamma}_{i}\nabla_{y}g_{i}(\bar{x},\bar{y})\tilde{\beta}\neq0.$ If
	$\tilde{\gamma}_{i}>0$ then for all $k$ large enough, $\tilde{\gamma}_{i}
	^{k}>0$ and so from (\ref{S8}), $i\in\mathrm{supp}\gamma^{k}$. Moreover, as
	$\mathrm{supp}\gamma^{k}\cap$ \textrm{supp}$\delta^{k}=\emptyset,$
	$i\notin\mathrm{supp}\delta^{k}$ and from (\ref{S11}) and (\ref{S8}),
	\[
	\nabla_{y}g_{i}(x^{k},y^{k})\frac{\tilde{\beta}^{k}}{\left\Vert \chi
		_{k}\right\Vert }=-\frac{\delta_{i}^{k}}{\left\Vert \chi_{k}\right\Vert }
	g_{i}(x^{k},y^{k})
	\]
	and at the limit, we get $\nabla_{y}g_{i}(\bar{x},\bar{y})\tilde{\beta
	}=0$ as $g_{i}(x^{k},y^{k})$ converges to $g_{i}(\bar{x},\bar{y})=0$
	($i\in\theta$) and the sequence $\left(\dfrac{\delta_{i}^{k}}{\left\Vert \chi
		_{k}\right\Vert}\right)$ is bounded. This leads to a contradiction ($\tilde{\gamma}_{i}\nabla_{y}g_{i}(\bar{x},\bar{y})\tilde{\beta}\neq0$ by
	assumption). Similarly, if $\nabla_{y}g_{i}(\bar{x},\bar{y})\tilde{\beta}>0$ so
	for all $k$ large enough, $\nabla_{y}g_{i}(x^{k},y^{k})\tilde{\beta}^{k}>0$
	and from (\ref{S11}), $i\in\mathrm{supp}\gamma^{k}.$ Then $\dfrac
	{\tilde{\gamma}_{i}^{k}}{\left\Vert \chi_{k}\right\Vert }=\dfrac{\gamma
		_{i}^{k}}{\left\Vert \chi_{k}\right\Vert }u_{i}^{k}$ converges to
	$\tilde{\gamma}_{i}=0$ ($\bar{u}_{i}=0$) which leads to a contradiction too.
	Consequently,
	\[
	\left\{
	\begin{array}{l}
	\nabla_{y}\mathcal{L}(\bar{x},\bar{y},\bar{u})^{T}\tilde{\beta}+\nabla
	_{y}g^{T}(\bar{x},\bar{y})\tilde{\gamma}=0,\\
	\nabla_{y}g_{v}(\bar{x},\bar{y})\tilde{\beta}=0,\text{ \ }\tilde{\gamma}_{\eta}=0,\\
	(\tilde{\gamma}_{i}<0\wedge\nabla_{y}g_{i}(\bar{x},\bar{y})\tilde{\beta}<0)\vee\tilde{\gamma}_{i}\nabla_{y}g_{i}(\bar{x},\bar{y})\tilde{\beta
	}=0\; \mbox{ for }\; i\in\theta,
	\end{array}
	\right.
	\]
	so that $(-\tilde{\beta},-\tilde{\gamma})\in\Lambda_{y}^{em}(\bar{x},\bar
	{y},\bar{u},0)$ and from $(A_{2}^{m})$, we get
	$
	\nabla_{x}\mathcal{L}(\bar{x},\bar{y},\bar{u})^{T}\tilde{\beta}+\nabla
	_{x}g^{T}(\bar{x},\bar{y})\tilde{\gamma}=0$. Using (\ref{CQ}), it yields $\nabla G(\bar{x})^{T}\alpha=0$; this implies that $\alpha=0$ ($\bar{x}$ being upper-level regular). Thus,
	\[
	\left\{
	\begin{array}{l}
	\nabla_{x,y}\mathcal{L}(\bar{x},\bar{y},\bar{u})^{T}\tilde{\beta}+\nabla g^{T}(\bar{x},\bar{y})\tilde{\gamma}=0,\\
	\nabla_{y}g_{v}(\bar{x},\bar{y})\tilde{\beta}=0,\text{ \ }\tilde{\gamma}_{\eta}=0,\\
	(\tilde{\gamma}_{i}<0\wedge\nabla_{y}g_{i}(\bar{x},\bar{y})\tilde{\beta}<0)\vee\tilde{\gamma}_{i}\nabla_{y}g_{i}(\bar{x},\bar{y})\tilde{\beta
	}=0\; \mbox{ for }\; i\in\theta;
	\end{array}
	\right.
	\]
	that is, $(-\tilde{\beta},-\tilde{\gamma})\in\Lambda^{em}(\bar{x},\bar{y},
	\bar{u},0)$ and by the condition $(A_{1}^{m}),$ we get $(\tilde{\beta},
	\tilde{\gamma})=0.$ Hence $(\alpha,\tilde{\beta},\tilde{\gamma}
	)=0$  so
	$(\gamma_{\eta\cup\upsilon},\delta_{\eta\cup\upsilon})\neq0.$ Assume that for $i\in\eta$, $\gamma_{i}>0.$ Thus for $k$ large enough,
	$i\in\mathrm{supp}\gamma_{i}^{k}\backslash\upsilon$ and from (\ref{S11}),
	$\nabla_{y}g_{i}(x^{k},y^{k})\tilde{\beta}^{k}=
	-\gamma _{i}^{k}g_{i}(x^{k},y^{k})$. Dividing by $\left\Vert \chi_{k}\right\Vert$ and passing to the limit, we get $\nabla_{y}g_{i}(\bar{x},\bar{y})\tilde{\beta}=-\tilde{\gamma} _{i}g_{i}(\bar{x},\bar{y}) > 0$  and the
	contradiction follows since $\tilde{\beta}=0$. Assume now that there
	exists $i\in\upsilon$ such that $\gamma_{i}>0.$ Thus for $k$ large enough,
	$i\in\mathrm{supp}\gamma_{i}^{k}\backslash\eta$ so that from (\ref{S8}),
	$\tilde{\gamma}_{i}^{k}= \gamma_{i}^{k}u_{i}^{k}$.  Taking the limit, we get $\tilde{\gamma}_{i}=\gamma _{i}\bar{u}_{i} > 0$  which contradict the fact that  $\tilde{\gamma}=0$).
	Similary, if
	$\delta_{\eta\cup\upsilon}\neq0$ we get a contradiction.
	Consequently, the sequence $(
	\alpha^{k},\tilde{\beta}^{k},\tilde{\gamma}^{k},\gamma
	_{\eta\cup\upsilon}^{k},\delta
	_{\eta\cup\upsilon}^{k})_{k}$ is bounded.
	Consider $(\bar{\alpha},\overline{\beta},\overline{\tilde{\gamma}},\bar{\gamma}_{\eta\cup\upsilon},\bar{\delta}_{\eta\cup\upsilon})$ its limit (up to a
	subsequence) so that 
	$
	\mathrm{supp}\bar{\alpha}\subset I_{G},\text{ }\, \mathrm{supp}\overline{\tilde{\gamma}
	}\subset\theta\cup\upsilon.
	$
	Taking the limit in (\ref{S10}), 
	\[
	\left\{
	\begin{array}{l}
	\nabla_{x}F(\bar{x},\bar{y})+\nabla G(\bar{x})^{T}\bar{\alpha}+\nabla
	_{x}\mathcal{L}(\bar{x},\bar{y},\bar{u})^{T}\overline{\beta}+\sum
	\limits_{i=1}^{q}\overline{\tilde{\gamma}}_{i}\nabla_{x}g_{i}(\bar{x},\bar
	{y})=0,\\[1ex]
	\multicolumn{1}{l}{\nabla_{y}F(\bar{x},\bar{y})+\nabla_{y}\mathcal{L}(\bar
		{x},\bar{y},\bar{u})^{T}\overline{\beta}+\sum\limits_{i=1}^{q}\overline
		{\tilde{\gamma}}_{i}\nabla_{y}g_{i}(\bar{x},\bar{y})=0,}
	\end{array}
	\right.
	\]
	with $\bar{\alpha}\geq0$, $\bar{\alpha}^\top G(\bar{x})=0$, $\overline{\tilde{\gamma}}_{\eta}=0$, and $\nabla_{y}g_{v}(\bar
	{x},\bar{y})\overline{\beta}=0$. 	
	It therefore remains to prove that $\overline{\tilde{\gamma}}_{i}\nabla_{y}g_{i}(\bar
	{x},\bar{y})\overline{\beta}\geq0$ for $i\in\theta.$ Let us assume that for
	some $i\in\theta$,
	$\overline{\tilde{\gamma}}_{i}<0$ and $\nabla_{y}g_{i} (\bar{x},\bar{y})\overline{\beta}>0.$
	So from (\ref{S8}), $i\in\mathrm{supp}\delta^{k}$ for all $k$
	large enough so that,
	\begin{equation}
	\nabla_{y}g_{i}(x^{k},y^{k})\tilde{\beta}^{k}=\left\{
	\begin{array}{lll}
	
	-\delta _{i}^{k}(-g_{i}(x^{k},y^{k})+t_{k})& \mbox{ for }\,\, i\in\mathrm{supp}\delta^{k}\backslash\upsilon,& \\
	-\gamma
	_{i}^{k}g_{i}(x^{k},y^{k})-\delta _{i}^{k}(-g_{i}(x^{k},y^{k})+t_{k})  & \mbox{ for }\,\,  i\in \upsilon,&
	\end{array}
	\right.
	\end{equation}
		Thus, it holds that
	$
	\nabla_{y}g_{i}(\bar{x},\bar{y})\overline{\beta} \leq0,
	$
	which is a contradiction with our assumption that $\nabla
	_{y}g_{i}(\bar{x},\bar{y})\overline{\beta}>0.$ Similarly, if $\overline
	{\tilde{\gamma}}_{i}>0$ and $\nabla_{y}g_{i}(\bar{x},\bar{y})\overline{\beta
	}<0.$ Thus, $\tilde{\gamma}^{k}>0$ for all $k$ large enough and by (\ref{S8}),
	$i\in\mathrm{supp}\gamma^{k}$ for all $k $ large enough. Hence,

	\begin{equation}
	\nabla_{y}g_{i}(x^{k},y^{k})\tilde{\beta}^{k}=\left\{
	\begin{array}{lll}
	
	-\delta _{i}^{k}g_{i}(x^{k},y^{k})& \mbox{ for }\,\, i\in\mathrm{supp}\gamma^{k}\backslash\upsilon,& \\
	-\gamma
	_{i}^{k}g_{i}(x^{k},y^{k})-\delta _{i}^{k}(-g_{i}(x^{k},y^{k})+t_{k})  & \mbox{ for }\,\,  i\in \upsilon,&
	\end{array}
	\right.
	\end{equation}
	and the limit leads to the contradiction that  $\nabla_{y}g_{i}(\bar{x},\bar{y})\overline{\beta}\geq0$. Hence, $\overline{\tilde{\gamma}}_{i}(\nabla_{y}g_{i}(\bar
	{x},\bar{y})\bar{\beta})\geq0$ for $i\in\theta$ and the conclusion
	follows.


${}$\\
\noindent \textbf{\underline{Case $ \mathcal{R}$= SU}:}  Similarly to the previous case, we start the proof here by considering the index sets
\[
\begin{array}
[c]{rll}
I_{G}(x) & := & \left\{  i\in\{1,\ldots,p\}:\;\;G_{i}(x)=0\right\}  ,\\
I_{u}(x,y,u) & := & \left\{  i\in\{1,\ldots,q\}:\;u_{i}=0\right\}  ,\\
I_{g}(x,y,u) & := & \left\{  i\in\{1,\ldots,q\}:\text{ }g_{i}(x,y)=0\right\}  ,\\
I_{\varphi^{t}_{SU}}(x,y,u) & := & \left\{  i\in\ 1,\ldots,q:\text{ }\varphi_{i,SU}^{t}(x,y,u)=0\right\},
\end{array}
\]
 for $(x,(y,u))\in X\times\mathcal{D}_{SU}^{t}(x)$ ($t>0$). Hence, 
$i\in I_{\varphi^{t}_{SU}}(x,y,u)$ if and only if $\max(u_{i},-g_{i}(x,y))\geq t$. 
	We know that there exist a
	vector $(y^{k},u^{k})\in\mathcal{S}_{SU}^{t_{k}}(x^{k})$ and multipliers
	$(\alpha^{k},\beta^{k},\mu^{k},\gamma^{k},\delta^{k})$ such that 
	\begin{equation}
	\left\{
	\begin{array}
	[c]{l}
	\nabla_{x}F(x^{k},y^{k})+\nabla G(x^{k})^{\top}\alpha^{k}+\nabla
	_{x}\mathcal{L}(x^{k},y^{k},u^{k})^{\top}\beta^{k}+\sum\limits_{i=1}
	^{q}\left(  \xi_{1i}^{k}\delta_{i}^{k}-\gamma_{i}^{k}\right)  \nabla_{x}
	g_{i}(x^{k},y^{k})=0,\\[1ex]
	\nabla_{y}F(x^{k},y^{k})+\nabla_{y}\mathcal{L}(x^{k},y^{k},u^{k})^{\top}
	\beta^{k}+\sum\limits_{i=1}^{q}\left(  \xi_{1i}^{k}\delta_{i}^{k}-\gamma
	_{i}^{k}\right)  \nabla_{y}g_{i}(x^{k},y^{k})=0,\\[1ex]
	\nabla_{y}g_{i}(x^{k},y^{k})\beta^{k}+\mu_{i}^{k}-(2-\xi_{1i}^{k})\delta
	_{i}^{k}=0,\text{ \ \ }\forall i=1,\ldots,q
	\end{array}
	\right.  \label{t2}
	\end{equation}
	hold together with the following  conditions:
	\begin{equation}
	\left\{
	\begin{array}
	[c]{l}
	\forall j=1,\ldots,p:\;\;\alpha_{j}^{k}\geq0,\;\;G_{j}(x^{k})\leq
	0,\;\;\alpha_{j}^{k}G_{j}(x^{k})=0,\\[1ex]
	\forall i=1,\ldots,q:\;\mu_{i}^{k}\geq0,\;\ \gamma_{i}^{k}\geq0,\,\,\,\delta
	_{i}^{k}\geq0,\\[1ex]
	\forall i=1,\ldots,q:\,\mu_{i}^{k}u_{i}^{k}=0,\ \ \gamma_{i}^{k}g_{i}
	(x^{k},y^{k})=0,\;\;\delta_{i}^{k}\varphi^{t_{k}}_{i,SU}(x^{k},y^{k},u^{k})=0.
	\end{array}
	\right.  \label{t3}
	\end{equation}
	Let for all $k\in
		\mathbb{N}
	$,
	\begin{equation}
	\varepsilon_{1i}^{k}:=\xi_{1i}^{k}\delta_{i}^{k}\text{ and }\varepsilon
	_{2i}^{k}:=(2-\xi_{1i}^{k})\delta_{i}^{k}.\label{t4}
	\end{equation}
	Noting that
	\begin{align}
	\alpha_{i}^{k} &  \geq0,\text{ \ }\mathrm{supp}\alpha^{k}\subset I_{G}
	(x^{k}),\medskip\label{t6}\\
	\mu_{i}^{k} &  \geq0,\text{ \ }\mathrm{supp}\mu^{k}\subset I_{u}(x^{k}
	,y^{k},u^{k}),\medskip\label{t7}\\
	\gamma_{i}^{k} &  \geq0,\text{ \ }\mathrm{supp}\gamma^{k}\subset I_{g}
	(x^{k},y^{k},u^{k}),\medskip\label{t8}\\
	\delta_{i}^{k} &  \geq0,\text{ \ }\mathrm{supp}\delta^{k}\subset I_{\phi^{t_{k}} _{SU}
	}(x^{k},y^{k},u^{k})\label{t9}
	\end{align}
	and
	\begin{align}
	\varepsilon_{1i}^{k}  & >0 \quad \Longleftrightarrow \quad \xi_{1i}^{k}>0\text{ and }\delta
	_{i}^{k}>0,\label{t5}\\[1ex]
	\varepsilon_{2i}^{k}  & >0\quad \Longleftrightarrow \quad \xi_{1i}^{k}<2\text{ and }\delta
	_{i}^{k}>0,\label{t5*}
	\end{align}
	so
	\begin{align}
	\mathrm{supp}\varepsilon_{1}^{k}  & \subset I_{\varphi^{t_{k}} _{SU}
	}(x^{k},y^{k},u^{k})\backslash I_{u}(x^{k},y^{k},u^{k}),\label{T5}\\[1ex]
	\mathrm{supp}\varepsilon_{2}^{k}  & \subset I_{\varphi^{t_{k}} _{SU}
	}(x^{k},y^{k},u^{k})\backslash I_{g}(x^{k},y^{k},u^{k}),\label{T6}
	\end{align}
	for all $k$ large enough. Hence, the system of optimality conditions in
	(\ref{t2}) becomes
	\begin{equation}
	\left\{
	\begin{array}[c]{l}%
	\nabla_{x}F(x^{k},y^{k})+\nabla G(x^{k})^{T}\alpha^{k}+\nabla_{x}%
	\mathcal{L}(x^{k},y^{k},u^{k})^{T}\beta^{k}+\sum\limits_{i=1}^{q}\lambda
	_{1i}^{k}\nabla_{x}g_{i}(x^{k},y^{k})=0,\\[1ex]
	\nabla_{y}F(x^{k},y^{k})+\nabla_{y}\mathcal{L}(x^{k},y^{k},u^{k})^{T}\beta
	^{k}+\sum\limits_{i=1}^{q}\lambda_{1i}^{k}\nabla_{y}g_{i}(x^{k},y^{k})=0,
	\end{array}
	\right.  \label{t11}%
	\end{equation}
	and
	\begin{equation}
	\nabla_{y}g_{i}(x^{k},y^{k})\beta^{k}=\lambda_{2i}^{k}\label{T12}%
	\end{equation}
	with $\lambda_{1i}^{k}:=\xi_{1i}^{k}\delta_{i}^{k}-\gamma_{i}^{k}$ and
	$\lambda_{2i}^{k}:=-\mu_{i}^{k}+(2-\xi_{1i}^{k})\delta_{i}^{k}$. 
	  As $(t_{k})\downarrow0$ and  $u^{k}\rightarrow\bar{u}$ (the sequence $(y^{k}
	,u^{k})_{k}$ (along a subsequence) converges to some point $(\bar{y},\bar
	{u})\in\mathcal{S}_{p}(\bar{x})$), 
	 then for sufficiently large $k\in\mathbb{N}$ it holds
	$u_{i}^{k}>t_{k}$ for all $i\notin I_{u}(\bar{x},\bar{y},\bar{u})=\theta
	\cup\eta$. Hence, by the feasibility of $(y^{k},u^{k})_{k}$, 	$i\in I_{g}(x^{k},y^{k},u^{k})\cap I_{\varphi^{t_{k}} _{SU}
	}(x^{k},y^{k},u^{k})$ for all $i\notin\theta\cup\eta.$ Similarly, we get $i\in I_{u}
	(x^{k},y^{k},u^{k})\cap I_{\varphi^{t_{k}} _{SU}
	}(x^{k},y^{k},u^{k})$ for all
	$i\notin\theta\cup\nu=I_{g}(\bar{x},\bar{y},\bar{u}).$ Therefore, for
	sufficiently large $k\in\mathbb{N}$,
	\[
	(\mu_{i}^{k},\xi_{1i}^{k})=(0,2)\text{\ \ \ }\forall i\notin\theta\cup
	\eta\text{ \ and \ }(\gamma_{i}^{k},\xi_{1i}^{k})=0\text{ \ \ }\forall
	i\notin\theta\cup\nu
	\]
	so that 
    $
	\lambda_{1i}^{k}=0\text{ }\forall i\notin\theta\cup\nu\text{\ \ and \ }%
	\lambda_{2i}^{k}=0\text{ }\forall i\notin\theta\cup\eta.
    $
    Next, we are going to show that the sequence $(\chi_{k}):=
	\begin{pmatrix}
	\alpha^{k},\beta^{k},\lambda_{1}^{k},\lambda_{2}^{k}
	\end{pmatrix}
	_{k}$ is bounded. Assume that this is not the case and consider the sequence
	$\left(  \dfrac{\chi_{k}}{\left\Vert \chi_{k}\right\Vert }\right)  _{k}$ which
	converges to a nonvanishing vector $\chi:=(\bar{\alpha},\bar{\beta}
	,\bar{\lambda}_{1},\bar{\lambda}_{2})$ (otherwise, a suitable subsequence is
	chosen). Dividing (\ref{t11}) by $\left\Vert \chi_{k}\right\Vert $ and taking
	the limit $k\rightarrow\infty$ while taking into account the continuous
	differentiability of all involved functions, we come up with
	\begin{align}
	\nabla G(\bar{x})^{T}\bar{\alpha}+\nabla_{x}\mathcal{L}(\bar{x},\bar{y}
	,\bar{u})^{T}\bar{\beta}+\sum\limits_{i=1}^{q}\bar{\lambda}_{1i}\nabla
	_{x}g_{i}(\bar{x},\bar{y})  & =0,\bigskip\label{t13}\\
	\mbox{ and }\nabla_{y}\mathcal{L}(\bar{x},\bar{y},\bar{u})^{T}\bar{\beta}
	+\sum\limits_{i=1}^{q}\bar{\lambda}_{1i}\nabla_{y}g_{i}(\bar{x},\bar{y})  &
	=0,\label{t14}%
	\end{align}
	where, clearly, $\bar{\lambda}_{1\mid\eta}=0$ and $\bar{\lambda}_{2\mid\nu}=0.$
	We prove now that for all $i\in\theta$: 
	\[
	(\bar{\lambda}_{1i}<0\wedge\nabla_{y}g_{i}(\bar{x},\bar{y})\bar{\beta}
	<0)\vee\bar{\lambda}_{1i}\nabla_{y}g_{i}(\bar{x},\bar{y})\bar{\beta}=0. 
	\]
Let $\bar{\lambda}_{1i}>0$ or $\nabla_{y}
	g_{i}(\bar{x},\bar{y})\bar{\beta}>0$ for some $i\in\theta$ with $\bar{\lambda
	}_{1i}\nabla_{y}g_{i}(\bar{x},\bar{y})\bar{\beta}\neq0.$ If $\bar{\lambda
	}_{1i}>0$ then for all $k$ large enough, $\lambda_{1i}^{k}=\varepsilon
	_{1i}^{k}-\gamma_{i}^{k}>0$ so $\varepsilon_{1i}^{k}>0.$ Now using (\ref{t5})
	and (\ref{T5}) and the feasibility of $(x^{k},y^{k},u^{k})$,  $\xi
	_{1i}^{k}=2$ and $\mu_{i}^{k}=0$ since 
    \[i\in I_{g}(x^{k},y^{k},u^{k})\cap I_{\varphi^{t_{k}} _{SU}
	}(x^{k},y^{k},u^{k})\backslash I_{u}(x^{k}
	,y^{k},u^{k})
    \]
    so $\lambda_{2i}^{k}=0$. Thus, $\nabla_{y}g_{i}
	(\bar{x},\bar{y})\bar{\beta}=0$,  which contradicts our assumption. Suppose now
	that $\nabla_{y}g_{i}(\bar{x},\bar{y})\bar{\beta}>0$, that is $\bar{\lambda
	}_{2,i}>0$ then for any $k$ large, $\bar{\lambda}_{2i}^{k}=-\mu_{i}^{k}
	+\varepsilon_{2i}^{k}>0$ this implies that $\varepsilon_{2i}^{k}>0.$ And from
	(\ref{t5*}) and (\ref{T6}), $\xi_{1i}^{k}=0$ and $\gamma_{i}^{k}=0$ (since
	$i\in I_{u}(x^{k},y^{k},u^{k})\cap I_{\varphi^{t_{k}} _{SU}
	}(x^{k},y^{k},u^{k})\backslash I_{g}(x^{k},y^{k},u^{k})$) and this leads to a
	contradiction too.
	Consequently, 
	\[
	\left\{
	\begin{array}
	[c]{l}
	\nabla_{y}\mathcal{L}(\bar{x},\bar{y},\bar{u})^{T}\bar{\beta}+\nabla_{y}
	g^{T}(\bar{x},\bar{y})\bar{\lambda}_{1}=0,\\[1ex]
	\nabla_{y}g_{\nu}(\bar{x},\bar{y})\bar{\beta}=\bar{\lambda}_{2\mid\nu
	}=0,\text{ \ }\bar{\lambda}_{1\mid\eta}=0,\\[1ex]
	(\bar{\lambda}_{1i}<0\wedge\nabla_{y}g_{i}(\bar{x},\bar{y})\bar{\beta}
	<0)\vee\bar{\lambda}_{1i}\nabla_{y}g_{i}(\bar{x},\bar{y})\bar{\beta
	}=0\;\mbox{ for }\;i\in\theta,
	\end{array}
	\right.
	\]
	so that $(-\bar{\beta},-\bar{\lambda}_{1})\in\Lambda_{y}^{em}(\bar{x},\bar
	{y},\bar{u},0)$ and by condition $(A_{2}^{m})$, we get
	$
	\nabla_{x}\mathcal{L}(\bar{x},\bar{y},\bar{u})^{T}\bar{\beta}+\nabla_{y}%
	g^{T}(\bar{x},\bar{y})\bar{\lambda}_{1}=0
	$
	using (\ref{t13}), it yields $\nabla G(\bar{x})^{T}\bar{\alpha}=0$; this
	implies that $\bar{\alpha}=0$. Thus,
	\[
	\left\{
	\begin{array}
	[c]{l}
	\nabla_{x,y}\mathcal{L}(\bar{x},\bar{y},\bar{u})^{T}\bar{\beta}+\nabla
	_{y}g^{T}(\bar{x},\bar{y})\bar{\lambda}_{1}=0,\\[1ex]
	\nabla_{y}g_{\nu}(\bar{x},\bar{y})\bar{\beta}=0,\text{ \ }\bar{\lambda}
	_{1\mid\eta}=0,\\[1ex]
	(\bar{\lambda}_{1i}<0\wedge\nabla_{y}g_{i}(\bar{x},\bar{y})\bar{\beta}
	<0)\vee\bar{\lambda}_{1i}\nabla_{y}g_{i}(\bar{x},\bar{y})\bar{\beta
	}=0\;\mbox{ for }\;i\in\theta,
	\end{array}
	\right.
	\]
	that is, $(-\beta,-\bar{\lambda}_{1})\in\Lambda^{em}(\bar{x},\bar{y},\bar
	{u},0)$ and by $(A_{1}^{m}),$ we get $(-\beta,-\bar{\lambda}
	_{1})=0.$ Hence $(\bar{\alpha},\bar{\beta},\bar{\lambda}_{1},\bar{\lambda}
	_{2})=0$ (here $\bar{\lambda}_{2i}^{k}=\nabla_{y}g_{i}(x^{k},y^{k})\beta
	^{k}\rightarrow\nabla_{y}g_{i}(\bar{x},\bar{y})\bar{\beta}=0$) which is a contradiction. 	Consequently, the sequence $
	\begin{pmatrix}
	\alpha^{k},\beta^{k},\lambda_{1}^{k},\lambda_{2}^{k}
	\end{pmatrix}
	_{k}$ is bounded. Consider $(\bar{\alpha},\overline{\beta},\bar{\lambda}
	_{1},\bar{\lambda}_{2})$ its limit (up to a subsequence) so that
$\mathrm{supp}\bar{\alpha}\subset I_{G}$, $\bar{\lambda}_{1\mid\eta
	}=0$, and $\bar{\lambda}_{2\mid\nu}=0$. 
	Taking the limit in (\ref{t11}), we obtain
	\[
	\left\{
	\begin{array}
	[c]{l}
	\nabla_{x}F(\bar{x},\bar{y})+\nabla G(\bar{x})^{T}\bar{\alpha}+\nabla
	_{x}\mathcal{L}(\bar{x},\bar{y},\bar{u})^{T}\overline{\beta}+\nabla_{x}
	g^{T}(\bar{x},\bar{y})\bar{\lambda}_{1}=0,\\[1ex]
	\nabla_{y}F(\bar{x},\bar{y})+\nabla_{y}\mathcal{L}(\bar{x},\bar{y},\bar
	{u})^{T}\overline{\beta}+\nabla_{x}g^{T}(\bar{x},\bar{y})\bar{\lambda}_{1}=0,
	\end{array}
	\right.
	\]
	with
	\[
	\bar{\alpha}\geq0,\text{ \ }G(\bar{x})\leq0,\text{  }\bar{\alpha}^{\top}
	G(\bar{x})=0, \text{ \ and  }\nabla_{y}g_{\nu}(\bar{x},\bar{y})\overline{\beta
	}=\bar{\lambda}_{2\mid\nu}=0.
	\]
	Let us prove now that for any $i\in\theta,$ $\bar{\lambda}_{1i}\nabla_{y}
	g_{i}(\bar{x},\bar{y})\overline{\beta}\geq0.$ Assume that there is $i\in
	\theta$ such that $\bar{\lambda}_{1i}<0$ and $\nabla_{y}g_{i}(\bar{x},\bar
	{y})\overline{\beta}>0.$ So for $k$ large enough  $\lambda_{1i}^{k}
	=\varepsilon_{1i}^{k}-\gamma_{i}^{k}<0$ and $\lambda_{2i}^{k}=\varepsilon
	_{2i}^{k}-\mu_{i}^{k}>0$ so that $\varepsilon_{2i}^{k}>0$ since $\mu_{i}
	^{k}\geq0.$ Hence from (\ref{T6}), $\gamma_{i}^{k}=0$ then $\varepsilon
	_{1i}^{k}<0$ which impossible. Assume now that $\bar{\lambda}_{1i}>0$ and
	$\nabla_{y}g_{i}(\bar{x},\bar{y})\overline{\beta}<0.$ Similarly, we get that
	for $k$ large, $\lambda_{1i}^{k}=\varepsilon_{1i}^{k}-\gamma_{i}^{k}>0$ and
	$\lambda_{2i}^{k}=\varepsilon_{2i}^{k}-\mu_{i}^{k}<0$ so that $\varepsilon
	_{1i}^{k}>0.$ This implies from (\ref{T5}) that $\mu_{i}^{k}=0$ and so
	$\varepsilon_{2i}^{k}<0$ which contradicts the fact that $\varepsilon_{2i}
	^{k}\geq0$ for any $k.$ Consequently,
	$\bar{\lambda}_{1i}\nabla_{y}g_{i}(\bar{x},\bar{y})\overline{\beta}\geq0$ for all $i\in\theta$; i.e., $\bar{x}$ is a C-stationary point. \\

\noindent \textbf{\underline{Case $ \mathcal{R}$= KDB}:} Also in this case, we need  the following index subsets  defined for
$(x,(y,u))\in X\times\mathcal{D}_{KDB}^{t}(x)$ ($t>0$):
\begin{align*}
I_{u}(x,y,u,t)  & :=\{i:  u_{i}+t=0\},\\
I^{u}(x,y,u,t)  & :=\{i: u_{i}-t=0\},\\
I^{u+}(x,y,u,t)  & :=\{i: u_{i}-t>0,\text{ }g_{i}(x,y)+t=0\},\\
I^{u-}(x,y,u,t)  & :=\{i: u_{i}-t<0,\text{\ }g_{i}(x,y)+t=0\},\\
I_{g}(x,y,u,t)  & :=\{i: g_{i}(x,y)-t=0\},\\
I^{g}(x,y,u,t)  & :=\{i: g_{i}(x,y)+t=0\},\\
I^{g+}(x,y,u,t)  & :=\{i: u_{i}-t=0,\text{ }g_{i}(x,y)+t>0\},\\
I^{g-}(x,y,u,t)  & :=\{i: u_{i}-t=0,\text{ }g_{i}(x,y)+t<0\},\\
I^{ug}(x,y,u,t)  & :=\{i: u_{i}-t=g_{i}(x,y)+t=0\}.
\end{align*}
From the assumptions of Theorem \ref{ConvergenceResult} for $\mathcal{R}=KDB$, there exist a vector
	$(y^{k},u^{k})\in\mathcal{S}_{KDB}^{t_{k}}(x^{k})$ and multipliers
	$(\alpha^{k},\beta^{k},\mu^{k},\gamma^{k},\delta^{k})$ such that the following
	properties (\ref{k3}), (\ref{k4}), and (\ref{k5}) hold: 
	\begin{equation}
	\left\{
	\begin{array}
	[c]{c}
	\nabla_{x}F(x^{k},y^{k})+\nabla G^{T}(x^{k})\alpha^{k}-\nabla_{x}
	\mathcal{L}(x^{k},y^{k},u^{k})^{T}\beta^{k}+\sum\limits_{i=1}^{q}(\delta
	_{i}^{k}(u_{i}^{k}-t_{k})-\gamma_{i}^{k})\nabla_{x}g_{i}(x^{k},y^{k})=0,\\[1ex]
	\multicolumn{1}{l}{\nabla_{y}F(x^{k},y^{k})-\nabla_{y}\mathcal{L}(x^{k}
		,y^{k},u^{k})^{T}\beta^{k}+\sum\limits_{i=1}^{q}(\delta_{i}^{k}(u_{i}
		^{k}-t_{k})-\gamma_{i}^{k})\nabla_{y}g_{i}(x^{k},y^{k})=0,}\\[1ex]
	\multicolumn{1}{l}{-\nabla_{y}g_{i}(x^{k},y^{k})\beta^{k}+\mu_{i}^{k}
		+\delta_{i}^{k}(g_{i}(x^{k},y^{k})+t_{k})=0,\;i=1,\ldots,q}
	\end{array}
	\right. \label{k3}
	\end{equation}
	with
	\begin{equation}
	\alpha_{i}^{k}\geq0,\text{ }\alpha_{i}^{k}G_{i}(x^{k}
	)=0\ \ \ i=1,\ldots,p\label{k4}
	\end{equation}
	and for $i=1,\ldots,q$,
	\begin{equation}
	\left\{
	\begin{array}[l]{l}
	\mu_{i}^{k}\geq0\text{, }\mu_{i}^{k}(u_{i}^{k}+t_{k})=0,\\[1ex]
	\gamma_{i}^{k}\geq0,\text{ }\gamma_{i}^{k}(g_{i}(x^{k},y^{k})-t_{k}
	)=0,\\[1ex]
	\delta_{i}^{k}\geq0,\text{ }\delta_{i}^{k}(u_{i}^{k}-t_{k})(g_{i}(x^{k}
	,y^{k})+t_{k})=0.
	\end{array}
	\right. \label{k5}
	\end{equation}
	Putting
	\begin{equation}
	\xi_{i}^{u,k}:=\delta_{i}^{k}(g_{i}(x^{k},y^{k})+t_{k}),\text{ \ \ \ }\xi
	_{i}^{g,k}:=-\delta_{i}^{k}(u_{i}^{k}-t_{k}), \label{kdb6}
	\end{equation}
	 $\xi_{i}^{u,k}>0$ and $\xi_{i}^{g,k}>0$ whenever $i\in
	\mathrm{supp}\xi^{u,k}\cap\mathrm{supp}\gamma^{k}$ and $i\in\mathrm{supp}
	\xi^{g,k}\cap\mathrm{supp}\mu^{k}$ respectively, we also have 
	\begin{align}
	\mathrm{supp}\mu^{k}\cap\mathrm{supp}\xi^{u,k}  & =\varnothing,\medskip
	\label{kd11}\\
	\mathrm{supp}\gamma^{k}\cap\mathrm{supp}\xi^{g,k}  & =\varnothing
	,\medskip\label{kd12}\\
	\mathrm{supp}\xi^{u,k}\cap\mathrm{supp}\xi^{g,k}  & =\varnothing.\label{kd13}%
	\end{align}
	Also one can check that for $k$ sufficiently large,
	\begin{align}
	\mathrm{supp}\mu^{k} &  \subset I_{u}(x^{k},y^{k},u^{k},t_{k})\subset
	\theta\cup\eta,\medskip\label{k7}\\
	\mathrm{supp}\gamma^{k} &  \subset I_{g}(x^{k},y^{k},u^{k},t_{k})\subset
	\theta\cup\nu,\medskip\label{k8}\\
	\mathrm{supp}\xi^{u,k} &  \subset I^{u}(x^{k},y^{k},u^{k},t_{k})\subset
	\theta\cup\eta,\medskip\label{k9}\\
	\mathrm{supp}\xi^{g,k} &  \subset I^{g}(x^{k},y^{k},u^{k},t_{k})\subset
	\theta\cup\nu.\label{k10}
	\end{align}
	By (\ref{kdb6}) and setting $\tilde{\beta}^{k}:=-\beta^{k},$ the optimality
	conditions become as follows:
	\begin{equation}
	\left\{
	\begin{array}
	[c]{c}%
	\nabla_{x}F(x^{k},y^{k})+\nabla G^{T}(x^{k})\alpha^{k}+\nabla_{x}
	\mathcal{L}(x^{k},y^{k},u^{k})^{T}\tilde{\beta}^{k}-\sum\limits_{i=1}^{q}
	(\xi_{i}^{g,k}+\gamma_{i}^{k})\nabla_{x}g_{i}(x^{k},y^{k})=0,\\[1ex]
	\multicolumn{1}{l}{\nabla_{y}F(x^{k},y^{k})+\nabla_{y}\mathcal{L}(x^{k}
		,y^{k},u^{k})^{T}\tilde{\beta}^{k}-\sum\limits_{i=1}^{q}(\xi_{i}^{g,k}
		+\gamma_{i}^{k})\nabla_{y}g_{i}(x^{k},y^{k})=0}
	\end{array}
	\right. \label{kd16}
	\end{equation}
	with%
	\begin{equation}
	\nabla_{y}g_{i}(x^{k},y^{k})\tilde{\beta}^{k}=\left\{
	\begin{tabular}
	[c]{lll}%
	$-\mu_{i}^{k}$ & if & $i\in\mathrm{supp}\mu^{k}$,\\
	$-\xi_{i}^{u,k}$ & if & $i\in\mathrm{supp}\xi^{u,k}$,\\
	$0$ & if & $i\notin(\mathrm{supp}\mu^{k}\cup\mathrm{supp}\xi^{u,k}).$
	\end{tabular}
	\right. \label{kd17}
	\end{equation}
	 Now let us prove that the sequence
	$(\chi_{k}):=
	\begin{pmatrix}
	\alpha^{k},\tilde{\beta}^{k},\mu^{k},\gamma^{k},\xi^{u,k},\xi^{g,k}
	\end{pmatrix}
	_{k}$ is bounded. To this end, we assume that this is not the case.
	Nevertheless, we may suppose, without loss of generality, that there is a
	nonvanishing vector $(\alpha,\tilde{\beta},\mu,\gamma,\xi^{u},\xi^{g})$ such
	that the sequence $\left(\dfrac{\chi_{k}}{\left\Vert \chi_{k}\right\Vert }\right)_{k}$
	converges to $(\alpha,\tilde{\beta},\mu,\gamma,\xi^{u},\xi^{g})$.
	Hence, $\alpha_{i}\geq0,\ \alpha_{i}G_{i}(\bar{x})=0\ (\forall
	i=1,..,p$) and for $k$ sufficiently large,%
	\begin{align*}
	\mathrm{supp}\alpha & \subset\mathrm{supp}\alpha^{k},\text{ }\mathrm{supp}%
	\mu\subset\mathrm{supp}\mu^{k},\text{ }\mathrm{supp}\gamma\subset
	\mathrm{supp}\gamma^{k},\\[1ex]
	\mathrm{supp}\xi^{u}  & \subset\mathrm{supp}\xi^{u,k},\text{ }\mathrm{supp}%
	\xi^{g}\subset\mathrm{supp}\xi^{g,k},
	\end{align*}
	and observe also that $\mathrm{supp}\xi^{u}\cap\mathrm{supp}\xi^{g}=\varnothing.$
	Taking the limit as $k\rightarrow\infty$ in the system of optimality
	conditions after dividing by $\left\Vert \chi_{k}\right\Vert $,
	\begin{equation}
	\left\{
	\begin{array}
	[c]{l}
	\nabla G(\bar{x})^{T}\alpha+\nabla_{x}\mathcal{L}(\bar{x},\bar{y},\bar{u}%
	)^{T}\tilde{\beta}-\sum\limits_{i=1}^{q}(\xi_{i}^{g}+\gamma_{i})\nabla
	_{x}g_{i}(\bar{x},\bar{y})=0,\\[1ex]
	\nabla_{y}\mathcal{L}(\bar{x},\bar{y},\bar{u})^{T}\tilde{\beta}-\sum
	\limits_{i=1}^{q}(\xi_{i}^{g}+\gamma_{i})\nabla_{y}g_{i}(\bar{x},\bar
	{y})=0,\\[1ex]
	\nabla_{y}g_{i}(\bar{x},\bar{y})\tilde{\beta}+\mu_{i}+\xi_{i}^{u}=0\text{
		\ \ \ \ \ }\forall i=1,\ldots,q.
	\end{array}
	\right. \label{kd19}%
	\end{equation}
	Consider now the following new multipliers by setting%
	\begin{equation}
	\tilde{\gamma}_{i}:=\left\{
	\begin{array}
	[c]{ll}%
	-\gamma_{i} & \mbox{ if }\text{ }i\in\mathrm{supp}\gamma,\medskip\\
	-\xi_{i}^{g} & \mbox{ if }\text{ }i\in\mathrm{supp}\xi^{g},\medskip\\
	0 & \mbox{otherwise}
	\end{array}
	\right.  \text{ \ and \ }\tilde{\mu}_{i}:=\left\{
	\begin{array}
	[c]{ll}
	\mu_{i} & \mbox{ if }\text{ }i\in\mathrm{supp}\mu,\medskip\\
	\xi_{i}^{u} & \mbox{ if }\text{ }i\in\mathrm{supp}\xi^{u},\medskip\\
	0 & \mbox{otherwise}.
	\end{array}
	\right. \label{kd21}
	\end{equation}
	Thanks to properties (\ref{kd11}), (\ref{kd12}) and (\ref{kd13}), we can
	rewrite (\ref{kd19}) as:
	\begin{equation}
	\left\{
	\begin{array}
	[c]{l}
	\nabla G(\bar{x})^{T}\alpha+\nabla_{x}\mathcal{L}(\bar{x},\bar{y},\bar{u}
	)^{T}\tilde{\beta}+\sum\limits_{i=1}^{q}\tilde{\gamma}_{i}\nabla_{x}g(\bar{
		x},\bar{y})=0,\\[1ex]
	\nabla_{y}\mathcal{L}(\bar{x},\bar{y},\bar{u})^{T}\tilde{\beta} +\sum
	\limits_{i=1}^{q}\tilde{\gamma}_{i}\nabla_{y}g(\bar{x},\bar{y} )=0,\\[1ex]
	\nabla_{y}g_{i}(\bar{x},\bar{y})\tilde{\beta}=-\tilde{\mu}_{i}\text{
		\ \ \ \ \ }\forall i=1,\ldots,q.
	\end{array}
	\right. \label{kd21*}
	\end{equation}
	On the other hand, we have $\tilde{\gamma}_{\eta}=0$ since $i\notin
	\mathrm{supp}\gamma\cup\mathrm{supp}\xi^{g}$ when $i\in\eta$ and we have
	$\nabla_{y}g_{\nu}(\bar{x},\bar{y})\tilde{\beta}=0$ since $i\notin
	\mathrm{supp}\mu\cup\mathrm{supp}\xi^{u}$ whenever $i\in v$. Let us prove now
	that for all $i\in\theta$,
	\[
	\tilde{\gamma}_{i}\nabla_{y}g_{i}(\bar{x},\bar{y})\tilde{\beta}\geq0.
	\]
	Assume that it is not the case for some $i\in\theta.$ If $\tilde{\gamma}
	_{i}>0$ and $\nabla_{y}g_{i}(\bar{x},\bar{y})\tilde{\beta}<0$ thus
	$i\in\mathrm{supp}\xi^{g}\cap\mathrm{supp}\mu$ (since $\mathrm{supp}\xi
	^{g}\cap\mathrm{supp}\xi^{u}=\emptyset$)$.$ The case implies that for $k $
	large enough, $i\in\mathrm{supp}\xi^{g,k}\cap\mathrm{supp}\mu^{k}$ so $\xi
	_{i}^{g,k}>0$ thus $\tilde{\gamma}_{i}=-\xi_{i}^{g}\leq0$ and the
	contradiction follows. Similarly, if $\tilde{\gamma}_{i}<0$ and
	$\nabla_{y}g_{i}(\bar{x},\bar{y})\tilde{\beta}>0$ so from (\ref{kd21})
	-(\ref{kd21*}), $i\in\mathrm{supp}\gamma\cap\mathrm{supp}\xi^{u}$ so that, for
	$k$ large enough, $i\in\mathrm{supp}\gamma^{k}\cap\mathrm{supp}\xi^{u,k}$ thus
	$\xi_{i}^{u,k}>0$ which implies that $\nabla_{y}g_{i}(\bar{x},\bar{y}
	)\tilde{\beta}=-\xi_{i}^{u}\leq0.$ We conclude that $\tilde{\gamma}_{i}\leq0$
	and $\nabla_{y}g_{i}(\bar{x},\bar{y})\tilde{\beta}\leq0$ for all $i\in\theta.$
	Consequently, 
	\[
	\left\{
	\begin{array}
	[c]{l}
	\nabla_{y}\mathcal{L}(\bar{x},\bar{y},\bar{u})^{T}\tilde{\beta}+\nabla
	_{y}g^{T}(\bar{x},\bar{y})\tilde{\gamma}=0,\medskip\\
	\nabla_{y}g_{v}(\bar{x},\bar{y})\tilde{\beta}=0,\text{ \ }\gamma_{\eta
	}=0,\medskip\\
	\tilde{\gamma}_{i}\nabla_{y}g_{i}(\bar{x},\bar{y})\tilde{\beta}\geq0\text{
		\ \ \ }\forall i\in\theta,
	\end{array}
	\right.
	\]
	so that, $(-\tilde{\beta},-\tilde{\gamma})\in\Lambda_{y}^{ec}(\bar{x},\bar
	{y},\bar{u},0)$ and by condition $(A_{2}^{c}),$ we get
	$
	\nabla_{x}\mathcal{L}(\bar{x},\bar{y},\bar{u})^{T}\tilde{\beta}+\nabla
	_{x}g^{T}(\bar{x},\bar{y})\tilde{\gamma}=0.
	$
	Using (\ref{kd21*}), $\nabla G(\bar{x})^{T}\alpha=0$ so that
	$\alpha=0$ ($\bar{x}$ is upper-level regular). Thus it holds that
	\[
	\left\{
	\begin{array}
	[c]{l}%
	\nabla_{x,y}\mathcal{L}(\bar{x},\bar{y},\bar{u})^{T}\tilde{\beta}+\nabla
	g^{T}(\bar{x},\bar{y})\tilde{\gamma}=0,\medskip\\
	\nabla_{y}g_{v}(\bar{x},\bar{y})\tilde{\beta}=0,\text{ }\tilde{\gamma}_{\eta
	}=0,\medskip\\
	\tilde{\gamma}_{i}(\nabla_{y}g_{i}(\bar{x},\bar{y}))\tilde{\beta}\geq0\text{
		\ \ \ }\forall i\in\theta.
	\end{array}
	\right.
	\]
	That is, $(-\tilde{\beta},-\tilde{\gamma})\in\Lambda^{ec}(\bar{x},\bar{y}
	,\bar{u},0)$ so by the CQ condition $(A_{1}^{c}),$ $(-\tilde{\beta}
	,-\tilde{\gamma})=0.$ Hence $(\alpha,\tilde{\beta},\mu,\gamma,\xi^{u},\xi
	^{g})=0$ (observe that $(\mu,\xi^{u},\xi^{g})=0$ since $\tilde{\mu}
	=\tilde{\gamma}=0$), and the contradiction follows. Consequently, the sequence
	$(\alpha^{k},\tilde{\beta}^{k},\mu^{k},\gamma^{k},\xi^{u,k},\xi^{g,k})_{k}$ is
	bounded. Consider $(\bar{\alpha},\bar{\beta},\bar{\mu},\overline{\gamma}%
	,\bar{\xi}^{u},\bar{\xi}^{g})$ its limit (up to a subsequence) then for $k$
	large enough, the following inclusions are immediate
	\begin{equation}
	\left\{
	\begin{array}
	[c]{l}
	\mathrm{supp}\bar{\alpha}\subset\mathrm{supp}\alpha^{k},\text{ }
	\mathrm{supp}\bar{\mu}\subset\mathrm{supp}\mu^{k},\text{ }\mathrm{supp}
	\bar{\gamma}\subset\mathrm{supp}\gamma^{k}\text{,\medskip\medskip} \\[1ex]
	\mathrm{supp}\bar{\xi}^{u}\subset\mathrm{supp}\xi^{u,k}\text{ , }
	\mathrm{supp}\bar{\xi}^{g}\subset\mathrm{supp}\xi^{g,k}.
	\end{array}
	\right. \label{sup}
	\end{equation}
	Taking the limit in the system of optimality conditions, we obtain%
	\begin{equation}
	\left\{
	\begin{array}
	[c]{l}
	\nabla_{x}F(\bar{x},\bar{y})+\nabla G(\bar{x})^{T}\bar{\alpha}+\nabla
	_{x}\mathcal{L}(\bar{x},\bar{y},\bar{u})^{T}\bar{\beta}-\sum\limits_{i=1}
	^{q}(\overline{\gamma}_{i}+\bar{\xi}_{i}^{g})\nabla_{x}g_{i}(\bar{x},\bar
	{y})=0,\\[1ex]
	\nabla_{y}F(\bar{x},\bar{y})+\nabla_{y}\mathcal{L}(\bar{x},\bar{y},\bar
	{u})^{T}\bar{\beta}-\sum\limits_{i=1}^{q}(\overline{\gamma}_{i}+\bar{\xi}
	_{i}^{g})\nabla_{y}g_{i}(\bar{x},\bar{y})=0,\\[1ex]
	\nabla_{y}g_{i}(\bar{x},\bar{y})\bar{\beta}+\bar{\mu}_{i}+\bar{\xi}_{i}
	^{u}=0\text{ \ \ \ \ \ }\forall i=1,\ldots,q
	\end{array}
	\right. \label{kd25}
	\end{equation}
	with
	$\bar{\alpha}_{i}  \geq0$, $\bar{\alpha}_{i}G_{i}(\bar{x})=0$ for $i=1,..,p$, 
    $\overline{\gamma}_{\eta}  =0$, and $\nabla_{y}g_{v}(\bar{x},\bar{y})\bar{\beta}=0$.
	Consider now
	\[
	\lambda_{i}^{u}:=\left\{
	\begin{array}
	[c]{ll}
	\bar{\mu}_{i} & \mbox{ if }\text{ }i\in\mathrm{supp}\bar{\mu},\text{\medskip}\\
	\bar{\xi}_{i}^{u} & \mbox{ if }\text{ }i\in\mathrm{supp}\bar{\xi}%
	^{u},\text{\medskip}\\
	0 & \mbox{otherwise}
	\end{array}
	\right.  \text{ \ and\ \ }\lambda_{i}^{g}:=\left\{
	\begin{array}
	[c]{ll}%
	-\bar{\gamma}_{i} & \mbox{ if }\text{ }i\in\mathrm{supp}\bar{\gamma
	},\text{\medskip}\\
	-\bar{\xi}_{i}^{g} & \mbox{ if }\text{ }i\in\mathrm{supp}\bar{\xi}
	^{g},\text{\medskip}\\
	0 & \mbox{otherwise}
	\end{array}
	\right.
	\]
	for $i=1,\ldots,q$. So (\ref{kd25}) becomes
	\[
	\left\{
	\begin{array}
	[c]{l}
	\nabla_{x}F(\bar{x},\bar{y})+\nabla G(\bar{x})^{T}\bar{\alpha}+\nabla
	_{x}\mathcal{L}(\bar{x},\bar{y},\bar{u})^{T}\bar{\beta}+\sum\limits_{i=1}
	^{q}\lambda_{i}^{g}\nabla_{x}g_{i}(\bar{x},\bar{y})=0,\\[1ex]
	\nabla_{y}F(\bar{x},\bar{y})+\nabla_{y}\mathcal{L}(\bar{x},\bar{y},\bar
	{u})^{T}\bar{\beta}+\sum\limits_{i=1}^{q}\lambda_{i}^{g}\nabla_{y}g_{i}
	(\bar{x},\bar{y})=0,\\[1ex]
	\nabla_{y}g_{i}(\bar{x},\bar{y})\bar{\beta}=-\lambda_{i}^{u}\text{
		\ \ \ \ \ }\forall i=1,\ldots,q.
	\end{array}
	\right.
	\]
	It remains to prove that
	\[
	\left\{
	\begin{array}
	[c]{l}
	\nabla_{y}g_{\nu}(\bar{x},\bar{y})\bar{\beta}=0,\text{ \ \ \ }\lambda_{\eta
	}^{g}=0,\\[1ex]
	(\lambda_{i}^{g}<0\wedge\nabla_{y}g_{i}(\bar{x},\bar{y})\bar{\beta}
	<0)\vee\lambda_{i}^{g}\nabla_{y}g_{i}(\bar{x},\bar{y})\bar{\beta}=0\text{
		\ \ \ }\forall i\in\theta.
	\end{array}
	\right.
	\]
	Let $i\in\nu$ so from (\ref{k7}) and (\ref{k9},) $i\notin\mathrm{supp}\mu
	^{k}\cup\mathrm{supp}\xi^{u,k}$. Using (\ref{sup}), $i\notin\mathrm{supp}%
	\bar{\mu}\cup\mathrm{supp}\bar{\xi}^{u}$ which leads to $\lambda_{i}^{u}=0$
	and so is $\nabla_{y}g_{i}(\bar{x},\bar{y})\bar{\beta}.$ Similary when
	$i\in\eta,$ $i\notin\mathrm{supp}\gamma^{k}\cup\mathrm{supp}\xi^{g,k}$ by
	(\ref{k8}-\ref{k10}) which implies from (\ref{sup}) that $i\notin
	\mathrm{supp}\bar{\gamma}\cup\mathrm{supp}\bar{\xi}^{g}$ so that, $\lambda
	_{i}^{g}=0.$ Consider now $i\in\theta.$ If $\lambda_{i}^{u}=0$ or $\lambda
	_{i}^{g}=0,$ the condition of M-stationarty is trivial for such $i$. Assume
	that for some $i\in\theta,$ $\lambda_{i}^{g}>0$ so $i\in\mathrm{supp}\bar{\xi
	}^{g}.$ If in addition, $i\in\mathrm{supp}\bar{\mu}$ this implies from
	(\ref{sup}) that $i\in\mathrm{supp}\xi^{g,k}\cap\mathrm{supp}\mu^{k}$ for $k$
	sufficiently large which in turn implies that $\xi^{g,k}>0 $ for $k$ large
	enaugh$.$ Therefore $\bar{\xi}^{g}\geq0$ and then $\lambda_{i}^{g}\leq0$ so
	that such $i\notin\mathrm{supp}\bar{\mu}.$ If now that the considered index
	$i\in\mathrm{supp}\bar{\xi}^{u}$ so from (\ref{sup}), $i\in\mathrm{supp}%
	\xi^{u,k}\cap\mathrm{supp}\xi^{g,k}$ for $k$ sufficiently large which
	contradicts property (\ref{kd13}) and so such an index $i$ does not exists.
	Suppose now that for $i\in\theta,$ $\nabla_{y}g_{i}(\bar{x},\bar{y})\bar
	{\beta}>0$ that is, $\lambda_{i}^{u}<0$ so $i\in\mathrm{supp}\bar{\xi}^{u}.$
	If in addition $i\in\mathrm{supp}\bar{\gamma}$ so $i\in\mathrm{supp}\xi
	^{u,k}\cap\mathrm{supp}\gamma^{k}$ for $k$ sufficiently large which gives that
	$\xi^{u,k}>0.$ Thus $\bar{\xi}^{u}\geq0$ so that $\lambda_{i}^{u}\geq0$ and
	such index $i\notin\mathrm{supp}\bar{\gamma}$. Moreover, as $i\in
	\mathrm{supp}\bar{\xi}^{u},$ under properties (\ref{kd13}) and (\ref{sup}),
	such $i\notin\mathrm{supp}\bar{\xi}^{g}.$ Finally, we conclude that $\bar{x}$
	is a M-stationary point. 

${}$\\
\noindent \textbf{\underline{Case $ \mathcal{R}$= KS}:} Here also we need to introduce new index sets defined for
$(x,(y,u))\in X\times\mathcal{D}_{KS}^{t}(x)$ as
\begin{align*}
I_{G}(x,y,u)  & :=\{i: G_{i}(x)=0\},\\
I_{u}(x,y,u)  & :=\{i: u_{i}=0\},\\
I_{g}(x,y,u)  & :=\{i: g_{i}(x,y)=0\},\\
I_{\varphi^{t}_{KS}}(x,y,u)  & :=\{i: \varphi^{t}_{i,KS}(x,y,u)=0\},
\end{align*}
together with the following partition of the index set $I_{\varphi^{t}_{KS}}(x,y,u)$: 
\begin{align*}
I_{\varphi^{t}_{KS}}^{00}(x,y,u)  & :=\left\{i\in I_{\varphi^{t}_{KS}}(x,y,u,): u_{i}-t=0,\text{
}g_{i}(x,y)+t=0\right\},\\
I_{\varphi^{t}_{KS}}^{+0}(x,y,u)  & :=\left\{i\in I_{\varphi^{t}_{KS}}(x,y,u): u_{i}-t>0,\text{
}g_{i}(x,y)+t=0\right\},\\
I_{\varphi^{t}_{KS}}^{0+}(x,y,u)  & :=\left\{i\in I_{\varphi^{t}_{KS}}(x,y,u): u_{i}-t=0,\text{
}g_{i}(x,y)+t<0\right\}.
\end{align*}
       	Similarly from the stationality assumption of the point $x^{k}$ of \eqref{KKT-RG} for $t:=t^k$ and $\mathcal{R}=KS$,
       	we can find $(y^{k},u^{k})\in\mathcal{S}_{KS}^{t_{k}}(x^{k})$ and
       	multipliers $(\alpha^{k},\beta^{k},\mu^{k},\gamma^{k},\delta^{k})$ such that
       	\begin{equation}
       	\left\{
       	\begin{array}
       	[c]{l}
       	\nabla_{x}F(x^{k},y^{k})+\nabla G(x^{k})^{\top}\alpha^{k}+\nabla
       	_{x}\mathcal{L}(x^{k},y^{k},u^{k})^{\top}\beta^{k}-\sum\limits_{i=1}^{q}
       	\gamma_{i}^{k}\nabla_{x}g_{i}(x^{k},y^{k})\\
       	\qquad \qquad \qquad \qquad \qquad \qquad \qquad \qquad \qquad\qquad \qquad \qquad \quad \;\, -\sum\limits_{i=1}^{q}\delta_{i}^{k}\nabla_{x}\varphi_{{i}_{KS}}^{t_{k}}(x^{k},y^{k},u^{k})=0,\\
       	\nabla_{y}F(x^{k},y^{k})+\nabla_{y}\mathcal{L}(x^{k},y^{k},u^{k})^{\top}
       	\beta^{k}-\sum\limits_{i=1}^{q}\gamma_{i}^{k}\nabla_{y}g_{i}(x^{k},y^{k}
       	)
       	-\sum\limits_{i=1}^{q}\delta_{i}^{k}\nabla_{y}\varphi_{{i}_{KS}}^{t_{k}}(x^{k},y^{k},u^{k})=0,\medskip \\
       	\nabla_{y}g_{i}(x^{k},y^{k})\beta^{k}+\mu_{i}^{k}-\delta_{i}^{k}\nabla_{u}
       	\varphi_{{i}_{KS}}^{t_{k}}(x^{k},y^{k},u^{k})=0, \;\, i=1,\ldots,q,
       	\end{array}
       	\right. \label{k13bis}
       	\end{equation}
       	with the following complementarity conditions:
       	\begin{equation}
       	\left.
       	\begin{array}[c]{r}
       	\mu_{i}^{k}\geq0,\text{ }u_{i}^{k}\geq0\text{\ and }\mu_{i}^{k}u_{i}%
       	^{k}=0\\[1ex]%
       	\gamma_{i}^{k}\geq0,\text{ }g_{i}(x^{k},y^{k})\leq0\text{ \ and }\gamma
       	_{i}^{k}g_{i}(x^{k},y^{k})=0\\
       	\delta_{i}^{k}\geq0,\text{ }\varphi_{{i}_{KS}}^{t_{k}}(x^{k},y^{k},u^{k})\leq0\text{ \ and
       	}\delta_{i}^{k}\varphi_{{i}_{KS}}^{t_{k}}(x^{k},y^{k},u^{k})=0
       	\end{array}
       	\right\}  \text{ \ \ }\forall i=1,\ldots,q,\label{k14}
       	\end{equation}
       	together with
       	$
       	\alpha_{j}^{k}\geq0$, $G_{j}(x^{k})\leq 0$, and $\alpha_{j}^{k}
       	G_{j}(x^{k})=0$ for $j=1,\ldots, p,
       	$
        where the gradients of $\varphi$ are given by
\begin{equation}
\nabla_{x}\varphi^{t}_{i,KS}(x,y,u)=\left\{
\begin{tabular}
[c]{lll}
$-\nabla_{x}g_{i}(x,y)(u_{i}-t)$ & if & $u_{i}-g_{i}(x,y)\geq2t,$\\[1ex]
$-\nabla_{x}g_{i}(x,y)(g_{i}(x,y)+t)$ & if & $u_{i}-g_{i}(x,y)<2t,$
\end{tabular}
\right. \label{dx}
\end{equation}
\begin{equation}
\nabla_{y}\varphi^{t}_{i,KS}(x,y,u)=\left\{
\begin{tabular}
[c]{lll}
$-\nabla_{y}g_{i}(x,y)(u_{i}-t)$ & if & $u_{i}-g_{i}(x,y)\geq2t,$\\[1ex]
$-\nabla_{y}g_{i}(x,y)(g_{i}(x,y)+t)$ & if & $u_{i}-g_{i}(x,y)<2t,$
\end{tabular}
\right. \label{dy}
\end{equation}
and
\begin{equation}
\nabla_{u}\varphi^{t}_{i,KS}(x,y,u)=\left\{
\begin{tabular}
[c]{lll}
$-g_{i}(x,y)-t$ & if & $u_{i}-g_{i}(x,y)\geq2t,$\\[1ex]
$-(u_{i}-t)$ & if & $u_{i}-g_{i}(x,y)<2t.$
\end{tabular}
\right. \label{du}
\end{equation}
Using formulas (\ref{dx}), (\ref{dy}) and (\ref{du}), equation
       	(\ref{k13bis}) becomes
       	\begin{equation}
       	\left\{
       	\begin{array}
       	[c]{l}
       	\nabla_{x}F(x^{k},y^{k})+\nabla G(x^{k})^{\top}\alpha^{k}+\nabla
       	_{x}\mathcal{L}(x^{k},y^{k},u^{k})^{\top}\beta^{k}-\nabla_{x}g(x^{k}%
       	,y^{k})^{T}\gamma^{k}+\nabla_{x}g(x^{k},y^{k})^{T}\delta^{g,k}=0,\\[1ex]
       	\nabla_{y}F(x^{k},y^{k})+\nabla_{y}\mathcal{L}(x^{k},y^{k},u^{k})^{\top}%
       	\beta^{k}-\nabla_{y}g_{i}(x^{k},y^{k})^{T}\gamma^{k}+\nabla_{y}g_{i}%
       	(x^{k},y^{k})^{T}\delta^{g,k}=0,\\[1ex]
       	\nabla_{y}g_{i}(x^{k},y^{k})\beta^{k}+\mu_{i}^{k}-\delta_{i}^{u,k}=0\text{
       		\ \ \ }\forall i=1,\ldots,q
       	\end{array}
       	\right. \label{k15bis}
       	\end{equation}
       	where
       	\begin{equation}
       	\delta_{i}^{g,k}:=\left\{
       	\begin{tabular}
       	[c]{ll}
       	$\delta_{i}^{k}(u_{i}^{k}-t_{k})$ & if $i\in I_{\varphi^{t_{k}}_{KS}}^{+0}(x^{k},y^{k},u^{k})$,\\[1ex]
       	$0$ & otherwise
       	\end{tabular}
       	\right.  \label{k17}
       	\end{equation}
       	\text{ and }
       	\begin{equation}
       	\delta_{i}^{u,k}:=\left\{
       	\begin{tabular}
       	[c]{ll}
       	$-\delta_{i}^{k}(g_{i}(x^{k},y^{k})+t_{k})$ & if $i\in I_{\varphi^{t_{k}}_{KS}}^{0+}
       	(x^{k},y^{k},u^{k})$,\\[1ex]
       	$0$ & otherwise.
       	\end{tabular}
       	\right. \label{k17*}
       	\end{equation}
       	Observe that $\delta_{i}^{g,k},\delta_{i}^{u,k}\geq0$ for all $i,k$ and
       	\begin{equation}
       	\mathrm{supp}\delta^{g,k}\cap\mathrm{supp}\delta^{u,k}=\emptyset\label{15b}
       	\end{equation}
       	for all $k.$ 
        Let us prove that the sequence $(\chi_{k}):=
       	\begin{pmatrix}
       	\alpha^{k},\beta^{k},\mu^{k},\gamma^{k},\delta^{u,k},\delta^{g,k}
       	\end{pmatrix}_{k}$ is bounded. let's assume the opposite and consider the sequence$\left(  \dfrac{\chi_{k}}{\left\Vert \chi_{k}\right\Vert }\right)  _{k}$ which converges to a non vanishing vector $\chi:=(\alpha,\beta,\mu,\gamma,\delta^{u},\delta^{g})$ (otherwise, a suitable subsequence is chosen). So $\delta_{i}^{u}\geq0,$ $\delta_{i}^{g}\geq0$ and from (\ref{k15bis}), we obtain
       	\begin{align}
       	\nabla G(\bar{x})^{T}\alpha+\nabla_{x}\mathcal{L}(\bar{x},\bar{y},\bar{u}^{T}\beta-\nabla_{x}g(\bar{x},\bar{y})^{T}\gamma+\nabla_{x}g(\bar{x},\bar{y})^{T}\delta^{g}  & =0,\label{k18bis}\\
        \nabla_{y}\mathcal{L}(\bar{x},\bar{y},\bar{u})^{T}\beta-\nabla_{y}g(\bar{x},\bar{y})^{T}\gamma
        +\nabla_{y}g(\bar{x},\bar{y})^{T}\delta^{g} & =0, \label{k19bis}\\
       \forall i=1,\ldots, q: \;\; \nabla_{y}g_{i}(\bar{x},\bar{y})\beta+\mu_{i}-\delta_{i}^{u}  & =0.\label{k19b}
       \end{align}
       	Let
       	\begin{equation}
       	\tilde{\gamma}_{i}:=\left\{
       	\begin{tabular}
       	[c]{ll}
       	$-\gamma_{i}$ & if $i\in\mathrm{supp}\gamma$,\\
       	$\delta_{i}^{g}$ & if $i\in\mathrm{supp}\delta^{g}$,\\
       	$0$ & otherwise
       	\end{tabular}
       	\right.  \text{ \ and \ }\tilde{\mu}_{i}:=\left\{
       	\begin{tabular}
       	[c]{ll}%
       	$\mu_{i}$ & if $i\in\mathrm{supp}\mu$,\\
       	$-\delta_{i}^{u}$ & if $i\in\mathrm{supp}\delta^{u}$,\\
       	$0$ & otherwise,
       	\end{tabular}
       	\right. \label{k20}
       	\end{equation}
       	which are well define since we have
       	\begin{equation}
       	\mathrm{supp}\gamma\cap\mathrm{supp}\delta^{g}=\emptyset\text{ \ and
       		\ }\mathrm{supp}\mu\cap\mathrm{supp}\delta^{u}=\emptyset\label{k24}%
       	\end{equation}
       	from the fact that
       	\[
       	\left\{
       	\begin{array}
       	[c]{c}
       	I_{u}(x^{k},y^{k},u^{k})\cap I_{\varphi^{t_{k}}_{KS}}^{0+}(x^{k},y^{k},u^{k})=\emptyset,\\
       	I_{g}(x^{k},y^{k},u^{k})\cap I_{\varphi^{t_{k}}_{KS}}^{+0}(x^{k},y^{k},u^{k})=\emptyset,
       	\end{array}
       	\right.
       	\]
       	for all $k\in\mathbb{N}$. Hence, it holds that
       	\begin{align}
       	\nabla G(\bar{x})^{T}\alpha+\nabla_{x}\mathcal{L}(\bar{x},\bar{y},\bar{u}
       	)^{T}\beta+\nabla_{x}g^{T}(\bar{x},\bar{y})\tilde{\gamma}  & =0,\label{K22}\\
       	\nabla_{y}\mathcal{L}(\bar{x},\bar{y},\bar{u})^{T}\beta+\nabla_{y}g^{T}
       	(\bar{x},\bar{y})\tilde{\gamma}  & =0,\label{K23}\\
       	\forall i=1,\ldots, q: \;\; \nabla_{y}g_{i}(\bar{x},\bar{y})\beta+\tilde{\mu}_{i}  & =0 \label{K23b}
       	\end{align}
       	with
       	\begin{equation}
       	\mathrm{supp}\alpha\subset I_{G}(x^{k})\subset I_{G},\label{K26}
       	\end{equation}%
       	\begin{equation}
       	\mathrm{supp}\tilde{\mu}=\mathrm{supp}\mu\cup\mathrm{supp}\delta^{u}\subset
       	I_{u}(x^{k},y^{k},u^{k})\cup I_{\varphi^{t_{k}}_{KS}}^{0+}(x^{k},y^{k},u^{k}
     )\subset\theta\cup\eta\label{K27}
       	\end{equation}
       	and
       	\begin{equation}
       	\mathrm{supp}\tilde{\gamma}=\mathrm{supp}\gamma\cup\mathrm{supp}\delta
       	^{g}\subset I_{g}(x^{k},y^{k},u^{k})\cup I_{\varphi^{t_{k}}_{KS}}^{+0}(x^{k},y^{k},u^{k})\subset\theta\cup\nu,\label{K28}%
       	\end{equation}
       	while observing that for any $k$ sufficiently large,%
       	\begin{align*}
       	I_{u}(x^{k},y^{k},u^{k})  & \subset I_{\varphi^{t_{k}}_{KS}}^{00}(x^{k},y^{k}
       	,u^{k})\cup I_{\varphi^{t_{k}}_{KS}}^{0+}(x^{k},y^{k},u^{k})\subset\theta\cup
       	\eta,\\
       	I_{g}(x^{k},y^{k},u^{k})  & \subset I_{\varphi^{t_{k}}_{KS}}(x^{k},y^{k}
       	,u^{k})\cup I_{\varphi^{t_{k}}_{KS}}^{+0}(x^{k},y^{k},u^{k})\subset\theta\cup\nu.
       	\end{align*}
       	Consequently,
       	$  	\tilde{\gamma}_{\eta}=0$ and $\nabla_{y}g_{\nu}(\bar{x},\bar
       	{y})\beta=0$. On the other hand,
       	$\alpha_{j}\geq0$, $G_{j}(\bar{x})\leq0$, $\alpha_{j}G_{j}(\bar{x})=0$ for $j=1, \ldots, p$.	
       	Let us now prove that
       	$
       	\tilde{\gamma}_{i}(\nabla_{y}g_{i}(\bar{x},\bar{y}))\beta\geq0
       	$
       	for all $i\in\theta.$ Assume that $\tilde{\gamma}_{i}>0$ and $\nabla_{y}
       	g_{i}(\bar{x},\bar{y})\beta<0$ for some $i\in\theta$ then from (\ref{k20}),
       	$i\in\mathrm{supp}\delta^{g}.$ So that for all $k$ large enough,
       	$i\in\mathrm{supp}\delta^{g,k}\subset I_{\varphi^{t_{k}}_{KS}}^{+0}(x^{k},y^{k},u^{k})$
       	so $\mu_{i}^{k}=0$ and $\delta_{i}^{u,k}=0 $ from properties (\ref{k14}) and
       	(\ref{15b}) respectively. This leads to a contradiction since $\nabla_{y}
       	g_{i}(x^{k},y^{k})\beta_{i}^{k}=0 $ and the limit $\nabla_{y}g_{i}(\bar
       	{x},\bar{y})\beta<0$ by assumption. Similarly, if $\tilde{\gamma}_{i}<0$ with
       	$\nabla_{y}g_{i}(\bar{x},\bar{y})\beta>0$ that is from (\ref{k20}),
       	$i\in\mathrm{supp}\gamma.$ So that for all $k$ large enough, $i\in
       	\mathrm{supp}\gamma^{k}$ and thus $i\notin\mathrm{supp}\delta^{k}.$ This
       	implies that $\delta_{i}^{u,k}=0$ and so $\nabla_{y}g_{i}(x^{k},y^{k}%
       	)\dfrac{\beta^{k}}{\left\Vert \chi_{k}\right\Vert }=\dfrac{-\mu_{i}^{k}%
       	}{\left\Vert \chi_{k}\right\Vert }$ converges to $\nabla_{y}g_{i}(\bar{x}%
       	,\bar{y})\beta=-\mu_{i}\leq0$ which leads to a contradiction too.
       	Consequently, 
       	\[
       	\left\{
       	\begin{array}
       	[c]{l}%
       	\nabla_{y}\mathcal{L}(\bar{x},\bar{y},\bar{u})^{T}\beta+\nabla_{y}g^{T}%
       	(\bar{x},\bar{y})\tilde{\gamma}=0,\\
       	\nabla_{y}g_{v}(\bar{x},\bar{y})\beta=0,\text{ \ }\tilde{\gamma}_{\eta}=0,\\
       	\tilde{\gamma}_{i}(\nabla_{y}g_{i}(\bar{x},\bar{y}))\beta\geq
       	0\;\;\;\mbox{ for }\;i\in\theta,
       	\end{array}
       	\right.
       	\]
       	so that $(\beta,\tilde{\gamma})\in\Lambda_{y}^{ec}(\bar{x},\bar{y},\bar{u},0)$
       	and by condition \eqref{CQMM-2}, we get
       	$
       	\nabla_{x}\mathcal{L}(\bar{x},\bar{y},\bar{u})^{T}\beta+\nabla_{x}g^{T}
       	(\bar{x},\bar{y})\tilde{\gamma}=0
       	$
       	using (\ref{K22}), it yields $\nabla G(\bar{x})^{T}\alpha=0$; this implies
       	that $\alpha=0$ ($\bar{x}$ being upper-level regular). Thus,
       	\[
       	\left\{
       	\begin{array}
       	[c]{l}%
       	\nabla_{x,y}\mathcal{L}(\bar{x},\bar{y},\bar{u})^{T}\beta+\nabla g^{T}(\bar
       	{x},\bar{y})\tilde{\gamma}=0,\\
       	\nabla_{y}g_{v}(\bar{x},\bar{y})\beta=0,\text{ \ }\tilde{\gamma}_{\eta}=0,\\
       	\tilde{\gamma}_{i}(\nabla_{y}g_{i}(\bar{x},\bar{y}))\beta\geq
       	0\;\;\;\mbox{ for }\;i\in\theta,
       	\end{array}
       	\right.
       	\]
       	that is, $(\beta,\tilde{\gamma})\in\Lambda^{ec}(\bar{x},\bar{y},\bar{u},0) $
       	and by the condition \eqref{CQMM-1} we get $(\beta,\tilde{\gamma})=0$ and from
       	(\ref{K23b}) $\tilde{\mu}=0.$ Hence $(\alpha,\beta,\mu,\gamma)=0$ so that,
       	$(\delta^{u},\delta^{g})\neq0.$ Assume that there is $i$ such that $\delta
       	_{i}^{u}>0$ (resp. $\delta_{i}^{g}>0$). Thus $\tilde{\mu}_{i}=-\delta_{i}
       	^{u}<0$ (resp. $\tilde{\gamma}_{i}=\delta_{i}^{g}>0$) which contradicts the
       	fact that $\tilde{\mu}=0$ (resp. $\tilde{\gamma}=0$)$.$       	
       	Consequently, the sequence $(\alpha^{k},\beta^{k},\mu^{k},\gamma^{k}
       	,\delta^{u,k},\delta^{g,k})_{k}$ is bounded. Consider $(\bar{\alpha}
       	,\overline{\beta},\bar{\mu},\bar{\gamma},\bar{\delta}^{u},\bar{\delta}^{g})$
       	its limit (up to a subsequence) so that we have clearly
       	\[
       	\mathrm{supp}\bar{\alpha}\subset I_{G},\text{ }\mathrm{supp}\bar{\mu}
       	\subset\theta\cup\eta,\text{ }\mathrm{supp}\bar{\gamma}\subset\theta
       	\cup\upsilon.
       	\]
       	Taking the limit in (\ref{k15bis}), we get
       	\[
       	\left\{
       	\begin{array}
       	[c]{l}
       	\nabla_{x}F(\bar{x},\bar{y})+\nabla G(\bar{x})^{T}\bar{\alpha}+\nabla
       	_{x}\mathcal{L}(\bar{x},\bar{y},\bar{u})^{T}\overline{\beta}+\sum
       	\limits_{i=1}^{q}\lambda_{i}^{g}\nabla_{x}g_{i}(\bar{x},\bar{y})=0,	\\[1ex]%
       	\nabla_{y}F(\bar{x},\bar{y})+\nabla_{y}\mathcal{L}(\bar{x},\bar{y},\bar
       	{u})^{T}\overline{\beta}+\sum\limits_{i=1}^{q}\lambda_{i}^{g}\nabla_{y}
       	g_{i}(\bar{x},\bar{y})=0,\\[1ex]
       	\nabla_{y}g_{i}(\bar{x},\bar{y})\overline{\beta}+\lambda_{i}^{u}=0\text{
       		\ \ \ \ \ }\forall i=1,\ldots,q
       	\end{array}
       	\right.
       	\]
       	where we have put
       	\[
       	\lambda_{i}^{g}:=\left\{
       	\begin{array}
       	[c]{ll}%
       	-\bar{\gamma}_{i} & \mbox{ if }\text{ }i\in\mathrm{supp}\bar{\gamma
       	},\text{\medskip}\\
       	\bar{\delta}_{i}^{g} & \mbox{ if }\text{ }i\in\mathrm{supp}\bar{\delta}%
       	^{g},\text{\medskip}\\
       	0 & \mbox{otherwise}
       	\end{array}
       	\right.  \text{ \ and \ }\lambda_{i}^{u}:=\left\{
       	\begin{array}
       	[c]{ll}%
       	\bar{\mu}_{i} & \mbox{ if }\text{ }i\in\mathrm{supp}\bar{\mu},\text{\medskip}\\
       	-\bar{\delta}_{i}^{u} & \mbox{ if }\text{ }i\in\mathrm{supp}\bar{\delta}
       	^{u},\text{\medskip}\\
       	0 & \mbox{otherwise}.
       	\end{array}
       	\right.
       	\]
       	It is clear that $\bar{\alpha}_{j}\geq0,$ $G_{j}(\bar{x})\leq0$, 
       	$\bar{\alpha}_{j}G_{j}(\bar{x})=0$ for $i=1,..,p$. Moreover,  since
       	$\mathrm{supp}\bar{\mu}\cup\mathrm{supp}\bar{\delta}^{u}\subset\theta\cup\eta$
       	and $\mathrm{supp}\bar{\gamma}\cup\mathrm{supp}\bar{\delta}^{g}\subset
       	\theta\cup v$, 
       	\[
       	\nabla_{y}g_{v}(\bar{x},\bar{y})\overline{\beta}=0,\text{ \ }\lambda_{\eta
       	}^{g}=0.
       	\]
       	It remains to prove that $(\lambda_{i}^{g}<0\wedge\nabla_{y}g_{i}(\bar{x}
       	,\bar{y})\bar{\beta}<0)\vee\lambda_{i}^{g}\nabla_{y}g_{i}(\bar{x},\bar{y}
       	)\bar{\beta}=0$ for $i\in\theta.$ Let us assume that for some $i\in\theta$,
       	\[
       	\lambda_{i}^{g}>0\;\;\mbox{ or }\;\;\nabla_{y}g_{i}(\bar{x},\bar{y}
       	)\overline{\beta}>0\text{ with }\lambda_{i}^{g}\nabla_{y}g_{i}(\bar{x},\bar
       	{y})\bar{\beta}\neq0.
       	\]
       	Note that from the previous discussion, $\lambda_{i}^{g}\nabla_{y}g_{i}
       	(\bar{x},\bar{y})\bar{\beta}\geq0$ for all $i$ in $\theta.$ Assume now that
       	$\lambda_{i}^{g}>0$. Then $i\in\mathrm{supp}\bar{\delta}^{g}$ and as
        \[
        \mathrm{supp}\bar{\delta}^{u}\cap \textrm{supp} \bar{\delta}^{g}
       	=\emptyset, \quad i\notin\mathrm{supp}\bar{\delta}^{u}
        \]
        so for all $k$ large
       	enough, $i\notin\mathrm{supp}\delta^{u,k}$and from (\ref{k15bis}), $\nabla
       	_{y}g_{i}(x^{k},y^{k})\beta_{i}^{k}=-\mu_{i}^{k}$ which converges to
       	$-\bar{\mu}_{i}\leq0$, that is, $\nabla_{y}g_{i}(\bar{x},\bar{y})\bar{\beta
       	}=0.$ This leads to a contradiction ($\lambda_{i}^{g}\nabla_{y}g_{i}(\bar
       	{x},\bar{y})\beta\neq0$ by assumption). Similarly, if $\nabla_{y}g_{i}(\bar
       	{x},\bar{y})\bar{\beta}>0$ and from (\ref{k20}), $i\in\mathrm{supp}\bar
       	{\delta}^{u}.$ Then for all $k$ large enough $i\in\mathrm{supp}\delta^{u,k}$
       	and $\nabla_{y}g_{i}(x^{k},y^{k})\beta^{k}=\delta_{i}^{u,k}\rightarrow0$ so we
       	get $\nabla_{y}g_{i}(\bar{x},\bar{y})\bar{\beta}=0$ which leads to a
       	contradiction too.

\end{document}